\documentclass[11pt, reqno]{amsart}
\usepackage{a4wide}
\usepackage{wrapfig}
\usepackage{amsmath}
\usepackage{amsthm}
\usepackage{mathrsfs}
\usepackage{color}
\usepackage{xcolor}
\usepackage{graphicx}
\usepackage{amssymb}
\usepackage{dsfont}
\usepackage{url}
\usepackage{multicol}
\usepackage{tikz}
\usepackage{amsmath}
\usepackage{fancyhdr}
\usetikzlibrary{matrix}
\usetikzlibrary{arrows}
\usepackage{subfig}
\usepackage{hyperref}
\usepackage{listings}

\usepackage{esint}

\usepackage{enumitem}
\usepackage{varwidth}

\allowdisplaybreaks

\title[LWP for ZK on the background of a bounded function]{Local well-posedness for the Zakharov-Kuznetsov equation on the background of a bounded function}

\author[J.M. Palacios]{Jos\'e Manuel Palacios}
\address{Institut Denis Poisson, Universit\'e de Tours, Universit\'e d'Orleans, CNRS, Parc Grandmont 37200, Tours, France}
\email{jose.palacios@lmpt.univ-tours.fr}

\newcommand{\be}{\begin{equation}}
\newcommand{\ee}{\end{equation}}
\newcommand{\bp}{\begin{proof}}
\newcommand{\ep}{\end{proof}}
\newcommand{\bel}{\begin{equation}\label}
\newcommand{\eeq}{\end{equation}}
\newcommand{\bea}{\begin{eqnarray}}
\newcommand{\eea}{\end{eqnarray}}
\newcommand{\bee}{\begin{eqnarray*}}
\newcommand{\eee}{\end{eqnarray*}}
\newcommand{\ben}{\begin{enumerate}}
\newcommand{\een}{\end{enumerate}}

\newcommand{\R}{\mathbb{R}}

\newcommand{\N}{\mathbb{N}}
\newcommand{\Z}{\mathbb{Z}}

\newcommand{\supp}{\operatorname{supp}}

\newtheorem{thm}{Theorem}[section]
\newtheorem{cor}[thm]{Corollary}
\newtheorem{lem}[thm]{Lemma}
\newtheorem{prop}[thm]{Proposition}

\theoremstyle{remark}
\newtheorem{rem}{Remark}[section]

\definecolor{codegreen}{rgb}{0,0.6,0}
\definecolor{codegray}{rgb}{0.5,0.5,0.5}
\definecolor{codepurple}{rgb}{0.58,0,0.82}
\definecolor{backcolour}{rgb}{0.95,0.95,0.92}

\lstdefinestyle{mystyle}{
	backgroundcolor=\color{backcolour},   
	commentstyle=\color{codegreen},
	keywordstyle=\color{magenta},
	numberstyle=\tiny\color{codegray},
	stringstyle=\color{codepurple},
	basicstyle=\footnotesize,
	breakatwhitespace=false,         
	breaklines=true,                 
	captionpos=b,                    
	keepspaces=true,                 
	numbers=left,                    
	numbersep=5pt,                  
	showspaces=false,                
	showstringspaces=false,
	showtabs=false,                  
	tabsize=2
}
\numberwithin{equation}{section}

\usepackage{pgfplots}
\pgfplotsset{compat=newest}



\theoremstyle{definition}

\numberwithin{ej}{section}

\begin{document}





\renewcommand{\sectionmark}[1]{\markright{\thesection.\ #1}}
\renewcommand{\headrulewidth}{0.5pt}
\renewcommand{\footrulewidth}{0.5pt}
\begin{abstract}
We prove the local well-posedness for the two-dimensional Zakharov-Kuznetsov equation in $H^s(\R^2)$, for $s\in [1,2]$, on the background of an $L^\infty(\R^3)$-function $\Psi(t,x,y)$, with  $\Psi(t,x,y)$ satisfying some natural extra conditions. This result not only gives us a framework to solve the ZK equation around a Kink, for example, but also around a periodic solution, that is, to consider localized non-periodic perturbations of periodic solutions. Additionally, we show the global well-posedness in the energy space $H^1(\R^2)$.
\end{abstract}

\maketitle 

\section{Introduction}

\subsection{The classical model} 

In this work we seek to study the initial value problem associated with the Zakharov-Kuznetsov (ZK) equation in two space dimensions, namely \begin{align}\label{zk}
\begin{cases} 
\partial_tv+\partial_x\Delta v+\tfrac{1}{2}\partial_x(v^2)=0,
\\ v(0,x,y)=v_0(x,y)\in H^s(\R^2),
\end{cases}
\end{align}
where $v=v(t,x,y)$ is a real-valued function, $\Delta$ is the Laplacian operator and $(t,x,y)\in\R^3$. Equation \eqref{zk} was formally derived by Zakharov and Kuznetsov in \cite{ZK} as an asymptotic model to
describe the propagation of nonlinear ion-acoustic waves in a magnetized plasma. Equation \eqref{zk} has also been derived by Lannes, Linares and Saut in \cite{LaLiSa} from the Euler-Poisson system with magnetic field as a long-wave and small-amplitude limit (see also \cite{LiSa} for a formal derivation). Moreover, the Zakharov-Kuznetsov equation may also be seen as a natural two-dimensional generalization of the celebrated Korteweg-de Vries (KdV) equation \begin{align}\label{kdv}
\partial_tv+\partial_x^3v+\tfrac12\partial_x (v^2)=0.
\end{align}
One of the most interesting features of equation \eqref{zk} is the existence of solitary wave solutions. Indeed, the Zakharov-Kuznetsov equation has solitary waves solutions of many different types, namely, localized traveling waves solutions, kink solutions as well as periodic traveling waves solutions. Moreover, as a generalization of the KdV equation, a basic method to obtain solitary waves solutions for the ZK equation is simply by taking the solutions of the KdV equation and considering them as $y$-independent functions.

\medskip

Contrary to the Korteweg-de Vries equation, or the Kadomtsev-Petviashvili equation \[
\partial_tv+\partial_x^3v\pm\partial_x^{-1}\partial_y^2v+\tfrac12\partial_x(v^2)=0,
\]
which is another higher-dimensional generalization of the KdV equation \eqref{kdv}, the Zakharov-Kuznetsov equation \eqref{zk} is not completely integrable. Nevertheless, it keeps a Hamiltonian structure and possesses at least three different quantities that are formally conserved by the ZK flow, namely, the mean, the mass and the energy (respectively) \begin{align*}
I_1\left(v(t)\right)&:= \int_{\R^2} v(t,x,y)dxdy=I_1(v_0),
\\ I_2\left(v(t)\right)&:= \int_{\R^2}v^2(t,x,y)dxdy=I_2(v_0), 
\\ I_3\left(v(t)\right)&:= \dfrac12\int_{\R^2}\Big(\big\vert\nabla v(t,x,y)\big\vert^2-\dfrac13 v(t,x,y)^3\Big)=I_3(v_0).
\end{align*}
Therefore, $L^2$ and $H^1$ are two natural spaces to study well-posedness for the ZK equation. As a consequence of these conservation laws, global well-posedness has been proved in $L^2(\R^2)$ and $H^1(\R^2)$ in \cite{Ki,Fa} respectively.

\subsection{Model in the background of a bounded function} Motivated by the study of Kink solutions, in this work we seek to study a slight generalization of the initial value problem \eqref{zk}. More specifically, here we consider the problem \begin{align}\label{zk_initial_data}
\begin{cases} 
\partial_tv+\partial_x\Delta v+\tfrac{1}{2}\partial_x(v^2)=0,
\\ v(0,x,y)=\Phi(x,y),
\end{cases}
\end{align}
where, for the moment, we do not intend to assume any decay property of the initial data $\Phi(x,y)$ in any spacial direction, but only that $\Phi\in L^\infty(\R^2)$. Instead, we decompose the solution $v(t,x,y)$ in the following fashion \begin{align}\label{decomposition}
v(t,x,y)=u(t,x,y)+\Psi(t,x,y),
\end{align}
where we assume that $\Psi\in L^\infty(\R^3;\R)$ is a given function (see \eqref{hyp_psi_general} below for the specific hypotheses on $\Psi$) and we seek for $u(t)$ belonging to some Sobolev space. Then, it is natural to rewrite the above initial value problem in terms of the following Cauchy problem \begin{align}\label{zk_psi}
\begin{cases} 
\partial_tu+\partial_t\Psi+\partial_x \Delta u+\partial_x\Delta\Psi+\tfrac{1}{2}\partial_x(u+\Psi)^2=0,
\\ u(0,x,y)=u_0(x,y)\in H^s(\R^2).
\end{cases}
\end{align}
We stress that equation \eqref{zk_psi} is nothing but equation \eqref{zk_initial_data} once replacing the decomposition
given in \eqref{decomposition}.

\medskip

At this point it is worth to mention that, due to the presence of the background function $\Psi$, which is not integrable, none of the conservation laws presented in the previous subsection is well-defined, and there seems to be no evident well-defined conservation law for \eqref{zk_psi}. Although, a suitable modification of the energy functional $I_3$ shall play a key role in proving global well-posedness for \eqref{zk_psi} in $H^1(\R^2)$.

\medskip

As mentioned before, one of our main motivations comes from studying Kink solutions. However, at the same time we also seek to provide a framework to study localized non-periodic
perturbations of periodic solutions of \eqref{zk}, such as for example, the
famous cnoidal wave solutions of the KdV equaiton \begin{align*}
\Psi_{\mathrm{cn}}(t,x,y):=\alpha+\beta \mathrm{cn}^2\big(\gamma(x-ct);\kappa\big),
\end{align*}
where $c>0$ and $(\alpha,\beta,\gamma,\kappa)\in\R^4$ satisfying some suitable conditions, and where $\mathrm{cn}(\cdot,\cdot)$ stands for the Jacobi elliptic cnoidal function (see \cite{KdV}).

\medskip

\textbf{Hypotheses on the background function:} In the sequel we shall always assume that the given function $\Psi(t,x,y)$ satisfies the following hypotheses:
\begin{align}\label{hyp_psi_general}
\begin{cases}
\Psi \in L^\infty(\R,W^{4^+,\infty}_{xy}(\R^2)),
\\ \big(\partial_t\Psi+\partial_x\Delta\Psi+\tfrac12\partial_x( \Psi^2)\big)\in L^\infty(\R,H^{3^+}(\R^2)).
\end{cases}
\end{align}
\begin{rem}
Notice that if $\Psi$ solves equation \eqref{zk}, then the latter hypothesis in \eqref{hyp_psi_general} is immediately satisfied. This observation shall allow us to work in the background of kinks and periodic solutions.
\end{rem}

\smallskip

About proving local well-posedness for model  \eqref{zk_psi}, it is interesting to notice that, in \cite{LiPa}, Linares and Pastor showed the local well-posedness of the Zakharov-Kuznetsov equation \eqref{zk} for initial data in $H^s(\R^2)$, with $s>3/4$. With this aim, they adapted the method developed by Kenig, Ponce and Vega in \cite{KePoVe2} (to deal with the generalized KdV equation), which combines smoothing effects, Strichartz estimates, and a maximal function estimate, along with the Banach contraction principle. Thus, since equation \eqref{zk_psi} can be regarded as a perturbation of the classical ZK equation \eqref{zk}, one might think that, in order to prove local well-posedness for equation \eqref{zk_psi}, it is reasonable to proceed by using the contraction principle just as Linares and Pastor did in \cite{LiPa}. However, it seems that this is not possible, due to the occurrence of the term $\Psi\partial_xu$, since $\Psi$ is not integrable, which shows us that this problem is more involved.

\smallskip

\subsection{Main results}
In the rest of this work, we focus on studying the Cauchy problem associated with \eqref{zk_psi}. Our main goal is to prove global well-posedness for \eqref{zk_psi} in the energy space. The following theorem is our main result and provide us the local well-posedness in Sobolev spaces for $s\in [1,2]$. 

\begin{thm}[Local well-posedness]\label{MT1}
Let $s\in [1,2]$ fixed. Consider a fixed background function $\Psi(t,x,y)$ satisfying the conditions in \eqref{hyp_psi_general}. Then, for any initial data $u_0\in H^s(\R^2)$ there exists a positive time of existence $T=T(\Vert u_0\Vert_{H^s})>0$ and a unique solution to the equation \eqref{zk_psi}, emanating from $u_0$, such that \[
u\in C([-T,T],H^s(\R^s))\cap B^{s}(T)\cap F^{s}(T).
\]
Moreover, the data-to-solution map $\Phi:u_0\to u$ is continuous from $H^s(\R^2)$ into $C([-T,T],H^s(\R^2))$.
\end{thm}

\begin{rem}
The short-time Bourgain spaces $F^{s}(T):=F^{s}_{1/2}(T)$ and $B^{s}(T):=B^{s}_{1/2}(T)$ shall be defined in Section \ref{preliminaries}.
\end{rem}

Finally, by using some modification of the energy functional we prove global well-posedness in the energy space for equation \eqref{zk_psi}.

\begin{thm}[Global well-posedness]\label{MT_gwp}
The local solution $u(t)$ provided by the previous theorem can be extended for all times $T>0$.
\end{thm}

\begin{rem}
The local well-posedness Theorem \ref{MT1} only requires the regularity hypotheses in \eqref{hyp_psi_general} to hold with exponents $7/2^+$ and $5/2^+$, respectively. The fact that we assume regularity $4^+$ and $3^+$, respectively, is just to have a wider range of LWP in the regular case (c.f. Theorem \ref{regular_lwp_thm}), so that we can easily justify all our computations while integrating by parts, using the continuity of the data-to-solution map, hence assuming that the solution is regular enough.
\end{rem}

We now discuss the main ingredients in the proof of  Theorem \ref{MT1}. We shall adapt the method introduced by Ionescu, Kenig and Tataru in \cite{IoKeTa}, in the context of the KP-I equation, which consists in an energy method, based on the introduction of the dyadic short-time Bourgain spaces $F^{s}_\beta$ and their dual $N^{s}_\beta$, defined in the following section. Generally speaking, from a perturbative point of view, for an equation of the form \begin{align*}
\begin{cases}
\partial_t u+\partial_x\Delta u=f,
\\ u(0,x,y)=u_0(x,y),
\end{cases}
\end{align*}
one would like to prove a linear estimate for solutions of the above IVP, of the form \begin{align}\label{expl_linear}
\Vert u\Vert_{\mathbf{F}^{s}(T)}\lesssim \Vert u_0\Vert_{H^s}+\Vert f\Vert_{\mathbf{N}^{s}(T)},
\end{align}
for some suitable spaces $\mathbf{F}^{s}(T)$ and $\mathbf{N}^{s}(T)$, along with matching nonlinear estimates \[
\Vert \partial_x(u^2)\Vert_{\mathbf{N}^{s}(T)}\lesssim \Vert u\Vert_{\mathbf{F}^{s}(T)}^2 \quad \hbox{ and } \quad \Vert \partial_x(u\Psi)\Vert_{\mathbf{N}^{s}(T)}\lesssim \Vert u\Vert_{\mathbf{F}^{s}(T)}\Vert \Psi\Vert_{ L^\infty_t W^{4^+,\infty}_{x,y}}.
\]
However, due to the presence of $\Psi$, it is not clear whether such a choice of spaces $\mathbf{F}^{s}(T)$ and $\mathbf{N}^{s}(T)$ exist, which forces us to approach the problem in a less perturbative way. 

\medskip

Instead, the key idea of Ionescu, Kenig and Tataru in \cite{IoKeTa}, was to introduce some normed functional spaces $\mathbf{F}^{s}(T)$, $\mathbf{N}^{s}(T)$, and a semi-normed space $\mathbf{B}^{s}(T)$ so that, for smooth solutions of the equation, the following inequalities (reformulated in our present context, adding $\Psi$ into the equation) hold \begin{align}\label{explicacion}
\begin{cases}
\Vert u\Vert_{\mathbf{F}^{s}(T)}\lesssim \Vert u\Vert_{\mathbf{B}^{s}(T)}+\Vert \partial_x(u^2)\Vert_{\mathbf{N}^{s}(T)}+\Vert \partial_x(u\Psi)\Vert_{\mathbf{N}^{s}(T)}+|||\Psi|||_{s},
\\ \Vert \partial_x(u^2)\Vert_{\mathbf{N}^{s}(T)}\lesssim \Vert u\Vert_{\mathbf{F}^{s}(T)}^2,
\\ \Vert \partial_x(u\Psi)\Vert_{\mathbf{N}^{s}(T)}\lesssim \Vert u\Vert_{\mathbf{F}^{s}(T)}\Vert \Psi\Vert_{ L^\infty_t W^{4^+,\infty}_{xy}}
\\ \Vert u\Vert_{\mathbf{B}^{s}(T)}^2\lesssim \Vert u_0\Vert_{H^s}+\Vert u\Vert_{\mathbf{F}^{s}(T)}|||\Psi|||_{3^+}+\Vert u\Vert_{\mathbf{F}^{s}(T)}^2\Vert \Psi\Vert_{ L^\infty_tW^{4^+,\infty}_{xy}}+\Vert u\Vert_{\mathbf{F}^{s}(T)}^3,
\end{cases}
\end{align}
where the norm $|||\cdot|||_s$ is defined in  \eqref{norm_psi_3b}.
Then, the above inequality along with a simple continuity argument suffices to control $\Vert u\Vert_{\mathbf{F}^{s}(T)}$, provided that $\Vert u_0\Vert_{H^s}\ll 1$ and \[
\Vert \Psi\Vert_{L^\infty_tW^{4^+}_{xy}}+\Vert \partial_t\Psi+\partial_x\Delta\Psi+\tfrac12\partial_x (\Psi^2)\Vert_{L^\infty_tH^{3^+}_{xy}}\ll 1, 
\]
which can be arranged by scaling considerations. The first inequality in \eqref{explicacion} is the analogue of the linear estimate \eqref{expl_linear}, while the second and the third one correspond to the bilinear estimates. Finally, the last inequality in \eqref{explicacion} is an energy-type estimate. The main issue in \eqref{explicacion} is the appearance of the energy norm $\Vert u\Vert_{\mathbf{B}^{s}(T)}$ in the first inequality of \eqref{explicacion}, taking the place of the usual $H^s$-norm of the initial data $\Vert u_0\Vert_{H^s}$, which is introduced to control the small time localization appearing in the $\mathbf{F}^{s}$-structure. 

\medskip

The $\mathbf{F}^{s}(T)$ spaces enjoy a $X^{s,1/2,1}$-type structure, but with a localization in small, frequency dependent time intervals, whose length is of order $H^{-\beta}$, for some $\beta>0$ to be chosen (in our present case we take $\beta=1/2$), when the spatial frequency $(\xi,\mu)$ of the function is localized around $\xi^2+\mu^2\sim H$. This shall allow us, in a certain way, to work only with modulations of size $\vert \tau-\omega(\xi,\mu)\vert\gtrsim H^\beta$, thus neglecting the  contribution in the cases where the modulation variable is too small. Of course, the choice of $\beta$ shall be one of the key points in these definitions. As mentioned before, one of the main obstruction in proving local well-posedness, by a fixed point argument, in our present case is coming from the occurrence of $\partial_x(u\Psi)$ in the nonlinearity. More precisely, we choose $\beta$ in order to be able to deal with the norm $\Vert \partial_x(u\Psi)\Vert_{\mathbf{N}^{s}(T)}$, whose definition is also based on frequency dependent time intervals of length of order $H^{-\beta}$. Particularly harmful shall be terms of the form $\partial_xP_H(P_{\sim H}u\cdot P_{\ll H}\Psi)$. These types of terms lead us to consider $\beta=1/2$.   
Finally, we adapt the Bona-Smith method \cite{BS} to prove the continuity of the flow in the $H^s$ space.

\subsection{Previous results}

The Cauchy problem associated with the Zakharov-Kuznetsov equation \eqref{zk} has been extensively studied in the last years. In the two-dimensional case, the first result goes back to Faminskii \cite{Fa}, where he proved the local and global well-posedness in the energy space $H^1(\R^2)$ by adapting the prove of Kenig, Ponce and Vega to deal with the KdV equation \cite{KePoVe3}. This result was later improved by Linares and Pastor in \cite{LiPa}. They showed the LWP of \eqref{zk} in $H^s(\R^2)$ for $s>3/4$. This latter result was proved by using a fixed point argument, taking advantage of the dispersive smoothing effects associated to the linear part of the ZK equation \eqref{zk}, in a similar fashion as Kenig, Ponce and Vega did for the generalized KdV equation  \cite{KePoVe2}. Then, Gr\"unrock and Herr \cite{GrHe} and Molinet and Pilod \cite{MoPi} independently proved local well-posedness for $s>1/2$. This was done by using the Fourier restriction norm method. On the other hand, Shinya Kinoshita has recently proved the local and global well-posedness of equation \eqref{zk} in $H^s(\R^2)$ for $s>-1/4$ and $L^2(\R^2)$ respectively \cite{Ki}. This latter LWP result is almost optimal, at least from a Picard iteration approach point of view, since the data-to-solution map $u_0\mapsto u(t)$ fails to be $C^2$ for $s<-1/4$.

\medskip

Concerning the method of proof of Theorem \ref{MT1}, as we mentioned before, the technique used here was first introduced by Ionescu, Kenig and Tataru in \cite{IoKeTa}. Since then, these ideas have been adapted to deal with several other models. We refer to \cite{MoPi} and \cite{RiVe} for previous works using some similar ideas in the context of the ZK-type equation. Also, see \cite{Gu} and \cite{KePi} for some adaptations of these ideas in the context of KdV-type equations. Finally, we refer to \cite{ChCo} and \cite{KoTa} for some previous works using some similar spaces to prove a priori bounds for the 1D cubic NLS at low regularity.

\medskip

Finally, we believe that the technique employed here may be useful to prove LWP in the background of a bounded function for a broad class of  higher dimensional models, as for the KP-II equation, for example. We plan to address this issue in a forthcoming paper. For a method to treat one-dimensional models we refer to \cite{Pa}.

\subsection{Organization of this paper}

This paper is organized as follows. In Section
\ref{preliminaries} we provide the definition of the short-time Bourgain spaces and prove several of their basic properties. In Section \ref{sec_bilinear_estimates} we prove the main $L^2$-bilinear estimates. Then, we establish the main energy estimates for solutions and difference of solutions in Section \ref{sec_energy}. Finally, we prove the local and global well-posedness theorems in sections \ref{sec_loc} and \ref{sec_gwp} respectively.

\section{Preliminaries}\label{preliminaries}

\subsection{Basic notations}

For any pair of positive numbers $a$ and $b$, the notation $a\lesssim b$ means that there exists a positive constant $c$ such that $a\leq cb$. We also denote by $a\sim b$ when $a\lesssim b$ and $b\lesssim a$. Moreover, for $2\in\R$, we denote by $2^+$, respectively $2^-$,  a number slightly greater,
respectively lesser, than $2$. In the sequel we denote by $\mathbb{D}:=\{2^\ell:\, \ell\in\N\}$. Also, in this article we shall use an adapted (to dyadic numbers) version of the floor function $\lfloor\cdot\rfloor$ in the following sense: for $x\in\R$ we define the floor function $\lfloor x\rfloor$ as follows \[
\lfloor x\rfloor := \max\{L\in\mathbb{D}:\ L\leq x <2L\}.
\] 
For $a_1,a_2,a_3\in\R$ we define the quantities $a_\mathrm{max}\geq a_\mathrm{med} \geq a_\mathrm{min}$ to be the maximum, median and minimum of $a_1$, $a_2$ and $a_3$ respectively. Furthermore, we shall occasionally use the notation $F(x)$ to denote a primitive of the nonlinearity $f(x)$, that is,  $F(s)=\int_0^s f(s')ds'$.

\medskip

Now, for $u(t,x,y)\in \mathcal{S}'(\R^3)$, $\mathcal{F}u=\hat{u}$ shall denote its space-time Fourier transform, whereas $\mathcal{F}_tu$, respectively $\mathcal{F}_{x,y}u$, shall denote its time Fourier transform,
respectively space Fourier transform. 
Additionally, for $s\in\R$, we introduce the Bessel and Riesz potentials of order $-s$, namely $J^s_x$ and $D^s_x$, by (respectively) \[
J^su := \mathcal{F}^{-1}_{xy}\big((1+\vert(\xi,\mu)\vert^2)^{s/2}\mathcal{F}u\big) \quad \hbox{ and }\quad D^su:= \mathcal{F}^{-1}_{xy}\big(\vert(\xi,\mu)\vert^s\mathcal{F}u\big).
\]
We also denote by $U(t)$ the unitary group associated with the linear part of \eqref{zk}, that is, the  unitary group $e^{-t\partial_x\Delta}$ associated with the linear dispersive equation \[
\partial_t u+\partial_x\Delta u=0,
\]
which is defined via the Fourier transform as \[
U(t)u_0=\mathcal{F}^{-1}\big(e^{it\xi(\xi^2+\mu^2)}\mathcal{F}u_0\big).
\]
Moreover, from now on $\omega(\xi,\mu)$ shall denote the symbol associated with \eqref{zk}, that is, we set $\omega(\xi,\mu):=\xi(\vert\xi\vert^2+\mu^2)$ and the resonant relation
\begin{align*}
\Omega(\xi_1,\mu_1,\xi_2,\mu_2)&:=\omega(\xi_1+\xi_2,\mu_1+\mu_2)-\omega(\xi_1,\mu_1)-\omega(\xi_2,\mu_2).
\end{align*}
We shall also need to study the group velocity of the equation, in particular its first component, that is, the function \[
h(\xi,\mu):=\partial_\xi\omega(\xi,\mu)=3\xi^2+\mu^2.
\]
Throughout this work we consider a fixed smooth cutoff function $\eta$ satisfying \begin{align}\label{def_eta}
\eta\in C^\infty_0(\R),\quad 0\leq \eta\leq 1, \quad \eta\big\vert_{[-1,1]}=1 \quad \hbox{and} \quad \supp\eta\subset[-2,2].
\end{align}
We define $\phi(\xi):=\eta(\xi)-\eta(2\xi)$ and, for $\ell\in\N$, we denote by $\eta_{2^\ell}$ the function given by \[
\eta_{2^\ell}(\xi):=\phi\big(2^{-\ell}\xi\big) \quad \hbox{ and } \quad \eta_{0}(\xi):=\eta(\xi).
\]
Similarly, we define $\psi_{2^\ell}(\xi,\mu)$ and $\psi_{0}(\xi,\mu)$ as follows \[
\psi_{2^\ell}(\xi,\mu):=\eta_{2^\ell}(\xi^2+\mu^2) \quad \hbox{ and } \quad \psi_{0}(\xi,\mu):=\phi_{0}(\xi^2+\mu^2).
\]
Additionally, write a tilde above these functions to denote another cutoff function, in the same variables, but associated with a slightly larger region. Concretely, we define $(\tilde\eta_{2^\ell})_{\ell\geq 0}$ to be another nonhomogeneous dyadic partition of the unity satisfying that $\supp\tilde\eta_{2^\ell}\subset [2^{\ell-1},2^{\ell+1}] ,$ $\ell\geq1$, and such that $\tilde\eta_{2^\ell}\equiv 1$ on $\supp\eta_{2^\ell}$. We proceed in the same fashion for $\psi_{2^\ell}$.

\medskip

Any summations over capitalized variables such as $N$, $M$, $H$ or $L$ are presumed to be dyadic. Unless stated otherwise, we work with non-homogeneous dyadic decompositions for the space-frequency, time-frequency and modulation variables, i.e., these variables range over numbers of the form $\mathbb{D}=\{2^\ell:\, \ell\in\N\}$. Then, with the previous notations we have that \[
\sum_{N\in \mathbb{D}}\eta_{N}(\xi)=1 \, \hbox{ for all }\, \xi\in\R\setminus\{0\} \quad \hbox{ and } \quad \supp\eta_{N}\subset\{\tfrac{1}{2}N\leq \vert\xi\vert\leq 2N\}.
\]
We define the Littlewood-Paley multipliers by the following identities \begin{align}\label{def_pr_P_Q}
P_N^xu:=\mathcal{F}_{x}^{-1}\big(\eta_N(\xi)\mathcal{F}_xu\big)\quad \hbox{ and }\quad P_Hu:=\mathcal{F}^{-1}_{xy}\big(\psi_H(\xi,\mu)\mathcal{F}_{xy}u\big).
\end{align}
With these definitions at hand, we introduce the operators \begin{align*}
P_{\leq H}&:=\sum_{\mathcal{H}\leq H}P_\mathcal{H}, \quad P_{\geq H}:=\sum_{\mathcal{H}\geq H}P_\mathcal{H}, \quad P_{\sim H}:=\mathrm{Id}-P_{\ll H}-P_{\gg H},
\end{align*}
and similarly for $P_N^x$. Occasionally, for $A\subset\R^2$ we shall denote by $P_A=\mathcal{F}^{-1}_{xy}\mathds{1}_A\mathcal{F}_{xy}$, that is, the Fourier projection on $A$. We also define the set associated with the above decompositions. More precisely, for $N,H\in\mathbb{D}$ we define \[
I_N:=\big\{\xi\in\R: \ \tfrac12N\leq \vert \xi\vert 2N\big\} \quad \hbox{ and } \quad \Delta_H:=\big\{(\xi,\mu)\in\R^2: \ h(\xi,\mu)\in I_H\big\}.
\]
We borrow some notations from \cite{Ta} as well. For $k\geq 2$ natural number and $\xi\in\R$, we denote by $\Gamma^k(\xi,\mu)$ the $(2k-2)$-dimensional affine-hyperplane of $\R^{2k}$ given by \begin{align*}
\Gamma^k(\xi,\mu)&:=\{(\xi_1,\mu_1,...,\xi_k,\mu_k)\in\R^{2k}:\, \xi_1+...+\xi_k=\xi, \ \, \mu_1+...+\mu_k=\mu\}.
\end{align*}
endowed with the natural measure 
\begin{align*}
&\int_{\Gamma^k(\xi,\mu)}F(\xi_1,\mu_1,...,\xi_k,\mu_k)d\Gamma^k(\xi,\mu)
\\ & = \int_{\R^{2k-2}}F(\xi_1,\mu_1,...,\xi_{k-1},\mu_{k-1},\xi-\xi_1-...-\xi_{k-1},\mu-\mu_1-...-\mu_{k-1})d\xi_1d\mu_1...d\xi_{k-1}d\mu_{k-1}.
\end{align*}
for any function 
$F:\Gamma^k(\xi,\mu)\to \mathbb{C}$. Moreover, when $\xi=\mu=0$ we shall simply denote by $\Gamma^k=\Gamma^k(0)$ with the obvious modifications.

\subsection{Function spaces}

We shall work with short-time localized Bourgain spaces, introduced by Ionescu, Kenig and Tataru in \cite{IoKeTa}. First, for $H\in\mathbb{D}$ fixed, we introduce the frequency localized $\ell^1$-Besov type space $X_H$ of regularity $1/2$ with respect to modulations
\[
X_H:=\big\{\phi\in L^2(\R^3): \, \mathrm{supp}\,\phi\subset \R\times I_H \,\hbox{ and }\, \Vert \phi\Vert_{X_H}<+\infty\big\},
\]
where the $X_H$-norm is defined as \begin{align}\label{xh_norm}
\Vert \phi\Vert_{X_H}:=\sum_{L\in\mathbb{D}}L^{1/2}\big\Vert \eta_L\big(\tau-\omega(\xi,\mu)\big)\phi(\tau,\xi,\mu)\big\Vert_{L^2(\R^3)}.
\end{align}
This type of structures were previously used in \cite{Tataru}, and are useful to prevent high frequency losses in bilinear and trilinear estimates. Now, for $\beta\geq 0$ to be chosen and $H\in\mathbb{D}$ fixed, we define the Bourgain spaces $F_{H,\beta}$ localized in short-time intervals of length $H^{-\beta}$, \[
F_{H,\beta}:=\big\{f\in L^\infty(\R,L^2(\R^2)):\, \mathrm{supp}\,\mathcal{F}(f)\subset\R\times I_H\, \hbox{ and } \, \Vert f\Vert_{F_{H,\beta}}<+\infty\big\},
\]
where the $F_{H,\beta}$-norm is given by \[
\Vert f\Vert_{F_{H,\beta}}:=\sup_{t_H\in\R}\big\Vert \mathcal{F}\big(\eta_0(H^\beta(\cdot-t_H))f\big)\big\Vert_{X_H}.
\]
Its dual version $N_{H,\beta}$ is defined by \[
N_{H,\beta}:=\big\{f\in L^\infty(\R,L^2(\R^2)):\, \mathrm{supp}\,\mathcal{F}(f)\subset\R\times I_H\, \hbox{ and } \, \Vert f\Vert_{N_{H,\beta}}<+\infty\big\},
\]
where in this case the $N_{H,\beta}$-norm is given by \[
\Vert f\Vert_{N_{H,\beta}}:=\sup_{t_H\in\R}\big\Vert \big(\tau-\omega(\xi,\omega)+iH^\beta\big)^{-1}\mathcal{F}\big(\eta_0(H^\beta (\cdot-t_H))f\big)\big\Vert_{X_H}.
\]
Then, for $s\geq0$, we define the global $F^s_\beta$ and $N^s_\beta$ spaces from their frequency localized versions $F_{H,\beta}$ and $N_{H,\beta}$ by using a nonhomogeneous Littlewood-Paley decomposition as follows \begin{align}\label{Fs_norm}
F^s_\beta&:=\Big\{f\in L^\infty(\R, L^2(\R^2)):\, \Vert f\Vert_{F^s_\beta}^2:=\sum_{H\in\mathbb{D}} H^{s}\Vert P_Hf\Vert_{F_{H,\beta}}^2<+\infty \Big\},
\\ N^s_\beta &:=\Big\{f\in L^\infty(\R, L^2(\R^2)):\, \Vert f\Vert_{N^s_\beta}^2:=\sum_{H\in\mathbb{D}} H^{s}\Vert P_Hf\Vert_{N_{H,\beta}}^2<+\infty\Big\}.\label{Ns_norm}
\end{align}
The reason why we define the above norms with an $H^s$ weight instead of $H^{2s}$ comes from the fact that the localization already has a square power $H\sim \xi^2+\mu^2$. Hence, defining them with a $H^s$ weight makes them related to an $H^s(\R^2)$ level of regularity.

\medskip

The bounds we obtain for solutions of equation \eqref{zk_psi} are on a fixed
time interval, while the above function spaces are not. To remedy this, we also define the localized (in time) version of these spaces. For any time $T\in(0,1]$, let $Y$ denote either $F^s_\beta$ or $N^s_\beta$. Then, we define \begin{align*}
\Vert f\Vert_{Y(T)}:=\mathrm{inf}\big\{\Vert \tilde{f}\Vert_Y:\, \tilde{f}:\R^3\to\R \, \hbox{ and } \, \tilde{f}\big\vert_{[-T,T]\times \R^2}=f \big\}.
\end{align*}
Then, \begin{align}\label{FsT_norm}
F^s_\beta(T)&:=\big\{f\in L^\infty([-T,T],L^2(\R^2)):\, \Vert f\Vert_{F ^s_\beta(T)}<+\infty\big\},
\\ N^s_\beta(T)&:=\big\{f\in L^\infty([-T,T],L^2(\R^2)):\, \Vert f\Vert_{N ^s_\beta(T)}<+\infty\big\}.\label{NsT_norm}
\end{align}
Finally, we define the energy space $B^s(T)$ as follows, for $s\geq 0$ and $T\in (0,1]$ we set \[
\Vert f\Vert_{B^s(T)}^2:=\Vert P_1f(0,\cdot,\cdot)\Vert_{L^2_{x,y}}^2+\sum_{H\in\mathbb{D}\setminus\{1\}}H^{s}\sup_{t_H\in[-T,T]}\Vert P_Hf(t_H,\cdot,\cdot)\Vert_{L^2_{x,y}}^2.
\]
Then, the energy space $B^s(T)$ is given by \[
B^s(T):=\big\{f\in L^\infty([-T,T],L^2(\R^2):\, \Vert f\Vert_{B^s(T)}<+\infty\big\}.
\]

\subsection{Properties of the function spaces}

In this subsection we show
some basic but important properties concerning the short-time function spaces introduced in the previous subsection. They all have been proved in different contexts (see \cite{IoKeTa,KePi}).

\medskip

The following lemma gives us that $F^s_\beta(T)\hookrightarrow L^\infty([-T,T],H^s(\R^2))$.
\begin{lem}\label{lem_linear_estimate}
Let $T>0$, $s\geq 0$ and $\beta \geq 0$. Then, for all $f\in F^s_\beta(T)$, the following holds \begin{align}\label{basic_embedding}
\Vert f\Vert_{L^\infty_TH^s}\lesssim \Vert f\Vert_{F^s_\beta(T)}.
\end{align}
\end{lem}

\begin{proof}
In fact, let us consider $f\in F^s_\beta(T)$. Now, we choose an extension $\tilde{f}\in F^s_\beta$ such that \[
\tilde{f}\big\vert_{[-T,T]}=f \quad \hbox{ and } \quad \Vert \tilde{f}\Vert_{F^s_\beta}\leq 2\Vert f\Vert_{F^s_\beta(T)}.
\]
Then, it follows from the definition that \begin{align}\label{LP_decomposition_proof_lemma}
\Vert f(t,\cdot,\cdot)\Vert_{H^s}^2=\Vert \tilde{f}(t,\cdot,\cdot)\Vert_{H^s}^2\lesssim \sum_{H\in \mathbb{D}}H^{s}\Vert P_H\tilde{f}(t,\cdot,\cdot)\Vert_{L^2_{x,y}}^2.
\end{align}
Now, for $t\in[-T,T]$ and $H\in\mathbb{D}$ fixed, using the Fourier Inversion formula we have that \begin{align}\label{fourier_inversion}
\mathcal{F}_{x,y}(P_H\tilde{f})(t,\xi,\mu)=c\int_{\R^2}e^{it\tau} \mathcal{F}_{x,y}\left(\eta_0\big(H^\beta(\cdot-t)\big)P_H\tilde{f}\right) d\tau.
\end{align}
On the other hand, by using Cauchy-Schwarz inequality in $\tau$ along with the definition of $X_H$ we infer that \begin{align}\label{l1l2_xh}
\left\Vert \int_{\R}\vert f(\tau,\xi,\mu)\vert d\tau\right\Vert_{L^2_{\xi,\mu}}\lesssim \Vert f \Vert_{X_H},
\end{align}
for all $f\in X_H$. Therefore, gathering \eqref{fourier_inversion} and \eqref{l1l2_xh} we obtain that \[
\Vert P_H\tilde{f}(t,\cdot,\cdot)\Vert_{L^2_{x,y}}\lesssim \Vert P_H\tilde{f}\Vert_{F_{H,\beta}},
\]
for all $H\in\mathbb{D}$. Then, plugging the latter estimate into \eqref{LP_decomposition_proof_lemma} and then taking the supreme over $t\in[-T,T]$, we conclude the proof of the lemma.
\end{proof}

\begin{lem}\label{gro_characteristic_function_lemma}
Let $I\subset\R$ be a bounded interval and consider $H\in\mathbb{D}$ fixed. Then, the following inequality holds \begin{align}\label{property_space_ineq_sup}
\sup_{L\in\mathbb{D}}L^{1/2}\left\Vert \eta_L\big(\tau-\omega(\xi,\mu)\big)\mathcal{F}(\mathds{1}_I(t)f)\big)\right\Vert_{L^2}\lesssim \Vert \mathcal{F}(f)\Vert_{X_H},
\end{align}
for all $f$ such that $\mathcal{F}(f)\in X_H$.
\end{lem}

\begin{proof}
Let $L\in\mathbb{D}$ be fixed. First, by frequency decomposition we write \[
f=\sum_{\mathcal{L}\in\mathbb{D}}\mathcal{F}^{-1}\big(\eta_\mathcal{L}(\tau-\omega(\xi,\mu))\mathcal{F}(f)\big)=:\sum_{\mathcal{L}\in\mathbb{D}}f_\mathcal{L}.
\]
Then, plugging the above decomposition into the left-hand side of \eqref{property_space_ineq_sup} we obtain that \[
L^{1/2}\left\Vert \eta_L\big(\tau-\omega(\xi,\mu)\big)\mathcal{F}(\mathds{1}_I(t)f)\right\Vert_{L^2}\lesssim L^{1/2}\sum_{\mathcal{L}\in\mathbb{D}}\left\Vert \eta_L\big(\tau-\omega(\xi,\mu)\big)\mathcal{F}(\mathds{1}_I(t)f_\mathcal{L})\right\Vert_{L^2}.
\]
Now we split the sum over $\mathcal{L}\in\mathbb{D}$ into two sums, one over $\{\mathcal{L}\in\mathbb{D}: \, \mathcal{L}\gtrsim L\}$ and the remaining one over $\{\mathcal{L}\in\mathbb{D}:\, \mathcal{L}\ll L\}$. For the former, we observe that by using Plancharel identity we immediately obtain that \[
L^{1/2}\sum_{\mathcal{L}\gtrsim L}\left\Vert \eta_L\big(\tau-\omega(\xi,\mu)\big)\mathcal{F}(\mathds{1}_I(t)f_\mathcal{L})\right\Vert_{L^2}\lesssim \sum_{\mathcal{L}\gtrsim L}\mathcal{L}^{1/2}\left\Vert \eta_\mathcal{L}\big(\tau-\omega(\xi,\mu)\big)\mathcal{F}(f)\right\Vert_{L^2}.
\]
On the other hand, for the latter, since $I\subset\R$ is a bounded interval, it is not difficult to see that $\big\vert\mathcal{F}_t(\mathds{1}_I)(\tau)\big\vert\lesssim 1/\vert\tau\vert$. Thus, expanding the Fourier Transform and then applying Cauchy-Schwarz inequality along with the previous observation we infer that \begin{align*}
&L^{1/2}\sum_{\mathcal{L}\ll L}\left\Vert \eta_L\big(\tau-\omega(\xi,\mu)\big)\mathcal{F}(\mathds{1}_I(t)f_\mathcal{L})\right\Vert_{L^2}
\\ & \qquad \qquad \lesssim L^{1/2} \sum_{\mathcal{L}\ll L} \Big\Vert \eta_L\big(\tau-\omega(\xi,\mu)\big)\int_\R\big\vert \mathcal{F}(f)(\tau',\xi,\mu)\big\vert\cdot\dfrac{\eta_\mathcal{L}(\tau'-\omega(\xi,\mu))}{\vert \tau-\tau'\vert}d\tau'\Big\Vert_{L^2_{\tau,\xi,\mu}}
\\ & \qquad \qquad \lesssim \sum_{\mathcal{L}\ll L}\mathcal{L}^{1/2}\Vert \eta_\mathcal{L}\big(\tau-\omega(\xi,\mu)\big)\mathcal{F}(f)\Vert_{L^2},
\end{align*}
where we have used that, due to the presence of $\eta_L\cdot \eta_\mathcal{L}$ and since in this case we have that $\mathcal{L}\ll L$, then it holds \[
\vert \tau-\tau'\vert\leq\vert \tau-\omega(\xi,\mu)\vert+\vert\tau'-\omega(\xi,\mu)\vert\sim L.
\]
Therefore, gathering the above estimates and taking the supreme in $L$, we conclude the proof of the lemma.
\end{proof}

The following technical lemma provides a crucial property of the $X_H$ spaces. 

\begin{lem}\label{technical_estimate_lemma}
Let $L,H\in\mathbb{D}$ be fixed. Then, for all $f\in X_H$ the following inequalities hold \begin{align}
L^{\beta/2}\Big\Vert \eta_{\leq \lfloor L^\beta\rfloor}(\tau-\omega(\xi,\mu))\int_\R \vert f(\tau',\xi,\mu)\vert\cdot L^{-\beta}\big(1+L^{-\beta}\vert \tau-\tau'\vert\big)^{-4}d\tau' \Big\Vert_{L^2}&\lesssim \Vert f\Vert_{X_H},\label{technical_estimate_low}
\\ 
\sum_{\ell\in\mathbb{D}:\, \ell\geq L^\beta}\ell^{1/2}\Big\Vert \eta_{\ell}(\tau-\omega(\xi,\mu))\int_\R \vert f(\tau',\xi,\mu)\vert\cdot L^{-\beta}\big(1+L^{-\beta}\vert \tau-\tau'\vert\big)^{-4}d\tau' \Big\Vert_{L^2}&\lesssim \Vert f\Vert_{X_H},\label{technical_estimate_high}
\end{align} 
where $\lfloor L^\beta\rfloor$ denotes the largest dyadic number $d\in\mathbb{D}$ satisfying that $d\leq L^\beta$.
\end{lem}

\begin{proof}
First we seek to prove \eqref{technical_estimate_low}. In fact, by using Cauchy-Schwarz inequality we see that
\begin{align*}
&\mathrm{I}:=L^{\beta/2}\Big\Vert \eta_{\leq \lfloor L ^\beta\rfloor}(\tau-\omega(\xi,\mu))\int_\R \vert f(\tau',\xi,\mu)\vert\cdot L^{-\beta}\big(1+L^{-\beta}\vert \tau-\tau'\vert\big)^{-4}d\tau' \Big\Vert_{L^2_{\tau,\xi,\mu}}
\\ &\,\, = L^{-\beta/2}\Big\Vert \eta_{\leq \lfloor L ^\beta\rfloor}(\tau-\omega(\xi,\mu))\sum_{\mathsf{L}\in\mathbb{D}}\int_\R \eta_{\mathsf{L}}(\tau'-\omega(\xi,\mu))\vert f(\tau',\xi,\mu)\vert \big(1+L^{-\beta}\vert \tau-\tau'\vert\big)^{-4}d\tau' \Big\Vert_{L^2}
\\ & \leq L^{-\beta/2}\Big\Vert \eta_{\leq \lfloor L ^\beta\rfloor}(\tau-\omega(\xi,\mu))\sum_{\mathsf{L}\in\mathbb{D}}\mathrm{I}_{\mathsf{L}} \cdot\mathsf{L}^{1/2} \big\Vert \eta_{\mathsf{L}}(\tau'-\omega(\xi,\mu)) f(\tau',\xi,\mu)\big\Vert_{L^2_{\tau'}} \Big\Vert_{L^2_{\tau,\xi,\mu}},
\end{align*}
where $\mathrm{I}_{\mathsf{L}}$ is defined as \[
\mathrm{I}_{\mathsf{L}}:=\big\Vert \widetilde{\eta}_{\mathsf{L}}(\tau'-\omega(\xi,\mu))\big(1+L^{-\beta}\vert\tau-\tau'\vert\big)^{-4}\langle\tau'-\omega(\xi,\mu)\rangle^{-1/2}\big\Vert_{L^2_{\tau'}}.
\]
Now notice that, bounding $(1+L^{-\beta}\vert \tau-\tau'\vert)^{-4}\leq1$, we trivially get that \[
\mathrm{I}_{\mathsf{L}}\lesssim \mathsf{L}^{1/2} \mathsf{L}^{-1/2}\lesssim 1,
\]
which, plugging into the former estimate, and then recalling the definition of $X_H$, concludes the proof of \eqref{technical_estimate_low}.

\medskip

We now seek to prove \eqref{technical_estimate_high}.
In fact, first of all notice that by the Mean Value Theorem we have that \[
\big\vert \eta_\ell (\tau-\omega(\xi,\mu))-\eta_\ell(\tau'-\omega(\xi,\mu))\big\vert\lesssim \ell^{-1}\vert \tau-\tau'\vert.
\]
Then, plugging the latter estimate into the left-hand side of \eqref{technical_estimate_high} we infer that \begin{align*}
\sum_{\ell\in\mathbb{D}:\, \ell\geq L^\beta}\ell^{1/2}\Big\Vert \eta_{\ell}(\tau-\omega(\xi,\mu))\int_\R \vert f(\tau',\xi,\mu)\vert\cdot L^{-\beta}\big(1+L^{-\beta}\vert \tau-\tau'\vert\big)^{-4}d\tau' \Big\Vert_{L^2}\leq \mathrm{J}_1+\mathrm{J}_2,
\end{align*}
where \begin{align*}
\mathrm{J}_1&:= \sum_{\ell\in\mathbb{D}:\, \ell\geq L^\beta}\ell^{1/2}\left\Vert\Big(\big[\eta_\ell(\cdot-\omega(\xi,\mu))f(\cdot,\xi,\mu)\big]*\big[L^{-\beta}(1+L^{-\beta}\vert\cdot\vert)^{-4}\big]\Big)(\tau)\right\Vert_{L^2_{\tau,\xi,\mu}},
\\ \mathrm{J}_2&:=\sum_{\ell\in\mathbb{D}:\, \ell\geq L^\beta}\ell^{-1/2}\left\Vert \widetilde{\eta}_\ell(\tau-\omega(\xi,\mu))\int_\R\vert f(\tau',\xi,\mu)\vert  L^{-\beta}\vert \tau-\tau'\vert\big(1+L^{-\beta}\vert\tau-\tau'\vert\big)^{-4}d\tau'\right\Vert_{L^2_{\tau,\xi,\mu}}.
\end{align*}
Now, on the one-hand, notice that by basic computations we have $L^{-\beta}\Vert (1+L^{-\beta}\vert \cdot\vert )^{-4}\Vert_{L^1}\lesssim 1$. Thus, applying Young convolution inequality (in $\tau$) we obtain \[
\mathrm{J}_1\lesssim \sum_{\ell\in\mathbb{D}:\,\ell\geq L^\beta} \ell^{1/2}\Vert \eta_\ell(\tau-\omega(\xi,\mu))f(\tau,\xi,\mu)\Vert_{L^2}\leq \Vert f\Vert_{X_H}.
\]
On the other hand, to deal with $\mathrm{J}_2$ we can proceed similarly as we did for \eqref{technical_estimate_low}. In fact, decomposing into frequencies and then using Cauchy-Schwarz inequality we obtain that \[
\mathrm{J}_2\lesssim  L^{-\beta}\sum_{\ell\in\mathbb{D}:\,\ell\geq L^\beta}\ell^{-1/2}\Big\Vert \widetilde{\eta}_{\ell}(\tau-\omega(\xi,\mu))\sum_{\mathsf{L}\in\mathbb{D}}\mathrm{J}_{\mathsf{L}} \cdot\mathsf{L}^{1/2} \big\Vert \eta_{\mathsf{L}}(\tau'-\omega(\xi,\mu)) f(\tau',\xi,\mu)\big\Vert_{L^2_{\tau'}} \Big\Vert_{L^2_{\tau,\xi,\mu}},
\]
where in this case $\mathrm{J}_\mathsf{L}$ is given by \[
\mathrm{J}_\mathsf{L}(\tau,\xi,\mu):=\big\Vert \widetilde{\eta}_\mathsf{L}(\tau'-\omega(\xi,\mu))\vert \tau-\tau'\vert (1+L^{-\beta}\vert\tau-\tau'\vert)^{-4}\langle \tau'-\omega(\xi,\mu)\rangle^{-1/2}\big\Vert_{L^2_{\tau'}}.
\]
Now we split the analysis into two cases. First, assume that $\ell\lesssim \mathsf{L}$. Then, direct computations yield us to \[
\mathrm{J}_\mathsf{L}\lesssim \ell^{-1/2}\big\Vert \vert \cdot\vert (1+L^{-\beta}\vert \cdot\vert)^{-4}\big\Vert_{L^2}\lesssim\ell^{-1/2}L^{3\beta/2}.
\]
On the other hand, due to frequency localization, on the region $\ell \gg \mathsf{L}$ we have that $\vert \tau-\tau'\vert\sim \ell$. Hence, we have that \[
\mathrm{J}_\mathsf{L}\lesssim \ell(1+L^{-\beta}\ell)^{-4}\Vert \widetilde{\eta}_\mathsf{L}(\tau'-\omega(\xi,\mu))\langle \tau'-\omega(\xi,\mu)\rangle^{-1/2}\Vert_{L^2}\lesssim \ell^{-3} L ^{4\beta}.
\]
Plugging both estimates of $\mathrm{J}_\mathsf{L}$ for each of the previous cases into $\mathrm{J}_2$, we conclude that
\[
\mathrm{J}_2\lesssim \sum_{\mathsf{L}\in\mathbb{D}}\mathsf{L}^{1/2}\Vert \eta_\mathsf{L}(\tau-\omega(\xi,\mu))f(\tau,\xi,\mu)\Vert_{L^2}=\Vert f\Vert_{X_H}.
\]
The proof is complete.
\end{proof}

As a consequence of the previous lemma we derive several important properties. 
\begin{cor}\label{gro_gamma}
Let $\beta\geq0$, $H\in\mathbb{D}$, $t_H\in\R$ and consider $\gamma\in\mathcal{S}(\R)$. The, the following holds \begin{align}\label{gro_gamma_estimate}
\big\Vert \mathcal{F}\big(\gamma(H^\beta(\cdot-t_H))f\big)\Vert_{X_H}&\lesssim \Vert \mathcal{F}(f)\Vert_{X_H},
\end{align}
for all $f$ such that $\mathcal{F}(f)\in X_H$.
\end{cor}

\begin{proof}
In fact, it is enough to notice that, by standard properties of the Fourier transform along with the fact that $\mathcal{F}(\gamma)\in\mathcal{S}(\R^3)$, we have that \begin{align*}
\big\vert\mathcal{F}\big[\gamma(H^\beta(\cdot-t_H))f\big](\tau,\xi,\mu)\big\vert&=\big\vert \big(\mathcal{F}[f](\cdot,\xi,\mu)*[e^{-it_H(\cdot)}H^{-\beta}\mathcal{F}[\gamma](H^{-\beta}\cdot)]\big)(\tau)\big\vert 
\\ & \leq \int_\R \vert \mathcal{F}(f)(\tau',\xi,\mu)\vert \cdot H^{-\beta}(1+H^{-\beta}\vert\tau-\tau'\vert)^{-4}d\tau'.
\end{align*}
Therefore, taking the $X_H$-norm on both sides of the previous inequality, and then using \eqref{technical_estimate_low}-\eqref{technical_estimate_high}  we conclude the proof the corollary.
\end{proof}

\begin{cor}\label{lem_tau_w_ihb_xH}
Let $\beta\geq0$, $H\in\mathbb{D}$, $t_H\in\R$ and consider $\gamma\in\mathcal{S}(\R)$. Then, it holds that \begin{align}\label{tau_w_ihb_xH}
\big\Vert (\tau-\omega(\xi,\mu)+iH^\beta)^{ -1}\mathcal{F}\big(\gamma(H^\beta(\cdot-t_H))f\big)\big\Vert_{X_H}&\lesssim \big\Vert (\tau-\omega(\xi,\mu)+iH^\beta)^{-1}\mathcal{F}(f)\big\Vert_{X_H},
\end{align}
for all $f$ such that $\mathcal{F}(f)\in X_H$.
\end{cor}

\begin{proof}
In fact, first of all, by estimating as follows we can split the analysis into two cases \begin{align*}
&\big\Vert (\tau-\omega(\xi,\mu)+iH^\beta)^{ -1}\mathcal{F}\big(\gamma(H^\beta(\cdot-t_H))f\big)\big\Vert_{X_H}
\\ & \qquad \qquad \lesssim H^{-\beta}\sum_{L\in\mathbb{D}:\, L\leq  H^\beta} L^{1/2}\big\Vert \eta_L(\tau-\omega(\xi,\mu))\mathcal{F}\big[\gamma(H^\beta(\cdot-t_H))f\big]\big\Vert_{L^2_{\tau,\xi,\mu}}
\\ & \qquad \qquad \quad + \sum_{L\in\mathbb{D}:\, L>H^\beta} L^{-1/2}\big\Vert \eta_L(\tau-\omega(\xi,\mu))\mathcal{F}[\gamma(H^\beta(\cdot-t_H))f]\big\Vert_{L^2_{\tau,\xi,\mu}} .%
\end{align*}
For the first term in the right-hand side of the previous inequality, we proceed as in the proof of estimate \eqref{gro_gamma_estimate}. We write the Fourier transform of $\gamma(H^\beta(\cdot-t_H))f$ as the convolution of the Fourier transforms. Then, expanding the convolution, we use estimate \eqref{technical_estimate_low} to conclude this case. Similarly, to deal with the second term we expand the Fourier transform as before and then proceed as in the proof of \eqref{technical_estimate_high}. Thus, gathering both estimates we obtain \eqref{tau_w_ihb_xH}.
\end{proof}

More generally, for any $H\in\mathbb{D}$ and any $\beta \geq 0$ fixed, we define the set of $H$-acceptable time multiplication factors $S_{H,\beta}$ as follows \[
S_{H,\beta}:=\big\{m_H:\R\to\R:\, \Vert m_H\Vert_{S_{H,\beta}}:=\sum_{j=0}^{10}H^{-j\beta}\Vert \partial_jm_H\Vert_{L^\infty}<+\infty \big\}.
\]
\begin{lem}\label{time_mult_lema}
Let $\beta\geq 0$, $H\in\mathbb{D}$ and consider $m_H\in S_{H,\beta}$. Then, it holds that
\[
\big\Vert m_H(t)f\Vert_{F_{H,\beta}}\lesssim \Vert m_H\Vert_{S_{H,\beta}}\Vert f\Vert_{F_{H,\beta}},
\]
and \[
\big\Vert m_H(t)f\Vert_{N_{H,\beta}}\lesssim \Vert m_H\Vert_{S_{H,\beta}}\Vert f\Vert_{N_{H,\beta}}.
\]
\end{lem}
\begin{proof}
We shall only prove the first estimate since the second one follows very similar lines. In fact, notice that arguing as in the proof of Corollary \ref{gro_gamma} we infer that it suffices to prove that \begin{align}\label{c_time_mult}
\big\vert \mathcal{F}_t\big[m_H(\cdot)\eta_0(H^\beta\big(\cdot-\tilde{t}\big))\big](\tau)\big\vert \lesssim H^{-\beta}\big(1+H^{-\beta}\vert \tau\vert \big)^{-4}\Vert m_H\Vert_{S_{H,\beta}},
\end{align}
for all $\tau,\tilde{t}\in\R$. Then, in order to prove inequality \eqref{c_time_mult}, first observe that, using the fact that the Fourier transform is a bounded operator from $L^1$ to $L^\infty$, we obtain that \[
\big\Vert \mathcal{F}_t\big[m_H(\cdot)\eta_0\big(H^\beta(\cdot-\tilde{t})\big)\big]\big\Vert_{L^\infty}\lesssim \big\Vert m_H(\cdot)\eta_0\big(H^\beta(\cdot-\tilde{t})\big)\big\Vert_{L^1}\lesssim H^{-\beta}\Vert m_H\Vert_{L^\infty}\Vert \eta_0\Vert_{L^1}.
\]
On the other hand, from basic properties of the Fourier transform we see that \begin{align*}
H^{-4\beta}\vert \tau\vert^4\big\vert \mathcal{F}_t\big[m_H(\cdot)\eta_0\big(H^\beta(\cdot-\tilde{t})\big)\big](\tau)\big\vert &\lesssim H^{-4\beta}\big\Vert \partial_t\big(m_H(\cdot)\eta_0\big(H^\beta(\cdot-\tilde{t})\big)\big)\big\Vert_{L^1}
\\ & \lesssim H^{-4\beta}\sum_{j=0}^4\Vert \partial_t^jm_H\Vert_{L^\infty}H^{(4-j)\beta}H^{-\beta}\big\Vert \partial_t^{4-j}\eta_0\Vert_{L^1}.
\end{align*}

Therefore, gathering the above estimates, along with the definition of $S_{H,\beta}$, we obtain inequality  \eqref{c_time_mult}, which concludes the proof of the lemma.
\end{proof}
The next corollary shall be useful in the proof of the bilinear and energy estimates.
\begin{cor}\label{key_cor_prelim}
Let $\beta \geq 0$, $\tilde t\in\R$ be fixed  and consider $H,H_1\in\mathbb{D}$ satisfying that $H_1\geq 2^5 H$. Then, the following inequalities hold \begin{align}\label{key_cor_prelim_low_estimate}
H^{\beta/2}_1\big\Vert \eta_{\leq \lfloor H^\beta_1\rfloor}(\tau-\omega(\xi,\mu))\mathcal{F}\big(\eta_0 (H^\beta_1(\cdot-\tilde t))f\big)\big\Vert_{L^2_{\tau,\xi,\mu}}\lesssim \Vert f\Vert_{F_{H,\beta}},
\end{align}
and \begin{align}\label{key_cor_prelim_high_estimate}
\sum_{L> H_1^\beta}L^{1/2}\big\Vert \eta_{L}(\tau-\omega(\xi,\mu))\mathcal{F}\big(\eta_0 (H^\beta_1(\cdot-\tilde t))f\big)\big\Vert_{L^2_{\tau,\xi,\mu}}\lesssim \Vert f\Vert_{F_{H,\beta}},
\end{align}
for all $f\in F_{H,\beta}$.
\end{cor}

\begin{proof}
In fact it is enough to notice that, on the one-hand, due to the hypothesis on $H$ and $H_1$, we have that \[
\eta_0\big(H_1^\beta(\cdot-\tilde{t})\big)f=\eta_0\big(H_1^\beta(\cdot-\tilde{t})\big)\eta_0\big(2^{-6\beta}H^{\beta}(\cdot-\tilde{t})\big)f.
\]
On the other hand, it follows from Lemma \ref{time_mult_lema}  that \[
\big\Vert \mathcal{F}\big[\eta_0\big( 2^{-6\beta}H^{\beta}(\cdot-\tilde{t})\big)f\big]\big\Vert_{X_H}\lesssim \Vert f\Vert_{F_{H,\beta}}.
\]
Therefore, using estimates \eqref{technical_estimate_low} and \eqref{technical_estimate_high} and arguing exactly as in the proof of Corollary \ref{gro_gamma} we conclude the proof of the lemma.
\end{proof}

\subsection{Linear estimates}

In this subsection we derive the linear properties enjoyed by the short-time function spaces introduced at the beginning of this section (see \cite{IoKeTa,KePi}).

\begin{prop}\label{linear_prop}
Let $s\geq 0$, $\beta >0$ and consider $T\in (0,1]$. Then, the following holds \begin{align}\label{linear_prop_estimate}
\Vert u\Vert_{F^s_\beta(T)}\lesssim \Vert u\Vert_{B^s(T)}+\Vert f\Vert_{N^s_\beta(T)},
\end{align}
for all $u\in B^s(T)$ and $f\in N^s_\beta(T)$ satisfying that \begin{align}\label{linear_PDE}
\partial_tu+ \partial_x\Delta u=f, \quad \hbox{ on } \quad [-T,T]\times\R^2.
\end{align}
\end{prop}

In order to prove the previous proposition, we split the analysis into two simpler lemma. First we treat the homogeneous case.

\begin{lem}\label{homogeneous_case}
Let $\beta\geq 0$ and $H\in\mathbb{D}$. Then, the following inequality holds \begin{align}\label{homogeneous_case_estimate}
\big\Vert \mathcal{F}\big(\eta_0(H^\beta t)e^{t\partial_x(D_x^2-\partial_y^2)}u_0\big)\big\Vert_{X_H}\lesssim \Vert u_0\Vert_{L^2},
\end{align}
for all $u_0\in L^2(\R^2)$ such that $\mathrm{supp}\,\mathcal{F}(u_0)\in I_H$.
\end{lem}

\begin{proof}
In fact, first of all notice that, by direct computations, using the standard properties of the Fourier transform, we can see that \[
\mathcal{F}\big[\eta_0(H^\beta t)e^{t\partial_x(D_x^2-\partial_y^2)}u_0\big](\tau,\xi,\mu)=H^{-\beta}\widehat{\eta}_0\big(H^{-\beta}(\tau-\omega(\xi,\mu))\big)\widehat{u}_0(\xi,\mu).%
\]
Hence, taking the $X_H$-norm to the previous identity, and then using Plancharel Theorem, it follows that \[
\big\Vert \mathcal{F}\big[\eta_0(H^\beta t)e^{t\partial_x(D_x^2-\partial_y^2)}u_0\big]\big\Vert_{X_H}\leq \sum_{L\in\mathbb{D}} L^{1/2}\big\Vert \eta_L(\cdot)H^{-\beta}\widehat{\eta}_0(H^{-\beta}\cdot)\big\Vert_{L^2}\Vert u_0\Vert_{L^2}.
\]
On the other hand, using the fact that $\widehat{\eta}_0\in\mathcal{S}(\R)$, it is not difficult to see that \[
\big\Vert \eta_L(\cdot)H^{-\beta}\widehat{\eta}_0(H^{-\beta}\cdot)\big\Vert_{L^2}\lesssim %
H^{-\beta}\big\Vert\eta_L(\cdot)(1+H^{-\beta}\vert \cdot\vert)^{-4}\big\Vert_{L^2}%
\lesssim H^{-\beta}L^{1/2}\mathrm{min}\{1,H^{4\beta}L^{-4}\}.
\]
Gathering the above estimates we conclude the proof of the lemma.
\end{proof}

Now we seek to study the non-homogeneous case.

\begin{lem}\label{nonhomogeneous_case}
Let $\beta\geq 0$ and consider $H\in\mathbb{D}$. Then, it holds that
\begin{align}
\Big\Vert \mathcal{F}\big[\eta_0(H^\beta t)\int_0^t e^{(t-s)\partial_x(D_x^2-\partial_y^2)}f(s,\cdot,\cdot)ds\big]\Big\Vert_{X_H}\lesssim \big\Vert (\tau-\omega(\xi,\mu)+iH^\beta)^{-1}\mathcal{F}(f)\big\Vert_{X_H},
\end{align}
for all functions $f$ such that $\mathrm{supp}\,\mathcal{F}(f)\in\R\times \Delta_H$.
\end{lem}

\begin{proof}
In fact, by direct computations, using Plancharel Theorem and  standard properties of the Fourier transform, we obtain that \begin{align*}
& \mathcal{F}\big[\eta_0(H^\beta t)\int_0^t e^{(t-s)\partial_x(D_x^2-\partial_y^2)}f(s,\cdot,\cdot)ds\big]%
\\ & \quad  =\mathcal{F}_t\big[\eta_0(H^\beta t)\int_\R \dfrac{e^{it\tilde\tau}-e^{it\omega(\xi,\mu)}}{i(\tilde\tau-\omega(\xi,\mu))}\mathcal{F}(f)d\tilde\tau \big](\tau)
\\ & \quad =H^{-\beta} \widehat\eta_0(H^{-\beta}\cdot)*\left[\dfrac{\mathcal{F}(f)(\cdot,\xi,\mu)}{i(\cdot-\omega(\xi,\mu))}\right](\tau)-\mathcal{F}_t\big[\eta_0(H^\beta t)e^{it\omega(\xi,\mu)}\big](\tau)\int_\R\dfrac{\mathcal{F}(f)(\tilde\tau,\xi,\mu)}{i(\tilde\tau-\omega(\xi,\mu))}d\tilde\tau
\\ & \quad = H^{-\beta}\int_\R \dfrac{\widehat\eta_0\big(H^{-\beta}(\tau-\tilde\tau)\big)-\widehat\eta_0\big(H^{-\beta}(\tau-\omega(\xi,\mu))\big)}{i(\tilde\tau-\omega(\xi,\mu))}\mathcal{F}(f)(\tilde\tau,\xi,\mu) d\tilde\tau
\\ & \quad = H^{-\beta}\int_\R \dfrac{\widehat\eta_0\big(H^{-\beta}(\tau-\tilde\tau)\big)-\widehat\eta_0\big(H^{-\beta}(\tau-\omega(\xi,\mu))\big)}{i(\tilde\tau-\omega(\xi,\mu))}\cdot\dfrac{\tilde\tau-\omega(\xi,\mu)+iH^\beta}{\tilde\tau-\omega(\xi,\mu)+iH^\beta}\mathcal{F}(f)(\tilde\tau,\xi,\mu) d\tilde\tau.
\end{align*}
Now, we claim that the following inequality holds \begin{align}\label{claim_nh_linear_estimate}
&\Big\vert \dfrac{\widehat\eta_0\big(H^{-\beta}(\tau-\tilde\tau)\big)-\widehat\eta_0\big(H^{-\beta}(\tau-\omega(\xi,\mu))\big)}{i(\tilde\tau-\omega(\xi,\mu))}\cdot\big(\tilde\tau-\omega(\xi,\mu)+iH^\beta\big)\Big\vert \nonumber
\\ & \qquad \qquad \qquad \qquad \qquad \qquad      \lesssim \big(1+H^{-\beta}\vert\tau-\tilde\tau\vert\big)^{-4}+\big(1+H^{-\beta}\vert\tau-\omega(\xi,\mu)\vert\big)^{-4}.%
\end{align}
We divide the analysis of the latter claim into two cases. First, notice that if $\vert \tilde\tau-\omega(\xi,\mu)\vert\geq H^\beta$, then it immediately follows that \[
\big\vert\tilde\tau-\omega(\xi,\mu)+iH^\beta\big\vert \lesssim \vert \tilde\tau-\omega(\xi,\mu)\vert.
\]
Therefore, in this case, the claim follows immediately from the above inequality along with the fact that $\widehat\eta_0\in\mathcal{S}(\R)$. Hence, from now on we assume that $\vert \tilde\tau-\omega(\xi,\mu)\vert < H^\beta$. In this case, by using the Mean Value Theorem we deduce that \[
\big\vert\widehat\eta_0\big(H^{-\beta}(\tau-\tilde\tau)\big)-\widehat\eta_0\big(H^{-\beta}(\tau-\omega(\xi,\mu))\big)\big\vert\leq H^{-\beta}\big\vert\widehat\eta_0'\big(H^{-\beta}\theta\big)\big\vert\cdot\big\vert\tilde\tau-\omega(\xi,\mu)\big\vert,
\]
form some $\theta\in (\vert\tau-\tilde\tau\vert,\vert\tau-\omega(\xi,\mu)\vert)\cup (\vert\tau-\omega(\xi,\mu)\vert,\vert\tau-\tilde\tau\vert)$. Then, notice that, since $\widehat\eta_0'\in\mathcal{S}(\R)$, we have the bound \[
\vert\widehat\eta_0'\big(H^{-\beta}\theta\big)\big\vert\lesssim \big(1+H^{-\beta}\theta\big)^{-4}.
\]
Thus, plugging the latter two inequalities into the left-hand side of \eqref{claim_nh_linear_estimate}, recalling that $\theta\in (\vert\tau-\tilde\tau\vert,\vert\tau-\omega(\xi,\mu)\vert)\cup (\vert\tau-\omega(\xi,\mu)\vert,\vert\tau-\tilde\tau\vert)$, we conclude the proof of the claim.

\medskip

Now notice that with the previous claim we can reduce the analysis to the study of the following two quantities \begin{align*}
I&:=\Big\Vert H^{-\beta}\int_\R \dfrac{\vert \mathcal{F}(f)(\tilde\tau,\xi,\mu)\vert }{\big\vert \tilde\tau-\omega(\xi,\mu)+iH^\beta\big\vert} \cdot\big(1+H^{-\beta}\vert \tau-\tilde\tau\vert\big)^{-4} d\tilde\tau\Big\Vert_{X_H},
\\ II&:=\Big\Vert H^{-\beta}\int_\R \dfrac{\vert \mathcal{F}(f)(\tilde\tau,\xi,\mu)\vert }{\big\vert \tilde\tau-\omega(\xi,\mu)+iH^\beta\big\vert} \cdot\big(1+H^{-\beta}\vert \tau-\omega(\xi,\mu)\vert\big)^{-4} d\tilde\tau\Big\Vert_{X_H}.
\end{align*}
In fact, for the first quantity, by directly using Lemma \ref{technical_estimate_lemma} we obtain that \[
I\lesssim H^{- \beta/2}\big\Vert \big(\tau-\omega(\xi,\mu)+iH^\beta\big)^{-1}\mathcal{F}(f)(\tau,\xi,\mu)\big\Vert_{X_H}.
\]
On the other hand, expanding the $X_H$-norm and then using \eqref{l1l2_xh} we get that \begin{align*}
II&\lesssim \sum_{L\in\mathbb{D}}L^{1/2}\big\Vert \eta_L(\cdot)H^{-\beta}\big(1+H^{-\beta}\vert\cdot\vert\big)^{-4}\big\Vert_{L^2}\Big\Vert \int_\R \dfrac{\vert \mathcal{F}(f)(\tilde\tau,\xi,\mu)\vert }{\big\vert \tilde\tau-\omega(\xi,\mu)+iH^\beta\big\vert} d\tilde\tau\Big\Vert_{L^2}
\\ & \lesssim\sum_{L\in\mathbb{D}}LH^{-\beta}\mathrm{min}\{1,H^{4\beta}L^{-4}\} \Big\Vert \int_\R \dfrac{\vert \mathcal{F}(f)(\tilde\tau,\xi,\mu)\vert }{\big\vert \tilde\tau-\omega(\xi,\mu)+iH^\beta\big\vert} d\tilde\tau\Big\Vert_{L^2}
\\ & \lesssim \big\Vert \big(\tau-\omega(\xi,\mu)+iH^\beta)^{-1}\mathcal{F}(f)(\tau,\xi,\mu)\big\Vert_{X_H}.
\end{align*}
Gathering the above estimates we conclude the proof of the lemma.
\end{proof}

Now we are in position to prove our main linear estimate.

\begin{proof}[Proof of Proposition \ref{linear_prop}]
In fact, let $u,f:[-T,T]\times \R^2\to\R$ satisfying \eqref{linear_PDE}. We consider $\tilde{f}$ an extension of $f$ on $\R^3$ satisfying that \[
\Vert \tilde{f}\Vert_{N^s_\beta}\leq 2\Vert f\Vert_{N^s_\beta(T)}.
\]
Now we consider $\rho\in C^\infty(\R^2)$ satisfying that $\rho(t)\equiv 1$ for all $t\geq 1$ and that $\rho(t)\equiv0$ for all $t\leq 0$. Then, for $H\in\mathbb{D}$ we define \[
\tilde{f}_H:=\rho\left(2^{10}H^\beta\big(t+T+2^{-10}H^{-\beta}\big)\right)\rho\left(-2^{10}H^\beta\big(t-T-2^{-10}H^{-\beta}\big)\right)P_H\tilde{f}.
\] 
Then, from these definitions and Corollary \eqref{time_mult_lema} it follows that \[
\tilde{f}_H\big\vert_{[-T,T]}\equiv P_Hf, \qquad  \Vert \tilde{f}_H\Vert_{N_{H,\beta}}\lesssim \Vert P_H\tilde{f}\Vert_{N_{H,\beta}},
\]
and that \[
\supp \tilde{f}_H\subset \big[-T-2^{-10}H^{-\beta},T+2^{-10}H^{-\beta}\big]\times \R^2.
\]
Now we extend $P_Hu$ on $\R^3$. To this end, for all $H\in\mathbb{D}$ we define $\tilde{u}_H(t)$ as follows \begin{align*}
\begin{cases}
\eta_0\big(2^5H^\beta(t+T)\big)\big(e^{(t+T)\partial_x(D_x^2-\partial_y^2)}P_Hu(-T)+\int_{-T}^te^{(t-s)\partial_x(D_x^2-\partial_y^2)}\tilde{f}_H(s)ds\big) & \hbox{if} \quad t<-T, 
\\ \eta_0\big(2^5H^\beta(t-T)\big)\big(e^{(t-T)\partial_x(D_x^2-\partial_y^2)}P_Hu(T)+\int_{T}^te^{(t-s)\partial_x(D_x^2-\partial_y^2)}\tilde{f}_H(s)ds\big) & \hbox{if} \quad t>T,
\end{cases}
\end{align*} 
and $\tilde{u}_H(t)=P_Hu(t)$ for $t\in[-T,T]$. Now we claim that \begin{align}\label{claim_linear}
\Vert \tilde{u}_H\Vert_{F_{H,\beta}}\lesssim \sup_{t_H\in[-T,T]} \big\Vert \mathcal{F}\big[\eta_0\big(H^\beta(t-t_H)\big)\tilde{u}_H\big]\big\Vert_{X_H}.
\end{align}
In fact, first of all notice that, from the definition of $\tilde{u}_H$ it follows that \[
\supp \tilde{u}_H\subseteq [-T-2^{-5}H^{-\beta},T+2^{-5}H^{-\beta}]\times \R^2.
\]
Therefore, for any $t_H>T$ and any $\tilde{t}_H\in[T-H^{-\beta},T]$, we have that \[
\eta_0\big(H^\beta(t-t_H)\big)\tilde{u}_H=\eta_0\big(H^\beta(t-t_H)\big)\eta_0\big(H^\beta(t-\tilde{t}_H)\big)\tilde{u}_H,
\]
and hence, the latter identity along with Lemma \ref{gro_gamma} imply that \[
\sup_{t_H>T}\big\Vert \mathcal{F}\big[\eta_0\big(H^\beta(\cdot-t_H)\big)\tilde{u}_H\big]\big\Vert_{X_H}\lesssim \sup_{\tilde{t}_H\in[-T,T]}\big\Vert \mathcal{F}\big[\eta_0\big(H^\beta(\cdot-t_H)\big)\tilde{u}_H\big]\big\Vert_{X_H}.
\]
Arguing similarly we obtain the same estimate for $t_H<-T$, which concludes the proof of \eqref{claim_linear}. Next, we fix $t_H\in[-T,T]$ and observe that \[
\big\Vert \mathcal{F}\big[\eta_0(H^\beta(\cdot-t_H))\tilde{u}_H\big] \big\Vert_{X_H}=\big\Vert \mathcal{F}\big[\eta_0\big(H^\beta\cdot\big)\tilde{u}_H(\cdot+t_H)\big]\big\Vert_{X_H}.
\]
On the other hand, from Duhamel principle we get that \begin{align*}
\eta_0\big(H^\beta t\big)\tilde{u}_H(t+t_H)&=m_H(t)\eta_0\big(H^\beta t\big)e^{it\partial_x(D_x^2-\partial_y^2)}P_Hu(t_H)
\\ & \quad +m_H(t)\eta_0\big(H^\beta t\big)\int_0^te^{i(t-s)\partial_x(D_x^2-\partial_y^2)}\widetilde{\eta}_0\big(H^\beta s\big)\tilde{f}_H(s+t_H)ds,
\end{align*}
where $m_H\in S_{H,\beta}$. Thus, we deduce from Lemmas \ref{time_mult_lema},  \ref{homogeneous_case} and \ref{nonhomogeneous_case}  that \begin{align*}
\big\Vert\mathcal{F}\big[\eta_0\big(H^\beta\cdot\big)\tilde{u}_H(\cdot+t_H)\big]\big\Vert_{X_H}&\lesssim \Vert P_Hu(t-H)\Vert_{L^2}
\\ & \quad +\big\Vert \big(\tau-\omega(\xi,\mu)+iH^\beta\big)^{-1}\mathcal{F}\big[\widetilde{\eta}_0\big(H^\beta\cdot\big)\tilde{f}_H(\cdot+t_H)\big] \big\Vert_{X_H}.
\end{align*}
Now we define $\tilde{u}=\sum_{H\in\mathbb{D}}\tilde{u}_H$. Notice that $\tilde{u}$ extends $u$ outside of the interval $[-T,T]$. Finally, it is not difficult to see, by direct computations, that 
\[
\big\Vert P_H\tilde{u}\big\Vert_{F_H}\lesssim \sum_{H'\sim H}\big\Vert \tilde{u}_{H'}\big\Vert_{F_{H'}}.
\]
Thus, gathering the above estimates, taking the supreme in $t_H\in[-T,T]$ and then summing over $H\in\mathbb{D}$ we conclude the proof of the proposition. 
\end{proof}

\subsection{Strichartz estimates}

In this sub-section we seek to prove some Strichartz estimates that shall be particularly useful while working outside the curves $\mu^2= 3\xi^2$. In \cite{CaKeZi},  Carbery, Kenig and Ziesler proved an optimal $L^4$-restriction theorem for homogeneous polynomial hypersurfaces in $\R^3$.

\begin{thm}[\cite{CaKeZi}]\label{CKZ_THM}
Let $\Gamma(\xi,\mu):=\big(\xi,\mu,\omega(\xi,\mu)\big)$, where $\omega(\xi,\mu)$ is an homogeneous polynomial of degree greater or equal than $2$. Then, there exists a positive constant $C>0$ (depending on $\omega$) such that \begin{align*}
\Big(\int_{\R^2}\big\vert \mathcal{F}[f]\big(\Gamma(\xi,\mu)\big)\big\vert^2 \big\vert K_\omega(\xi,\mu)\big\vert^{1/4}d\xi d\mu \Big)^{1/2}\leq C\Vert f\Vert_{L^{4/3}},
\end{align*}
for all $f\in L^{4/3}(\R^3)$, where $\vert K_\omega(\xi,\mu)\vert = \vert \mathrm{det}\,\mathrm{Hess}\,\omega(\xi,\mu)\vert $.
\end{thm}

As a direct consequence of the previous Theorem, we have the following general Strichartz estimate.

\begin{cor}
Let $e^{it\omega(D)}$ and $\vert K_\omega(D)\vert^{1/8}$ be the Fourier multipliers associated with $e^{it\omega(\xi,\mu)}$ and $\vert K_\omega(\xi,\mu)\vert^{1/8}$, that is, \[
\mathcal{F}_{xy}\big[e^{it\omega(D)}f\big](\xi,\mu)=e^{it\omega(\xi,\mu)}\mathcal{F}_{xy}[f](\xi,\mu),
\]
and \[
\mathcal{F}_{xy}\big[\vert K_\omega(D)\vert^{1/8}f\big](\xi,\mu)=\vert K_\omega(\xi,\mu)\vert^{1/8}\mathcal{F}_{xy}[f](\xi,\mu),
\]
where $K_\omega(\xi,\mu)$ is defined as above. Then, under the hypotheses of the previous theorem, the following inequality holds \[
\big\Vert \vert K_\omega(D)\vert^{1/8}e^{it\omega(D)}f\big\Vert_{L^4_{txy}}\lesssim \Vert f\Vert_{L^2_{xy}},
\]
for all $f\in L^2(\R^2)$.
\end{cor}

\begin{proof}
In fact, by duality we infer that it suffices to prove that \[
\int_{\R^3}\vert K_\omega(D)\vert^{1/8}e^{it\omega(D)}f(x,y)g(t,x,y)dxdydt \lesssim \Vert f\Vert_{L^2_{xy}}\Vert g\Vert_{L^{4/3}_{txy}}.
\]
Now, in order to prove the latter estimate, a direct application of Cauchy-Schwarz inequality show us that, it suffices, in turn, to prove that \begin{align}\label{ckz_cor_suffices}
\Big\Vert \int_\R \vert K_\omega(D)\vert^{1/8}e^{-it\omega(D)} g(t,x,y)dt\Big\Vert_{L^2_{xy}}\lesssim \Vert g\Vert_{L^{4/3}_{txy}}.
\end{align}
However, by direct computations we see that \[
\mathcal{F}_{xy}\Big[\int_{\R}\vert K_\omega(D)\vert^{1/8}e^{-it\omega(D)}g(t,x,y)dt\Big](\xi,\mu)=c \big\vert K_\omega(\xi,\mu)\big\vert^{1/8}\mathcal{F}_{txy}[f]\big(\omega(\xi,\mu),\xi,\mu\big).
\]
Therefore, estimate \eqref{ckz_cor_suffices} follows directly from the previous identity along with an application of Theorem \ref{CKZ_THM}. The proof is complete.
\end{proof}

In the particular case of the Zakharov-Kuznetsov equation, that is, in the case where $\omega(D)=-\partial_x\Delta$, we have that $\vert K_\omega(\xi,\mu)\vert =\vert 3\xi^2-\mu^2\vert$. Then, the previous Corollary translates into the following property.
\begin{prop}\label{Str_ZK_prop_final}
Let $\vert K(D)\vert^{1/8}$ be the Fourier multiplier associated with $\vert K(\xi,\mu)\vert^{1/8}=\vert 3\xi^2-\mu^2\vert^{1/8}$. Then, it holds that \begin{align*}
\big\Vert \vert K(D)\vert^{1/8}e^{-t\partial_x\Delta}u_0\big\Vert_{L^4_{txy}}\lesssim \Vert u_0\Vert_{L^2},
\end{align*}
for all $u_0\in L^2(\R^2)$. Moreover, for all $u\in X^{0,1/2,1}$, it also holds that \begin{align}\label{Str_ZK}
\big\Vert \vert K(D)\vert^{1/8}u\big\Vert_{L^4_{txy}}\lesssim \Vert u\Vert_{X^{0,1/2,1}}.
\end{align}
\end{prop}

\begin{rem}
In the previous statement, the $X^{0,1/2,1}$ space denotes the $\ell^1$-Besov version of the classical Bourgain space of regularity $1/2$ associated with the linear part of the Zakharov-Kuznetsov equation \eqref{zk}.
\end{rem}

\medskip

\section{$L^2$ bilinear estimates}\label{sec_bilinear_estimates}

Before going further let us recall that $\omega(\xi,\mu)=\xi(\xi^2+\mu^2)$ and \begin{align}\label{resonant}
\Omega(\xi_1,\mu_1,\xi_2,\mu_2)&=\omega(\xi_1+\xi_2,\mu_1+\mu_2)-\omega(\xi_1,\mu_1)-\omega(\xi_2,\mu_2)
\\ &= \xi_3^3-\xi_1^3-\xi_2^3+\xi_1\mu_2\mu_3+\xi_2\mu_1\mu_3+\xi_3\mu_1\mu_2,\nonumber
\end{align}
where we are denoting by $\xi_3:=\xi_1+\xi_2$ and $\mu_3:=\mu_1+\mu_2$. Now, for $N,H,L\in\mathbb{D}$ dyadic numbers, we define the sets $D_{N,H,L}$ and $D_{\infty,H,L}$ by 
\begin{align}\label{DNHL_set}
D_{N,H,L}&:=\{(\tau,\xi,\mu)\in\R^3: \, \xi\in I_N, \, (\xi,\mu)\in \Delta_H \hbox{ and } \vert\tau-\omega(\xi,\mu)\vert\leq L\},
\\ D_{\infty,H,L}&:=\{(\tau,\xi,\mu)\in\R^3: \, (\xi,\mu)\in \Delta_H \hbox{ and } \vert\tau-\omega(\xi,\mu)\vert\leq L\}=\cup_{N\in\mathbb{D}} D_{N,H,L}.\nonumber
\end{align}
The main result of this subsection are the following bilinear estimates, which shall be useful throughout the rest of this article. We refer to \cite{MoPi,RiVe} for some similar bilinear estimates.
\begin{lem}\label{lemma_bilineal}
Let $N_i,H_i,L_i\in\mathbb{D}$ be dyadic numbers and assume that  $f_i:\R^3\to\R_+$ are $L^2(\R^3)$ functions supported in $D_{\infty,H_i,L_i}$, for $i=1,2,3$.
\begin{enumerate}
\item[a.] Then, the following bound holds \begin{align}\label{general_case_bilinear}
\int_{\R^3} (f_1*f_2)\cdot f_3 \lesssim H_\mathrm{min}^{1/2}L_{\mathrm{min}}^{1/2}\Vert f_1\Vert_{L^2}\Vert f_2\Vert_{L^2}\Vert f_3\Vert_{L^2}.
\end{align}
\item[b.] Let us suppose that $H_{\mathrm{min}}\ll H_{\mathrm{max}}$. If we are in the case where $(H_i,L_i)=(H_{\mathrm{min}}, L_{\mathrm{max}})$ for some $i\in\{1,2,3\}$, then the following inequality holds
\begin{align}\label{hmin_ll_hmax_bc}
\int_{\R^3} (f_1*f_2)\cdot f_3 \lesssim  H_\mathrm{max}^{-1/2}H_\mathrm{min}^{1/4}L_\mathrm{min}^{1/2} L_{\mathrm{max}}^{1/2}  \Vert f_1\Vert_{L^2}\Vert f_2\Vert_{L^2}\Vert f_3\Vert_{L^2}.
\end{align}
Otherwise we have that \begin{align}\label{hmin_ll_hmax_aoc}
\int_{\R^3} (f_1*f_2)\cdot f_3 \lesssim    H_\mathrm{max}^{-1/2}H_\mathrm{min}^{1/4}L_\mathrm{min}^{1/2}L_{\mathrm{med}}^{1/2}\Vert f_1\Vert_{L^2}\Vert f_2\Vert_{L^2}\Vert f_3\Vert_{L^2}.
\end{align}
\item[c.] If $H_\mathrm{min}\sim H_\mathrm{max}$ and $f_i$ are supported in $D_{N_i,H_i,L_i}$ for $i=1,2,3$, then \begin{align}\label{hmin_sim_hmax_bilinear}
\int_{\R^3} (f_1*f_2)\cdot f_3 \lesssim    N^{-1}_\mathrm{max}H_\mathrm{min}^{1/4}L_\mathrm{med}^{1/2}L_{\mathrm{max}}^{1/2}\Vert f_1\Vert_{L^2}\Vert f_2\Vert_{L^2}\Vert f_3\Vert_{L^2}.
\end{align}
\end{enumerate}
\end{lem}
In order to prove Lemma \ref{lemma_bilineal} we shall the following definitions. For $\delta\in(0,1)$ sufficiently small, we define the sets \[
A_\delta:=\big\{(\xi,\mu)\in\R^2: \, (3-\delta)\xi^2\leq \mu^2 \leq (3+\delta)\xi^2\big\}.
\]
Now consider $\rho\in C^\infty(\R)$, satisfying that $0\leq \rho\leq 1$ and such that $\rho\equiv 0$ on $[-\tfrac{1}{2},\tfrac{1}{2}]$ and $\rho(\xi)\equiv1$ for all $\vert\xi\vert\geq 1$. We also denote by $\rho_\delta$ the function given by $\rho_\delta(\xi):=\rho(\xi/\delta)$. Notice that with these definitions we have that \begin{align}\label{rho_char_identity}
\rho_\delta\left(3-\dfrac{\mu^2}{\xi^2}\right)\mathds{1}_{A_\delta^c}(\xi,\mu)=\mathds{1}_{A_\delta^c}(\xi,\mu) \quad \hbox{ for all } \quad (\xi,\mu)\in\R^*\times \R,
\end{align}
as well as \[
\rho_\delta\left(3-\dfrac{\mu^2}{\xi^2}\right)\equiv0 \quad \hbox{ for all } \quad (\xi,\mu)\in A_{\delta/2}.
\]
From the last two identities we infer that we can see $\rho_\delta(3-\mu^2/\xi^2)$ as a sort of smooth version of $\mathds{1}_{A_\delta^c}$. The following technical lemma will be important in proving the above bilinear estimates.

\begin{lem}\label{lem_bil_incurve}
Let $\delta\in(0,1)$ be fixed. Consider $(\xi_i,\mu_i)\in\R^2$ for $i=1,2$, satisfying that \[
(\xi_1,\mu_1),(\xi_2,\mu_2)\in A_\delta, \qquad \xi_1\xi_2<0, \quad \ \hbox{ and } \ \quad \mu_1\mu_2<0.
\]
Then, it holds that \[
h(\xi_1+\xi_2,\mu_1+\mu_2)\leq \vert h(\xi_1,\mu_1)-h(\xi_2,\mu_2)\vert+f(\delta)\max\{h(\xi_1,\mu_1),h(\xi_2,\mu_2)\},
\]
where $f$ is a continuous function on $[0,1]$ satisfying that $\lim_{\delta \to 0^+} f(\delta)=0$.
\end{lem}

\begin{proof}
Indeed, first of all notice that without loss of generality we can assume that $h(\xi_1,\mu_1)\geq h(\xi_2,\mu_2)$. Therefore, by using the definition of $h(\cdot,\cdot)$ we infer that it is enough to prove that \[
3\big(( \xi_1+\xi_2)^2+\xi_2^2-\xi_1^2\big)+2\mu_2(\mu_1+\mu_2)\leq f(\delta)h(\xi_1,\mu_1).
\]
In fact, notice that if $(\xi_2,\mu_2)\in A_\delta$, then we have that \[
(6-\delta)\xi_2^2\leq 3\xi_2^2+\mu_2^2=h(\xi_2,\mu_2),
\]
which, since $(\xi_1,\mu_1)\in A_\delta$ and $h(\xi_2,\mu_2)\leq h(\xi_1,\mu_1)$ as well, in turn implies that \begin{align}\label{tech_estimacion_base}
\xi_2^2\leq \left(\dfrac{6+\delta}{6-\delta}\right)\xi_1^2\xrightarrow{\delta\to0}\xi_1^2.
\end{align}
Then, using the latter inequality, the fact that $\mu_1\mu_2<0$ and $(\xi_i,\mu_i)\in A_\delta$, we get that \begin{align*}
\mu_2(\mu_1+\mu_2)\leq (3+\delta)\xi_2^2-(3-\delta)\vert \xi_1\xi_2\vert\leq f_1(\delta)\xi_2^2\xrightarrow{\delta \to0}\xi_2^2,
\end{align*} 
where $f_1(\delta)$ is given by \[
f_1(\delta):= 3+\delta-(3-\delta)\left(\dfrac{6+\delta}{6-\delta}\right)^{-1/2}.
\]
Now on the one-hand, by using inequality \eqref{tech_estimacion_base} and the fact that $\xi_1\xi_2<0$, we infer that, if $\vert\xi_1\vert\leq \vert\xi_2\vert$, then we have that \[
( \xi_1+\xi_2)^2+\xi_2^2-\xi_1^2=\big( \vert\xi_2\vert-\vert\xi_1\vert\big)^2 +\xi_2-\xi_1^2\leq f_2(\delta)\xi_1^2\xrightarrow{\delta\to0}\xi_1^2,
\]
where \[
f_2(\delta):=\Big(\left(\tfrac{6+\delta}{6-\delta}\right)^{1/2}-1\Big)^2+\tfrac{6+\delta}{6-\delta}-1.
\]
On the other hand, if $\vert \xi_1\vert >\vert\xi_2\vert$, it is enough to notice that \[
( \xi_1+\xi_2)^2+\xi_2^2-\xi_1^2=\big( \vert\xi_1\vert-\vert\xi_2\vert\big)^2 +\xi_2^2-\xi_1^2\leq 0.
\]
Gathering the above computations we conclude the proof of the lemma.
\end{proof}

\begin{proof}[Proof of Lemma \ref{lemma_bilineal}]
In fact, let us start by proving estimate \eqref{general_case_bilinear}. First of all notice that
\begin{align}\label{definition_I_bilineal}
I:=\int_{\R^3} (f_1*f_2)\cdot f_3 = \int_{\R^3} (\widetilde{f}_1*f_3)\cdot f_2 = \int_{\R^3} (\widetilde{f}_2*f_3)\cdot f_1,
\end{align}
where $\widetilde{f}_i(\tau,\xi,\mu):=f_i(-\tau,-\xi,-\mu)$, $i=1,2$. Then, in view of the above symmetry, without loss of generality we can assume that $L_1=L_{\mathrm{min}}$. Now define $f_i^\#(\theta,\xi,\mu):=f_i(\theta+\omega(\xi,\mu),\xi,\mu)$, for $i=1,2,3$. Clearly we have that $\Vert f_i^\#\Vert_{L^2}=\Vert f_i\Vert_{L^2}$. Also, due to the assumptions on $f_i$, it follows that the functions $f_i^\#$ are supported in the sets \[
D_{\infty,H_i,L_i}^\#:=\{(\theta,\xi,\mu)\in\R^3:\, (\xi,\mu)\in \Delta_{H_i} \hbox{ and } \vert \theta\vert\leq L_i\}.
\]
Then, expanding the convolution, changing variables and using the above definitions, we rewrite the left-hand side of \eqref{definition_I_bilineal} as \begin{align}\label{definition_I_modi_bilinear}
I:=\int_{\R^6} f_1^\#(\theta_1,\xi_1,\mu_1)f_2^\#(\theta_2,\xi_2,\mu_2)f_3^\#\big(\theta_1+\theta_2+\Omega(\xi_1,\mu_1,\xi_2,\mu_2),\xi_1+\xi_2,\mu_1+\mu_2\big),
\end{align}
where $\Omega(\xi_1,\mu_1,\xi_2,\mu_2)$ is the resonant function defined in \eqref{resonant}. Now, for $i=1,2,3$ we define $F_i(\xi,\mu)$ as follows \[
F_i(\xi,\mu):=\left(\int_\R f_i^\#(\theta,\xi,\mu)^2d\theta\right)^{1/2}.
\]
Therefore, by using the Cauchy-Schwarz inequality along with Young inequalities in the $\theta$ variable, we obtain that \begin{align*}
I&\lesssim \int_{\R^4} \big\Vert f_1^\#(\cdot,\xi_1,\mu_1)\big\Vert_{L^1_\theta} F_2(\xi_2,\mu_2)F_3(\xi_1+\xi_2,\mu_1+\mu_2)d\xi_1d\mu_1d\xi_2d\mu_2
\\ & \lesssim  L_\mathrm{min}^{1/2} \int_{\R^4} F_1(\xi_1,\mu_1) F_2(\xi_2,\mu_2)F_3(\xi_1+\xi_2,\mu_1+\mu_2)d\xi_1d\mu_1d\xi_2d\mu_2.
\end{align*}
Then, estimate \eqref{general_case_bilinear} follows from the latter inequality by applying the same argument in the $\xi$ and $\mu$ variables, recalling that $\Vert \mathds{1}_{H_\mathrm{min}}(\xi,\mu)\Vert_{L^2}\sim H_\mathrm{min}^{\frac{1}{4}+\frac{1}{4}}$.

\medskip

We now seek to prove  inequality \eqref{hmin_ll_hmax_aoc}. Once again, in view of \eqref{definition_I_bilineal}, without loss of generality we assume that $L_\mathrm{max}=L_3$ and $H_\mathrm{min}=H_2$. Notice that, it suffices to prove that, if $g_i:\R^2\to\R_+$ are $L^2$ functions supported in $\Delta_{H_i}$ for $i=1,2$ and $g:\R^3\to\R_+$ is an $L^2$ function supported in $D^\#_{\infty,H_3,L_3}$, then \begin{align}\label{defi_J_bilineal}
J(g_1,g_2,g):=\int_{\R^4} g_1(\xi_1,\mu_1)g_2(\xi_2,\mu_2)g\Big(\Omega(\xi_1,\mu_1,\xi_2,\mu_2),\xi_1+\xi_2,\mu_1+\mu_2\Big)
\end{align}
satisfies that \begin{align}\label{claim_bilinear_1}
J(g_1,g_2,g)\lesssim H^{-1/2}_1H^{1/4}_2\Vert g_1\Vert_{L^2}\Vert g_2\Vert_{L^2}\Vert g\Vert_{L^2}.
\end{align}
Indeed, if estimate \eqref{claim_bilinear_1} holds, then we can define $g_i(\xi_i,\mu_i):=f_i^\#(\theta_i,\xi_i,\mu_i)$, $i=1,2$, and $g(\Omega,\xi,\mu):=f_3^\#(\theta_1+\theta_2+\Omega,\xi,\mu)$, for $\theta_1$ and $\theta_2$ fixed. Hence, by applying \eqref{claim_bilinear_1} along with Cauchy-Schwarz inequality to \eqref{definition_I_modi_bilinear} we would deduce that \begin{align*}
I&\lesssim H_1^{-1/2}H_2^{1/4}\Vert f_3^\#\Vert_{L^2}\int_{\R^2}\Vert f_1^\#(\theta_1,\cdot,\cdot)\Vert_{L^2_{\xi,\mu}}\Vert f_2^\#(\theta_2,\cdot,\cdot)\Vert_{L^2_{\xi,\mu}}d\theta_1 d\theta_2
\\ & \lesssim H_1^{-1/2}H_2^{1/4}L_1^{1/2}L_2^{1/2} \Vert f_1^\#\Vert_{L^2} \Vert f_2^\#\Vert_{L^2} \Vert f_3^\#\Vert_{L^2},
\end{align*}
which is exactly estimate \eqref{hmin_ll_hmax_aoc}. Now, to prove estimate \eqref{claim_bilinear_1}, we begin by applying twice the Cauchy-Schwarz inequality to obtain that \[
J(g_1,g_2,g) \leq \Vert g_1\Vert_{L^2}\Vert g_2\Vert_{L^2}\left(\int_{\Delta_{H_1}\times \Delta_{H_2}} g^2\Big(\Omega(\xi_1,\mu_1,\xi_2,\mu_2),\xi_1+\xi_2,\mu_1+\mu_2\Big)\right)^{1/2}
\]
We make the following change of variables $(\xi_1',\mu_1',\xi_2',\mu_2')=(\xi_1+\xi_2,\mu_1+\mu_2,\xi_2,\mu_2)$, so that \[
J(g_1,g_2,g) \leq \Vert g_1\Vert_{L^2}\Vert g_2\Vert_{L^2}\left(\int_{\Delta_{\sim H_1}\times \Delta_{H_2}} g^2\Big(\Omega(\xi_1'-\xi_2',\mu_1'-\mu_2',\xi_2',\mu_2'),\xi_1',\mu_1'\Big)\right)^{1/2}.
\]
Finally observe that \[
\left\vert\dfrac{\partial}{\partial\xi_2'}\Omega(\xi_1'-\xi_2',\mu_1'-\mu_2',\xi_2',\mu_2')\right\vert=\big\vert h(\xi_1'-\xi_2',\mu_1'-\mu_2')-h(\xi_2',\mu_2')\big\vert\sim H_\mathrm{max}.
\]
Thus, performing the change of variables $(\xi_1'',\mu_1'',\xi_2'',\mu_2'')=\big(\xi_1',\mu_1',\Omega(\xi_1'-\xi_2',\mu_1'-\mu_2',\xi_2',\mu_2'),\mu_2'\big)$ we conclude that \[
J(g_1,g_2,g)\lesssim H_\mathrm{max}^{-1/2}\Vert g_1\Vert_{L^2}\Vert g_2\Vert_{L^2}\left(\int_{\R^3\times[-cH_2^{1/2},cH_2^{1/2}]}g^2(\xi_2'',\xi_1'',\mu_1'')d\xi_1''d\xi_2''d\mu_1''d\mu_2''\right)^{1/2},
\]
for some harmless constant $c\sim1$, which leads us to \eqref{claim_bilinear_1} after integrating in $\mu_2$. Notice that inequality \eqref{hmin_ll_hmax_bc} follows the same lines, however, in that case we assume that $L_3=L_\mathrm{med}$ and $L_2=L_\mathrm{max}$ (and hence we also have that $H_2=H_\mathrm{min}$).

\medskip

It only remains to prove estimate \eqref{hmin_sim_hmax_bilinear}.  As before, our starting point is identity \eqref{definition_I_modi_bilinear}. For the sake of simplicity, from now on we denote by $\xi_3=\xi_1+\xi_2$ and $\mu_3=\mu_1+\mu_2$. Then, we decompose the integration domain into the following regions: \begin{align*}
\mathrm{R}_1&:=\big\{(\xi_1,\mu_1,\xi_2,\mu_2)\in\R^4:\, \mu_i^2\ll\vert \xi_i\vert^2 \hbox{ for all } i\in\{1,2,3\}\big\},
\\ \mathrm{R}_2&:=\big\{(\xi_1,\mu_1,\xi_2,\mu_2)\in\R^4:\, \mathrm{min}_{i\in\{1,2,3\}}\vert\xi_i\mu_i\vert\ll\mathrm{max}_{i\in\{1,2,3\}}\vert \xi_i\mu_i\vert\big\}\setminus\mathrm{R}_1,
\\ \mathrm{R}_3&:=\R^4\setminus(\mathrm{R}_1\cup \mathrm{R}_2).
\end{align*}
We also define the quantities $I_i$ being the restriction of $I$, given by identity \eqref{definition_I_modi_bilinear}, to the domain $\mathrm{R}_i$, for $i=1,2,3$ respectively. We split the analysis into several cases.

\medskip

\textbf{Estimate for $I_1$:} In fact, first of all notice that, in view of the symmetry \eqref{definition_I_bilineal}, we can always assume that $L_3=L_\mathrm{max}$. On the other hand, gathering the hypothesis $H_\mathrm{min}\sim H_\mathrm{max}$ along with $\mu_i^2\ll \vert\xi_i\vert^2$ for all $i\in\{1,2,3\}$, we infer that we must also have that $N_\mathrm{min}\sim N_\mathrm{max}$. Therefore, in this region we obtain that \[
\vert \Omega(\xi_1,\mu_1,\xi_2,\mu_2)\vert=\big\vert \xi_3\vert\xi_3\vert^2-\xi_1\vert\xi_1\vert^2-\xi_2\vert\xi_2\vert^2+\xi_1\mu_2\mu_3+\xi_2\mu_1\mu_3+\xi_3\mu_1\mu_2\big\vert\sim N_\mathrm{max}^{2+1}.
\]
Then, proceeding exactly as for the proof of item $(a)$, we obtain that \begin{align}\label{general_bilinear_forN}
I_1\lesssim N_\mathrm{min}^{1/2}H_\mathrm{min}^{1/4}L_\mathrm{min}^{1/2}\Vert f_1^\#\Vert_{L^2}\Vert f_2^\#\Vert_{L^2}\Vert f_3^\#\Vert_{L^2}.
\end{align}
Thus, using that $L_\mathrm{max}\gtrsim N_\mathrm{max}^{2+1}$ and observing that $N_\mathrm{max}^{2}\sim H_\mathrm{max}$, we conclude that \[
I_1\lesssim H_\mathrm{min}^{-1/4}L_\mathrm{min}^{1/2}L_\mathrm{max}^{1/2}\Vert f_1\Vert_{L^2}\Vert f_2\Vert_{L^2}\Vert f_3\Vert_{L^2},
\]
which is acceptable for our purposes since $N_\mathrm{max}^{2/4}\sim H_\mathrm{min}^{1/4}$.

\smallskip

\textbf{Estimate for $I_2$:} In this case observe that, due to frequency localization, we must have that $N_\mathrm{max}\sim N_\mathrm{med}$ and $\max_{1\leq j\leq 3}\vert \mu_j\vert \sim \mathrm{med}_{1\leq j\leq3}\vert\mu_j\vert$. From these relations it follows that \[
\max_{1\leq j\leq 3}\vert \xi_j\mu_j\vert \sim \max_{1\leq j \leq 3}\vert \xi_j\vert \max_{1\leq j\leq 3}\vert\mu_j\vert.
\]
Moreover, due to the definition of $\mathrm{R}_2$, we also have that there exists $i\in\{1,2,3\}$ such that $ \xi_i^2\lesssim\mu_i^2$. We infer that for any $j\in\{1,2,3\}$ it holds that $ \xi_j^2\lesssim H_j\sim H_i\sim \mu_i^2$, and hence we conclude that $N_\mathrm{max}^2\lesssim \max_{1\leq j\leq 3}\mu_j^2$. Moreover, from the symmetry properties of $I$ in \eqref{definition_I_bilineal} we can always assume that $\min_{1\leq j\leq 3}\vert\xi_j\mu_j\vert=\vert \xi_1\mu_1\vert$ and $\max_{1\leq j\leq 3}\vert \xi_j\mu_j\vert = \vert\xi_2\mu_2\vert$. Therefore, performing the change of variables $(\xi_1',\mu_1',\xi_2',\mu_2')=(\xi_1+\xi_2,\mu_1+\mu_2,\xi_2,\mu_2)$ we obtain that \[
\left\vert\dfrac{\partial}{\partial \mu_2'}\Omega\big(\xi_1'-\xi_2',\mu_1'-\mu_2',\xi_2',\mu_2'\big)\right\vert =2\big\vert \xi_1\mu_1-\xi_2\mu_2\big\vert \gtrsim N_{\mathrm{max}}H_\mathrm{max}^{1/2}.
\]
Thus, changing variables once again $(\xi_1'',\mu_1'',\xi_2'',\mu_2'')=(\xi_1',\mu_1',\xi_2',\Omega(\xi_1'-\xi_2',\mu_1'-\mu_2',\xi_2',\mu_2'))$ we infer that \[
J_2(g_1,g_2,g)\lesssim N_\mathrm{max}^{-1/2}H_\mathrm{max}^{-1/4}\Vert g_1\Vert_{L^2}\Vert g_2\Vert_{L^2}\left(\int_{\R^2\times I_{N_2}\times \R}g^2(\mu_2,\xi_1,\mu_1)%
d\xi_1d\mu_1d\xi_2d\mu_2\right)^{1/2},
\]
where $J_2$ is the restriction of the integral $J$ defined in \eqref{defi_J_bilineal} to the domain $\mathrm{R}_2$. Then, applying Cauchy-Schwarz inequality, the latter estimate leads us to \[
I_2\lesssim H_\mathrm{max}^{-1/4}L_\mathrm{med}^{1/2}L_\mathrm{max}^{1/2}\Vert f_1\Vert_{L^2}\Vert f_2\Vert_{L^2}\Vert f_3\Vert_{L^2},
\]
which is compatible with \eqref{hmin_sim_hmax_bilinear} since $N_\mathrm{max}^{1/2}\lesssim H_\mathrm{max}^{1/4}\sim H_\mathrm{min}^{1/4}$.

\smallskip

\textbf{Estimate for $I_3$:} First of all notice that, due to frequency localization and the definition of $\mathrm{R}_3$ we have that $
N_\mathrm{med}\sim N_\mathrm{max}$, 
$%
 \mathrm{med}_{1\leq j\leq3}\mu_j^2\sim\max_{1\leq j\leq 3}\mu_j^2$,
$%
\min_{1\leq j\leq 3}\vert \xi_j\mu_j\vert \sim \max_{1\leq j \leq 3}\vert \xi_j\mu_j\vert
$
and $\xi_i^2\lesssim \mu_i^2$ for some $i\in\{1,2,3\}$, which, along with $H_\mathrm{min}\sim H_\mathrm{max}$, implies that \[
N_\mathrm{min}^2\sim N_\mathrm{max}^2 \lesssim \min_{1\leq j\leq 3}\mu_j^2\sim \max_{1\leq j\leq 3} \mu_j^2.
\]
Now we split the integration domain into several subregions as follows \begin{align*}
\mathrm{R}_3^1&:=\big\{(\xi_1,\mu_1,\xi_2,\mu_2)\in\mathrm{R}_3: (\xi_1,\mu_1),(\xi_2,\mu_2)\in A_\delta\big\},
%
%
%
%
\\ \mathrm{R}_3^2&:=\big\{(\xi_1,\mu_1,\xi_2,\mu_2)\in\mathrm{R}_3: (\xi_1,\mu_1),(\xi_3,\mu_3)\in A_\delta\big\},
%
%
%
%
\\ \mathrm{R}_3^3&:=\big\{(\xi_1,\mu_1,\xi_2,\mu_2)\in\mathrm{R}_3: (\xi_2,\mu_2),(\xi_3,\mu_3)\in A_\delta\big\},
%
%
%
%
\\ \mathrm{R}_3^4&:=\mathrm{R}_3\setminus \cup_{i=1}^3\mathrm{R}_3^i,
\end{align*}
where $0<\delta\ll1$ denotes a sufficiently small number to be chosen. Then, we denote by $I_3^j$ the restriction of the integral $I_3$ to the domain $\mathrm{R}^j_3$. As we shall see, all the previous integrals $I_3^j$, but $I_3^4$, follows exactly the same proof, and hence we focus only on $I_3^1$.

\smallskip

\textit{Case $\{\xi_1\xi_2>0$ and $\mu_1\mu_2>0\}$}: Let us denote by $\mathrm{R}_3^{1,1}$ the integration domain $\mathrm{R}_3^1$ under the additional constraints $\xi_1\xi_2>0$ and $\mu_1\mu_2>0$. Then, it is enough to notice that, in this region, we have that \[
L_\mathrm{max}\gtrsim \big\vert\Omega(\xi_1,\mu_1,\xi_2,\mu_2)\big\vert\gtrsim N_\mathrm{max}H_\mathrm{max}.
\]
Therefore, once again, it follows arguing exactly as for \eqref{general_bilinear_forN} that \[
I_3^{1,1} \lesssim H_\mathrm{max}^{-1/4}L_\mathrm{min}^{1/2}L_\mathrm{max}^{1/2}\Vert f_1\Vert_{L^2}\Vert f_2\Vert_{L^2}\Vert f_3\Vert_{L^2},
\]
where we have denoted by $I_3^{1,1}$ the restriction of $I_3^1$ to the integration domain $\mathrm{R}_3^{1,1}$.

\smallskip

\textit{Case $\{\xi_1\xi_2>0$ and $\mu_1\mu_2<0\}$ or $\{\xi_1\xi_2<0$ and $\mu_1\mu_2>0\}$}: Let us denote by $\mathrm{R}_3^{1,2}$ the integration domain $\mathrm{R}_3^1$ under the additional current constraint, and by $I_3^{1,2}$ the restriction of $I_3^1$ to the domain $\mathrm{R}_3^{1,2}$. Then, performing the change of variables $(\xi_1',\mu_1',\xi_2',\mu_2')=(\xi_1+\xi_2,\mu_1+\mu_2,\xi_2,\mu_2)$ and then noticing that \[
\left \vert \dfrac{\partial}{\partial \mu_2'}\Omega(\xi_1'-\xi_2',\mu_1'-\mu_2',\xi_2',\mu_2')\right\vert =2\vert \xi_1\mu_1-\xi_2\mu_2\vert \gtrsim N_\mathrm{max}H_\mathrm{max}^{1/2}.
\]
Thus, arguing exactly as in the proof of estimate \eqref{claim_bilinear_1}, we infer that \[
J_3^{1,2}(g_1,g_2,g)\lesssim N_\mathrm{max}^{-1/2}H_\mathrm{max}^{-1/4}\Vert g_1\Vert_{L^2}\Vert g_2\Vert_{L^2}\left(\int_{\R^2\times I_{N_2}\times \R}g^2(\mu_2,\xi_1,\mu_1)d\xi_1d\mu_1d\xi_2d\mu_2\right)^{1/2},
\]
where we have denoted by $J_3^{1,2}$ the restriction of $J$ defined by \eqref{defi_J_bilineal} to the integration domain $\mathrm{R}_3^{1,2}$. Finally, observe that the latter estimate leads us to \[
I_3^{1,2}\lesssim H_\mathrm{max}^{-1/4}L_\mathrm{min}^{1/2}L_\mathrm{med}^{1/2}\Vert f_1\Vert_{L^2}\Vert f_2\Vert_{L^2}\Vert f_3\Vert_{L^2}.
\]

\smallskip

\textit{Case $\{\xi_1\xi_2<0$ and $\mu_1\mu_2<0\}$}: Let us denote by $\mathrm{R}_3^{1,3}$ the integration domain $\mathrm{R}_3^1$ restricted to the additional constraint $\{\xi_1\xi_2<0 \,\hbox{ and } \, \mu_1\mu_2<0\}$, and by $I_3^{1,3}$ the corresponding restriction of $I_3^1$ to the integration domain $\mathrm{R}_3^{1,3}$. Now we claim that, due to frequency localization, along with the fact that $H_\mathrm{min}\sim H_\mathrm{max}$ we infer that, there exists a sufficiently small constant $0<\gamma\ll 1$ such that \begin{align}\label{hh_gtr_max}
\big\vert h(\xi_1,\mu_1)-h(\xi_2,\mu_2)\big\vert \geq \gamma \max\big\{h(\xi_1,\mu_1),h(\xi_2,\mu_2)\big\},
\end{align}
for all $(\xi_1,\mu_1,\xi_2,\mu_2)\in \mathrm{R}_1^{1,3}$. In fact, by contradiction, if the latter inequality does not hold for any $\gamma\in (0,10^{-4})$, then Lemma \ref{lem_bil_incurve} with $f(\delta)=10^{-4}$ would imply that \[
\big\vert h(\xi_3,\mu_3)\big\vert \leq 10^{-3}\mathrm{max}\big\{h(\xi_1,\mu_1),h(\xi_2,\mu_2)\big\}.
\]
However, this contradicts the fact that $H_\mathrm{min}\sim H_\mathrm{max}$, and hence inequality \eqref{hh_gtr_max} holds. Thus, we can now proceed as for estimate \eqref{claim_bilinear_1}, performing the change of variables $(\xi_1',\mu_1',\xi_2',\mu_2'):=(\xi_1+\xi_2,\mu_1+\mu_2,\xi_2,\mu_2)$. In fact, it is enough to notice that, from \eqref{hh_gtr_max} we get that \[
\left\vert \dfrac{\partial}{\partial\xi_2'}\Omega(\xi_1'-\xi_2',\mu_1'-\mu_2',\xi_2',\mu_2')\right\vert=\big\vert h(\xi_1,\mu_1)-h(\xi_2,\mu_2)\big\vert\gtrsim H_\mathrm{max}.
\]
Then, following the same arguments as in the proof of \eqref{claim_bilinear_1} we are lead to \[
I_3^{1,3}\lesssim H_\mathrm{max}^{-1/2}H_\mathrm{min}^{1/4}L_\mathrm{min}^{1/2}L_\mathrm{med}^{1/2}\Vert f_1\Vert_{L^2}\Vert f_2\Vert_{L^2}\Vert f_3\Vert_{L^2},
\]
which concludes the proof for $I_3^{1,3}$.

\medskip

Therefore, gathering the above three cases we conclude the estimate for $I_3^1$. As we mentioned before, notice that $I_3^2$ and $I_3^3$ can be bounded in the exact same fashion, and hence it only remains to control $I_3^4$.

\smallskip

\textbf{Contribution of $I_3^4$:} In this case we take advantage of the improved Strichartz estimates derived in Proposition  \ref{Str_ZK_prop_final}. In fact, it is not difficult to see that, as a direct consequence of this corollary we have that, if $\mathrm{supp}\,f\subseteq D_{N,H,L}$, then \begin{align}\label{consequence_cor_stric}
\Vert P_{A_\delta^c}\mathcal{F}^{-1}(f)\Vert_{L^4}\lesssim N^{-1/4}L^{1/2}\Vert f\Vert_{L^2}.
\end{align}
Now, without loss of generality let us assume that $(\xi_1,\mu_1),(\xi_2,\mu_2)\in \R^2\setminus A_\delta$. Then, notice that H\"older inequality along with Plancharel identity lead us to \[
I_3^4\lesssim \Vert f_3\Vert_{L^2}\big\Vert \big(\mathds{1}_{\R^2\setminus A_\delta} f_1\big)*\big(\mathds{1}_{\R^2\setminus A_\delta} f_2\big)\big\Vert_{L^2}\lesssim \Vert f_3\Vert_{L^2}\Vert P_{A_\delta^c}\mathcal{F}^{-1}(f_1)\Vert_{L^4}\Vert P_{A_\delta^c}\mathcal{F}^{-1}(f_2)\Vert_{L^4}.
\]
Therefore, by plugging estimate \eqref{consequence_cor_stric}  into the right-hand side of the latter inequality we conclude that \[
I_3^4\lesssim N_\mathrm{max}^{-1/2}L_\mathrm{med}^{1/2}L_\mathrm{max}^{1/2}\Vert f_1\Vert_{L^2} \Vert f_2\Vert_{L^2} \Vert f_3\Vert_{L^2},
\]
which is an acceptable estimate since $N_\mathrm{max}^{1/2}\lesssim H_\mathrm{min}^{1/4}$. The proof is complete.
\end{proof}

To finish this subsection we restate the previous lemma in a form that is suitable for the bilinear estimates in the next section.

\begin{cor}\label{cor_bilinear_l2}
For $i=1,2,3$, let $N_i,H_i,L_i\in\mathbb{D}$ and assume that $f_i:\R^3\to \R_+$ are $L^2(\R^3)$ functions supported in $D_{\infty,H_i,L_i}$. Then, the following inequality holds \begin{align}\label{cor_bilinear_l2_general}
\big\Vert \mathds{1}_{D_{\infty,H_3,L_3}}\cdot(f_1*f_2)\big\Vert_{L^2}\lesssim H_{\mathrm{min}}^{1/2}L_\mathrm{min}^{1/2}\Vert f_1\Vert_{L^2}\Vert f_2\Vert_{L^2}.
\end{align}
Moreover, we have the following cases: \begin{enumerate}
\item[1.] Assume that $H_\mathrm{min}\ll H_\mathrm{max}$. If there exists $i\in\{1,2,3\}$ such that $(H_i,L_i)=(H_\mathrm{min},L_\mathrm{max})$, then it follows that \begin{align}\label{cor_l2bl_hminllhmax_particular}
\big\Vert \mathds{1}_{D_{\infty,H_3,L_3}}\cdot(f_1*f_2)\big\Vert_{L^2}\lesssim H_\mathrm{max}^{-1/2}H_\mathrm{min}^{1/4}L_\mathrm{min}^{1/2}L_\mathrm{max}^{1/2}\Vert f_1\Vert_{L^2}\Vert f_2\Vert_{L^2}.%
\end{align}
Otherwise, we have that \begin{align}\label{cor_l2bl_hminllhmax_general}
\big\Vert \mathds{1}_{D_{\infty,H_3,L_3}}\cdot(f_1*f_2)\big\Vert_{L^2}\lesssim H_\mathrm{max}^{-1/2}H_\mathrm{min}^{1/4}L_\mathrm{min}^{1/2}L_\mathrm{med}^{1/2}\Vert f_1\Vert_{L^2}\Vert f_2\Vert_{L^2}.
\end{align}
\item[2.] Assume that $H_\mathrm{min}\sim H_\mathrm{max}$, and suppose additionally that $f_i$ are supported in $D_{N_i,H_i,L_i}$ for $i=1,2,3$. Then, the following bound holds  \begin{align}\label{cor_l2bl_minsimmax}
\big\Vert \mathds{1}_{D_{N_3,H_3,L_3}}\cdot(f_1*f_2)\big\Vert_{L^2}\lesssim N_\mathrm{max}^{-1} H_\mathrm{min}^{1/4}L_\mathrm{med}^{1/2}L_\mathrm{max}^{1/2}\Vert f_1\Vert_{L^2}\Vert f_2\Vert_{L^2}.
\end{align}
\end{enumerate}
\end{cor}

\begin{proof}
The proof follows directly from Lemma \ref{lemma_bilineal} by using a duality
argument.
\end{proof}

\smallskip

\section{Short time bilinear estimates}

The main results of this section are the following bilinear estimates in the $F^s_\beta(T)$ spaces. Note that, in the classical case, namely $\Psi\equiv 0$, to overcome the high-low frequency interaction problem, the ``optimal'' choice of $\beta$ would be $\beta_0=0$ (we do not need to chop the time interval in this case). It is worth to notice that the method for $\beta_0=0$ corresponds to a fixed point argument in classical Bourgain spaces. However, due to the presence of $\Psi$, we shall need to fix $\beta=1/2\geq\beta_0$ to control terms of the form $\Vert P_H\partial_x(P_{\sim H}u \cdot P_{\ll H}\Psi)\Vert_{F^s_\beta(T)}$. The price to pay for this ``non optimal'' choice of $\beta$ is not being able to go beyond the energy space $H^1(\R^2)$ in our local well-posedness theorem, while in the classical case (without the background function $\Psi$), by choosing $\beta=0$, the present method would allow us to obtain LWP in $H^{1/2^+}(\R^2)$. 

\medskip

Since $\beta=1/2$ is fixed now, in the rest of the paper we denote $F^s_\beta(T)$, $N^s_\beta(t)$, $F^s_\beta$, $N^s_\beta$, $F_{H,\beta}$ and $N_{H,\beta}$ simply by $F^s(T)$, $N^s(t)$, $F^s$, $N^s$, $F_{H}$ and $N_{H}$ respectively. The main results of this section are the following bilinear estimates at the $H^s$ and $L^2$ level. We refer to \cite{KePi,RiVe} for similar bilinear estimates in different contexts.
\begin{prop}\label{shortime_bilinear_estimate}
Let $s\geq 1$ and $T\in (0,1]$ both be fixed. Then, the following holds
\begin{align}\label{short_t_estimate}
\big\Vert \partial_x(uv)\big\Vert_{N^{s}(T)} \lesssim \Vert u\Vert_{F^{0^+}(T)}\Vert v\Vert_{F^{s}(T)}+\Vert u\Vert_{F^{s}(T)}\Vert v\Vert_{F^{0^+}(T)}.
\end{align}
Moreover, for any $s>0$, if $u\in F^0(T)$ and $v\in F^s(T)$, then it holds that \begin{align}\label{short_t_2_estimate}
\Vert \partial_x(uv)\Vert_{N^0(T)}\lesssim \Vert u\Vert_{F^0(T)}\Vert v\Vert_{F^s(T)}.
\end{align}
\end{prop}
In order to prove the above proposition, we split the analysis into several technical lemmas.
\begin{lem}[high$\,\times\,$low$\,\rightarrow\,$high]\label{bilinear_hlh}
Let $H,H_1,H_2\in\mathbb{D}$ satisfying that $H\sim H_2\gg H_1$. Then, the following inequality holds \begin{align}\label{bilinear_hlh_estimate}
\big\Vert P_H\partial_x(u_{H_1}v_{H_2})\big\Vert_{N_H}\lesssim H^{-1/4}H_1^{1/4}\Vert u_{H_1}\Vert_{F_{H_1}}\Vert v_{H_2}\Vert_{F_{H_2}},
\end{align}
for all $u_{H_1}\in F_{H_1}$ and $v_{H_2}\in F_{H_2}$.
\end{lem}

\begin{proof}
In fact, first of all, notice that expanding the $N_H$ norm we have that \begin{align}\label{hlh_NH_norm_proof}
\big\Vert P_H\partial_x(u_{H_1}v_{H_2})\big\Vert_{N_H}\lesssim \sup_{t_H\in\R}\big\Vert (\tau-\omega(\xi,\mu)+iH^{1/2})^{-1}H^{1/2}\mathds{1}_{\Delta_H}\cdot(f_{H_1}*f_{H_2})\big\Vert_{X_H},
\end{align}
where $f_{H_1}$ and $f_{H_2}$ are defined by \[
f_{H_1}:=\big\vert \mathcal{F}\big(\eta_0(H^{1/2}(\cdot-t_H))u_{H_1}\big)\big\vert \quad \hbox{ and } \quad f_{H_2}:=\big\vert \mathcal{F}\big(\widetilde\eta_0(H^{1/2}(\cdot-t_H))v_{H_2}\big)\big\vert.
\]
For the sake of notation now we set \begin{align}\label{def_fhl_l2}
f_{H_i,\lfloor H^{1/2}\rfloor}(\tau,\xi,\mu)&:=\eta_{\leq \lfloor H^{1/2}\rfloor}(\tau-\omega(\xi,\mu))f_{H_i}(\tau,\xi,\mu), \quad \hbox{and}  \\ f_{H_i,L_i}(\tau,\xi,\mu)&:=\eta_{L_i}(\tau-\omega(\xi,\mu))f_{H_i}(\tau,\xi,\mu), \quad \hbox{for } \ L_i>\lfloor H^{1/2}\rfloor.\nonumber
\end{align}
Therefore, expanding the $X_H$-norm and then decomposing into frequencies, from \eqref{hlh_NH_norm_proof} we deduce that \begin{align}\label{hlh_proof_rewritten}
\big\Vert P_H\partial_x(u_{H_1}v_{H_2})\big\Vert_{N_H}\lesssim \sup_{t_H\in\R}H^{1/2}\sum_{L,L_1,L_2\geq  H^{1/2}} L^{-1/2}\Vert \mathds{1}_{D_{\infty,H,L}}\cdot(f_{H_1,L_1}*f_{H_2,L_2})\Vert_{L^2}.
\end{align}
Notice that, in the latter estimate, to bound the terms corresponding to $L< H^{1/2}$ appearing implicitly on the right-hand side of \eqref{hlh_NH_norm_proof}, we have used the fact that $\big\vert(\tau-\omega(\xi,\mu)+iH^{1/2})^{-1}\big\vert\lesssim H^{-{1/2}}$ to control them by the term corresponding to $L=\lfloor H^\beta\rfloor$ on the right-hand side of \eqref{hlh_proof_rewritten}. Thus, according to Corollary \ref{key_cor_prelim} and estimate \eqref{hlh_proof_rewritten}, we infer that it suffices to prove that \begin{align}\label{claim_bilinear_hlh}
&H^{1/2} \sum_{L\geq  H ^{1/2}} L^{-1/2} \big\Vert \mathds{1}_{\infty,H,L}\cdot(f_{H_1,L_1}*f_{H_2,L_2})\big\Vert_{L^2} \nonumber
\\ & \qquad \qquad \qquad \qquad \lesssim H^{-1/4}H_1^{1/4}L_1^{1/2}\Vert f_{H_1,L_1}\Vert_{L^2}L_2^{1/2}\Vert f_{H_2,L_2}\Vert_{L^2},
\end{align}
for all $L_1,L_2\geq H^{1/2}$, in order to prove estimate \eqref{bilinear_hhh_estimate}. On the other hand, by using \eqref{cor_l2bl_hminllhmax_particular} and \eqref{cor_l2bl_hminllhmax_general} we obtain that \begin{align*}
&H^{1/2} \sum_{L\geq  H ^{1/2}} L^{-1/2} \big\Vert \mathds{1}_{\infty,H,L}\cdot(f_{H_1,L_1}*f_{H_2,L_2})\big\Vert_{L^2}
\\ & \qquad \qquad \qquad \qquad \lesssim H^{1/2}\sum_{L\geq H^{1/2}}L^{-1/2}H^{-1/2}H_1^{1/4}L_1^{1/2}\Vert f_{H_1,L_1}\Vert_{L^2} L_2^{1/2}\Vert f_{H_2,L_2}\Vert_{L^2},
\end{align*}
which implies estimate \eqref{claim_bilinear_hlh} after summing over $L$. The proof is complete.
\end{proof}

\begin{lem}[high$\,\times\,$high$\,\rightarrow\,$high]\label{bilinear_hhh}
Let $H,H_1,H_2\in\mathbb{D}$ satisfying that $H\sim H_1\sim H_2\gg 1$. Then, it holds that \begin{align}\label{bilinear_hhh_estimate}
\big\Vert P_H\partial_x(u_{H_1}v_{H_2})\big\Vert_{N_H}\lesssim H^{0^+}\Vert u_{H_1}\Vert_{F_{H_1}}\Vert v_{H_2}\Vert_{F_{H_2}},
\end{align}
for all $u_{H_1}\in F_{H_1}$ and $v_{H_2}\in F_{H_2}$.
\end{lem}

\begin{proof}
In fact, arguing as in the proof of the previous lemma, along with an additional decomposition into frequencies in the $\xi$ variable, we infer that it is enough to prove that
\begin{align}\label{claim_bilinear_hhh}
&N\sum_{L\geq  H ^{1/2}} L^{-1/2} \big\Vert \mathds{1}_{N,H,L}\cdot(f_{N_1,H_1,L_1}*f_{N_2,H_2,L_2})\big\Vert_{L^2}\nonumber
\\ & \qquad \qquad \qquad \qquad \lesssim H^{0^+} L_1^{1/2}\Vert f_{N_1,H_1,L_1}\Vert_{L^2}L_2^{1/2}\Vert f_{N_2,H_2,L_2}\Vert_{L^2},
\end{align}
where $f_{N_i,H_i,L_i}$ is localized in $D_{N_i,H_i,L_i}$ and $L_1,L_2\geq H^{1/2}$. Notice that the decomposition in $N,N_1,N_2$ is harmless in the sense that, once we prove the above estimate, we can control $f_{N_i,H_i,L_i}$ with $f_{H_i,L_i}$ since the sums over $N,\,N_1$ and $N_2$ are controlled by $\log(H^{1/2})\lesssim H^{0^+}$. Thus, we can directly sum over $N,$ $N_1$ and $N_2$. Now we split the analysis into two cases. First we assume that either $L=L_\mathrm{min}$ or $L_\mathrm{med}\sim L_\mathrm{max}$. Then, by using estimate  \eqref{cor_l2bl_minsimmax}, we immediately obtain that \begin{align*}
&N\sum_{L\geq  H ^{1/2}} L^{-1/2} \big\Vert \mathds{1}_{N,H,L}\cdot(f_{N_1,H_1,L_1}*f_{N_2,H_2,L_2})\big\Vert_{L^2}\nonumber
\\ & \qquad \qquad \qquad \qquad \lesssim N \sum_{L\geq H^{1/2}}L ^{-1/2}N_\mathrm{max}^{-1}H^{1/4}L_\mathrm{med}^{1/2}L_\mathrm{max}^{1/2}\Vert f_{N_1,H_1,L_1}\Vert_{L^2} \Vert f_{N_2,H_2,L_2}\Vert_{L^2}\nonumber
\\ & \qquad \qquad \qquad \qquad \lesssim L_1^{1/2}\Vert f_{N_1,H_1,L_1}\Vert_{L^2}L_2^{1/2}\Vert f_{N_2,H_2,L_2}\Vert_{L^2}. 
\end{align*}
On the other hand, if neither of the previous cases holds, then $L_\mathrm{max}\sim\vert\Omega\vert\lesssim H^{3/2}$, and hence the sum over $L$ is bounded by $H^{0^+}$. Therefore, estimate \eqref{claim_bilinear_hhh} still holds, which concludes the proof of the lemma.
\end{proof}

\begin{lem}[high$\,\times\,$high$\,\rightarrow\,$low]\label{bilinear_hhl}
Let $H,H_1,H_2\in\mathbb{D}$ satisfying that $H_1\sim H_2\gg H$. Then, the following holds \begin{align}\label{bilinear_hhl_estimate}
\big\Vert P_H\partial_x(u_{H_1}v_{H_2})\big\Vert_{N_H}\lesssim H^{1/4}H_1^{(-1/4)^+}\Vert u_{H_1}\Vert_{F_{H_1}}\Vert v_{H_2}\Vert_{F_{H_2}},
\end{align}
for all $u_{H_1}\in F_{H_1}$ and $v_{H_2}\in F_{H_2}$.
\end{lem}

\begin{proof}
Indeed, first of all let us consider a smooth function $\gamma:\R\to[0,1]$ supported in $[-1,1]$ satisfying that \[
\sum_{m\in \Z}\gamma^2(x-m)=1, \quad \forall x\in\R.
\]
Then, expanding the $N_H$-norm in the right-hand side of \eqref{bilinear_hhl} and then using the definition of $\gamma$, we obtain that \begin{align}\label{expansion_NH_hhl_bilinear}
&\big\Vert P_H\partial_x(u_{H_1}v_{H_2})\big\Vert_{N_H}\nonumber
\\ & \qquad \qquad \lesssim \sup_{t_H\in\R}\big\Vert \big(\tau-\omega(\xi,\mu)+iH^{1/2}\big)^{-1}H^{1/2}\mathds{1}_{\Delta_H}\sum_{\vert m\vert\lesssim (H_1/H)^{1/2}}f^m_{H_1}*f_{H_2}^m\big\Vert_{X_H},
\end{align}
where $f^m_{H_1}$ and $f^m_{H_2}$ are defined by \begin{align*}
f^m_{H_1}&:=\left\vert \mathcal{F}\left( \eta_0\big(H^{1/2}(\cdot-t_H)\big)\gamma\big(H_1^{1/2}(\cdot-t_H)-m\big)u_{H_1}\right)\right\vert, 
\\ f^m_{H_2}&:=\left\vert \mathcal{F}\left( \widetilde\eta_0\big(H^{1/2}(\cdot-t_H)\big)\gamma\big(H_1^{1/2}(\cdot-t_H)-m\big)v_{H_2}\right)\right\vert.
\end{align*}
In the same fashion as before, for the sake of simplicity we define \begin{align*}
f^m_{H_i,\lfloor H_1^{1/2}\rfloor}(\tau,\xi,\mu)&:= \eta_{\leq \lfloor H_1^{1/2} \rfloor}(\tau-\omega(\xi,\mu))f^m_{H_i}(\tau,\xi,\mu), \quad \hbox{and}
\\ f^m_{H_i,L_i}(\tau,\xi,\mu)&:= \eta_{\leq L_i}(\tau-\omega(\xi,\mu))f^m_{H_i}(\tau,\xi,\mu), \quad \hbox{for} \quad L_i>\lfloor H_1^{1/2}\rfloor.
\end{align*}
From the above definitions, estimate \eqref{expansion_NH_hhl_bilinear}, the definition of the $X_H$-norm, arguing as in the proof of Lemma \ref{bilinear_hlh}, we infer that \begin{align*}
&\big\Vert P_H\partial_x(u_{H_1}v_{H_2})\big\Vert_{N_H}\nonumber
\\ & \qquad \qquad \lesssim H^{1/2}\sup_{t_H\in\R,\, m\in\Z}H_1^{1/2} H^{-{1/2}}\sum_{L\in\mathbb{D}}\sum_{L_1,L_2\geq H_1^{1/2}} L^{-1/2}\big\Vert \mathds{1}_{D_{\infty,H,L}}\cdot\big(f^m_{H_1,L_1}*f^m_{H_2,L_2}\big)\big\Vert_{L^2}.
\end{align*}
Therefore, according to Lemma  \ref{technical_estimate_lemma} and the latter estimate, it suffices to prove that \begin{align}\label{claim_hhl_lemma}
&H_1^{1/2} \sum_{L\in\mathbb{D}} L^{-1/2}\big\Vert \mathds{1}_{D_{\infty,H,L}}\cdot\big(f^m_{H_1,L_1}*f^m_{H_2,L_2}\big)\big\Vert_{L^2}\nonumber
\\ & \qquad \ \qquad \qquad \lesssim H_1^{(-1/4)^+}H^{1/4}L_1^{1/2}\Vert f_{H_1,L_1}^m\Vert_{L^2} L_2^{1/2}\Vert f_{H_2,L_2}^m\Vert_{L^2},
\end{align}
for all $L_1,L_2\geq H_1^{1/2}$, in order to prove estimate \eqref{bilinear_hhl_estimate}. Now, on the one-hand, if $L_\mathrm{max}=L_1$ or $L_\mathrm{max}=L_2$, say for example $L_\mathrm{max}=L_1$, then, we deduce from estimate  \eqref{cor_l2bl_hminllhmax_general}  that \begin{align*}
& H_1^{1/2} \sum_{L\in\mathbb{D}} L^{-1/2}\big\Vert \mathds{1}_{D_{\infty,H,L}}\cdot\big(f^m_{H_1,L_1}*f^m_{H_2,L_2}\big)\big\Vert_{L^2}
\\ & \qquad \qquad \lesssim H_1^{{1/2}-1/2}H^{1/4}\sum_{L\in\mathbb{D}} L^{-1/2}L^{1/2}\Vert f_{H_1,L_1}^m\Vert_{L^2} L_2^{1/2}\Vert f_{H_2,L_2}^m\Vert_{L^2}
\\ & \qquad \qquad \lesssim H_1^{(-1/4)^+}H^{1/4}L_1^{1/2}\Vert f_{H_1,L_1}^m\Vert_{L^2} L_2^{1/2}\Vert f_{H_2,L_2}^m\Vert_{L^2}%
\\ & \qquad \qquad \qquad +H_1^{1/2-1/2}H^{1/4}\sum_{L\geq H_1^{1/2}} L^{-1/2}L_1^{1/2}\Vert f_{H_1,L_1}^m\Vert_{L^2} L_2^{1/2}\Vert f_{H_2,L_2}^m\Vert_{L^2}.
\end{align*}
On the other hand, if $L_\mathrm{max}=L$ we have that $L\sim \max\{L_\mathrm{med},\vert\Omega\vert\}$, where $\Omega$ is defined in \eqref{resonant}. Notice that, if $L\sim L_\mathrm{med}$, then we are in one of the above cases, whereas in the case $L\sim \vert \Omega\vert$ we have that $L\sim \vert \Omega\vert\lesssim H_1^{3/2}$, and hence the sum over $L$ is bounded by $H_1^{0^+}$. Therefore, estimate \eqref{claim_hhl_lemma} still holds, which concludes the proof of the lemma.
\end{proof}

\begin{lem}[low$\,\times\,$low$\,\rightarrow\,$low]\label{bilinear_lll}
Let $H,H_1,H_2\in\mathbb{D}$ satisfying that $H,H_1, H_2\lesssim 1$. Then, the following inequality holds \begin{align}\label{bilinear_lll_estimate}
\big\Vert P_H\partial_x(u_{H_1}v_{H_2})\big\Vert_{N_H}\lesssim \Vert u_{H_1}\Vert_{F_{H_1}}\Vert v_{H_2}\Vert_{F_{H_2}},
\end{align}
for all $u_{H_1}\in F_{H_1}$ and $v_{H_2}\in F_{H_2}$.
\end{lem}

\begin{proof}
Once again, arguing as in the proof of the previous lemmas, we infer that it is enough to prove that \begin{align*}
\sum_{L\in\mathbb{D}}L^{-1/2}\big\Vert\mathds{1}_{D_{\infty,H,L}}\cdot(f_{H_1,L_1}*f_{H_2,L_2})\big\Vert_{L^2}\lesssim L_1^{1/2}\Vert f_{H_1,L_1}\Vert_{L^2}L_2^{1/2}\Vert f_{H_2,L_2}\Vert_{L^2}, 
\end{align*}
where $L_1,L_2\in\mathbb{D}$ are fixed dyadic numbers and $f_{H_i,L_i}$ are supported in $D_{\infty,H_i,L_i}$. However, this is a direct consequence of estimate \eqref{cor_bilinear_l2_general}  along with the fact that $H,H_1,H_2\lesssim 1$. The proof is complete.
\end{proof}

Finally, we are in position to prove Proposition \ref{shortime_bilinear_estimate}.

\begin{proof}[Proof of Proposition \ref{shortime_bilinear_estimate}]
First of all, since estimates
\eqref{short_t_estimate} and \eqref{short_t_2_estimate} follow very similar lines, we shall only prove \eqref{short_t_estimate} and we omit the proof of the second case. In fact, in order to take advantage of the previous estimates, we take two extensions $\tilde u$ and $\tilde v$, of $u$ and $v$ respectively, satisfying that \[
\Vert \tilde u\Vert_{F^s}\leq 2\Vert u\Vert_{F^s(T)} \quad \hbox{ and } \quad \Vert \tilde v\Vert_{F^s}\leq 2\Vert v\Vert_{F^s(T)}.
\]
On the other hand notice that, from the definition of the $N^s(T)$ norm in \eqref{NsT_norm}, after an application of Minkowski inequality we infer that \[
\big\Vert \partial_x(uv)\big\Vert_{N^s(T)}\lesssim \bigg(\sum_{H\in\mathbb{D}}H^{s}\Big(\sum_{H_1,H_2\in\mathbb{D}}\big\Vert\partial_xP_H(P_{H_1}\tilde{u}P_{H_2}\tilde{v})\big\Vert_{N_H}\Big)^2\bigg)^{1/2}.
\]
Now, for $H\in\mathbb{D}$ fixed, we split the analysis into several regions, namely \begin{align*}
\mathbf{H}_1&:=\big\{(H_1,H_2)\in\mathbb{D}^2: H\ll H_1\sim H_2\big\},
\\ \mathbf{H}_2&:=\big\{(H_1,H_2)\in\mathbb{D}^2: H_1\ll H_2\sim H\big\},
\\ \mathbf{H}_3&:=\big\{(H_1,H_2)\in\mathbb{D}^2: H_2\ll H_1\sim H\big\},
\\ \mathbf{H}_4&:=\big\{(H_1,H_2)\in\mathbb{D}^2: H\sim H_1\sim H_2\gg 1\big\},
\\ \mathbf{H}_5&:=\big\{(H_1,H_2)\in\mathbb{D}^2: H, H_1, H_2\lesssim 1\big\}.
\end{align*}
Therefore, due to the frequency localization we conclude that \begin{align*}
\big\Vert \partial_x(uv)\big\Vert_{N ^s(T)}&\lesssim \sum_{j=1}^5 \bigg(\sum_{H\in\mathbb{D}}H^{s}\Big(\sum_{(H_1,H_2)\in \mathbf{H}_j}\big\Vert\partial_xP_H(P_{H_1}\tilde{u}P_{H_2}\tilde{v})\big\Vert_{N_H}\Big)^2\bigg)^{1/2}=: \sum_{j=1}^5 \mathcal{N}_j.
\end{align*}
Then, in order to deal with $\mathcal{N}_1$ it suffices to use estimate \eqref{bilinear_hhl_estimate}, from where we get \begin{align*}
\mathcal{N}_1&\lesssim \bigg(\sum_{H\in\mathbb{D}}H^{s}\Big(\sum_{H_1\gg H}H_1^{0^+}\Vert P_{H_1}\tilde{u}\Vert_{F_{H_1}}  \Vert P_{H_1}\tilde{v}\Vert_{F_{H_1}}\Big)^2\bigg)^{1/2}\lesssim \Vert \tilde u\Vert_{F^{0^+}}\Vert \tilde v\Vert_{F^s}.
\end{align*}
In a similar fashion, by using estimate \eqref{bilinear_hlh_estimate} we obtain that \[
\mathcal{N}_2\lesssim \bigg(\sum_{H\in\mathbb{D}}H^{s}\Big(\sum_{H_1\ll H}\Vert P_{H_1}\tilde{u}\Vert_{F_{H_1}}  \Vert P_{H}\tilde{v}\Vert_{F_{H}}\Big)^2\bigg)^{1/2}\lesssim \Vert \tilde u\Vert_{F^{0^+}}\Vert \tilde v\Vert_{F^s}.
\]
Notice that, by symmetry, the exact same bound holds for $\mathcal{N}_3$, exchanging the roles of $\tilde u$ and $\tilde v$. Next, to handle $\mathcal{N}_4$ we use estimate \eqref{bilinear_hhh_estimate}, along with Cauchy-Schwarz, which leads us to \begin{align*}
\mathcal{N}_4\lesssim \left(\sum_{H\in\mathbb{D}}H^{s}H^{0^+}\Vert P_H\tilde u\Vert_{F_H}^2\Vert P_H\tilde v\Vert_{F_H}^2\right)^{1/2}\lesssim \Vert \tilde u\Vert_{F^{0^+}}\Vert \tilde v\Vert_{F^s}.
\end{align*}
Finally, by using
\eqref{bilinear_lll_estimate} we conclude that \[
\mathcal{N}_5\lesssim \Vert \tilde u\Vert_{F^0}\Vert \tilde v\Vert_{F^0}.
\]
Therefore, gathering the above estimates we conclude the proof of the proposition.
\end{proof}

\subsection{Short time estimates on the background of $\Psi$}

Here we state an alternative version of the lemmas in the previous subsection when one of the functions has no decay properties, but boundedness.

\begin{prop}\label{shortime_psi_bilinear_estimate}
Let $s\geq 0$ and $T\in (0,1]$ both be fixed. Then, for any $u\in F^s(T)$ and any $V\in L^\infty_tW^{1/2^+,\infty}_{xy}$, the following inequality hold
\begin{align}\label{short_t_psi_bil}
\big\Vert \partial_x(uV)\big\Vert_{N^s(T)} \lesssim \Vert u\Vert_{F^{s}(T)}\Vert V\Vert_{L^\infty_tW^{1/2^+,\infty}_{xy}}+\Vert u\Vert_{F^{0}(T)}\Vert V\Vert_{ L^\infty_t W^{s^+,\infty}_{xy}}.
\end{align}
\end{prop}

\begin{proof}
First of all, since estimates
\eqref{short_t_estimate} and \eqref{short_t_2_estimate} follow very similar lines, we shall only prove \eqref{short_t_estimate} and we omit the proof of the second case. In fact, in order to take advantage of the previous estimates, we take two extensions $\tilde u$ and $\tilde v$, of $u$ and $v$ respectively, satisfying that \[
\Vert \tilde u\Vert_{F^s}\leq 2\Vert u\Vert_{F^s(T)} \quad \hbox{ and } \quad \Vert \tilde v\Vert_{F^s}\leq 2\Vert v\Vert_{F^s(T)}.
\]
On the other hand notice that, from the definition of the $N^s(T)$ norm in \eqref{NsT_norm}, after an application of Minkowski inequality we infer that \[
\big\Vert \partial_x(uV)\big\Vert_{N^s(T)}\lesssim \bigg(\sum_{H\in\mathbb{D}}H^{s}\Big(\sum_{H_1,H_2\in\mathbb{D}}\big\Vert\partial_xP_H(P_{H_1}\tilde{u}P_{H_2}V)\big\Vert_{N_H}\Big)^2\bigg)^{1/2}.
\]
Now, for $H\in\mathbb{D}$ fixed, we split the analysis into several regions, namely \begin{align*}
\mathbf{H}_1&:=\big\{(H_1,H_2)\in\mathbb{D}^2: H_1\ll H_2\sim H\big\},
\\ \mathbf{H}_2&:=\big\{(H_1,H_2)\in\mathbb{D}^2: H_2\ll H_1\sim H\big\},
\\ \mathbf{H}_3&:=\big\{(H_1,H_2)\in\mathbb{D}^2: H\ll H_1\sim H_2\big\},
\\ \mathbf{H}_4&:=\big\{(H_1,H_2)\in\mathbb{D}^2: H\sim H_1\sim H_2\gg 1\big\},
\\ \mathbf{H}_5&:=\big\{(H_1,H_2)\in\mathbb{D}^2: H, H_1, H_2\lesssim 1\big\}.
\end{align*}
Therefore, due to the frequency localization we conclude that \begin{align*}
\big\Vert \partial_x(uv)\big\Vert_{N ^s(T)}&\lesssim \sum_{j=1}^5 \bigg(\sum_{H\in\mathbb{D}}H^{s}\Big(\sum_{(H_1,H_2)\in \mathbf{H}_j}\big\Vert\partial_xP_H(P_{H_1}\tilde{u}P_{H_2}V)\big\Vert_{N_H}\Big)^2\bigg)^{1/2}=: \sum_{j=1}^5 \mathcal{N}_j.
\end{align*}
Now, in the same fashion as before, except for the $\mathrm{high}\times \mathrm{high}\to\mathrm{low}$ case $\mathcal{N}_3$, we simply expand the $N_H$ norm, from where we obtain that, for each of these cases \begin{align}\label{hlh_NH_Psi_proof}
\big\Vert P_H\partial_x(\tilde u_{H_1}V_{H_2})\big\Vert_{N_H}\lesssim \sup_{t_H\in\R}\big\Vert (\tau-\omega(\xi,\mu)+iH^{1/2})^{-1}H^{1/2}\mathds{1}_{\Delta_H}\cdot(f_{H_1}*f_{H_2})\big\Vert_{X_H},
\end{align}
where $f_{H_1}$ and $f_{H_2}$ are defined by \[ 
f_{H_1}:= \mathcal{F}(\eta_0(H^{1/2}(\cdot-t_H))\tilde u_{H_1}) \quad \hbox{ and } \quad f_{H_2}:= \mathcal{F}(\widetilde{\eta}_0(H^{1/2}(\cdot-t_H))V_{H_2}).
\]
Notice that $f_{H_2}$ is well defined. For the sake of notation now we set \begin{align*}
f_{H_1,\lfloor H^{1/2}\rfloor}(\tau,\xi,\mu)&:=\eta_{\leq \lfloor H^{1/2}\rfloor}(\tau-\omega(\xi,\mu))f_{H_1}(\tau,\xi,\mu), \quad \hbox{and} \\ f_{H_1,L_1}(\tau,\xi,\mu)&:=\eta_{L_1}(\tau-\omega(\xi,\mu))f_{H_1}(\tau,\xi,\mu), \quad \hbox{for } \ L_1>\lfloor H^{1/2}\rfloor.
\end{align*}
Therefore, expanding the $X_H$-norm in  \eqref{hlh_NH_Psi_proof}, an application of Plancherel Theorem along with H\"older inequality leads us to \begin{align}\label{hlh_proofbil_Psi}
\big\Vert P_H\partial_x(\tilde u_{H_1}V_{H_2})\big\Vert_{N_H}&\lesssim \sup_{t_H\in\R}H^{1/2}\sum_{L,L_1\geq  H^{1/2}} L^{-1/2}\Vert \mathds{1}_{D_{\infty,H,L}}\cdot(f_{H_1,L_1}*f_{H_2})\Vert_{L^2} \nonumber
\\ & \lesssim  \Vert P_{H_1}\tilde u\Vert_{F_{H_1}}\Vert P_{H_2}V\Vert_{L^\infty} .
\end{align}
Then, plugging estimate \eqref{hlh_proofbil_Psi} into the definition of $\mathcal{N}_1$,  by direct computations we see that
\begin{align*}
\mathcal{N}_1&\lesssim \bigg(\sum_{H\in\mathbb{D}}H^{s}\Big(\sum_{H_1\ll H}\Vert P_{H_1}\tilde{u}\Vert_{F_{H_1}}  \Vert P_{H}V\Vert_{L^\infty}\Big)^2\bigg)^{1/2}\lesssim \Vert \tilde u\Vert_{F^{0}}\Vert V\Vert_{L^\infty_tW^{s^+,\infty}_{xy}}.
\end{align*}
In a similar fashion we get that \[
\mathcal{N}_2\lesssim \bigg(\sum_{H\in\mathbb{D}}H^{s}\Big(\sum_{H_2\ll H}\Vert P_{H}\tilde{u}\Vert_{F_{H}}  \Vert P_{H_2}V\Vert_{L^\infty}\Big)^2\bigg)^{1/2}\lesssim \Vert \tilde u\Vert_{F^{s}}\Vert V\Vert_{L^\infty_t W^{0^+,\infty}_{xy}}.
\]
Notice that the same bound also holds for $\mathcal{N}_4$. In the case of $\mathcal{N}_5$ we immediately obtain that  \[
\mathcal{N}_5\lesssim \Vert \tilde u\Vert_{F^0}\Vert V\Vert_{L^\infty_{txy}}.
\]
Finally, as in the previous case, in order to treat the $\mathrm{high}\times \mathrm{high}\to\mathrm{low}$ case $\mathcal{N}_3$ we need to introduce the $\gamma(\cdot)$ function, as in the proof of Lemma \ref{bilinear_hhl}. 
In the same fashion as before, for the sake of simplicity we define \begin{align*}
f^m_{H_i,\lfloor H_1^{1/2}\rfloor}(\tau,\xi,\mu)&:= \eta_{\leq \lfloor H_1^{1/2} \rfloor}(\tau-\omega(\xi,\mu))f^m_{H_i}(\tau,\xi,\mu), \quad \hbox{and}
\\ f^m_{H_i,L_i}(\tau,\xi,\mu)&:= \eta_{\leq L_i}(\tau-\omega(\xi,\mu))f^m_{H_i}(\tau,\xi,\mu), \quad \hbox{for} \quad L_i>\lfloor H_1^{1/2}\rfloor.
\end{align*}
From the above definitions, expanding the $N_H$ and the $X_H$-norm, arguing as in the proof of  \eqref{hlh_proofbil_Psi}, we obtain that \begin{align*}\big\Vert P_H\partial_x(\tilde u_{H_1}V_{H_2})\big\Vert_{N_H}\nonumber & \lesssim H^{1/2}\sup_{t_H\in\R,\, m\in\Z}H_1^{1/2} H^{-{1/2}}\sum_{L\in\mathbb{D}}\sum_{L_1\geq H_1^{1/2}} L^{-1/2}\big\Vert \mathds{1}_{D_{\infty,H,L}}\cdot\big(f^m_{H_1,L_1}*f^m_{H_2}\big)\big\Vert_{L^2}
\\ & \lesssim H^{1/4}\Vert P_{H_1}\tilde u\Vert_{F_{H_1}}\Vert P_{H_2}V\Vert_{L^\infty} .
\end{align*}
It is not difficult to see that the latter estimate leads us to \[
\mathcal{N}_3\lesssim  \bigg(\sum_{H\in\mathbb{D}}H^{s}\Big(\sum_{H_1\gg H}H_1^{1/4}\Vert P_{H_1}\tilde{u}\Vert_{F_{H_1}}  \Vert P_{H_1}V\Vert_{L^\infty}\Big)^2\bigg)^{1/2}\lesssim \Vert \tilde u\Vert_{F^{s}}\Vert V\Vert_{L^\infty_t W^{1/2^+,\infty}_{xy}}.
\]

Therefore, gathering the above estimates we conclude the proof of the proposition.
\end{proof}

\medskip

\section{Energy estimates}\label{sec_energy}

\subsection{A priori estimates for solutions} The main goal of this section is to prove the following key energy estimate for smooth solutions of \eqref{zk_psi}. We recall that, due to the linear estimate \eqref{linear_prop_estimate}, we need to control the $B^s(T)$-norm of a solution $u(t)$ to equation \eqref{zk_psi}, solely as a function of $\Vert u_0\Vert_{E^s}$ and $\Vert u\Vert_{F^s(T)}$. Moreover, recall also that, due to the short-time bilinear estimates derived in the last section, we need to work with $\beta=1/2$ in the definition of the spaces $F^s_\beta$, $F^s_\beta(T)$ and $F_{H,\beta}$, and hence we shall omit the index $\beta$ to simplify the notation.

\begin{prop}\label{energy_prop_state}
Let $T\in(0,1]$ and assume that $s\geq 1$. Consider $u\in C([-T,T]:\, H^\infty(\R^2))$ to be a smooth solution of the IVP \eqref{zk_psi}. Then, \begin{align}\label{energy_estimate_ineq}
\Vert u\Vert_{B^{s}(T)}^2&\lesssim \Vert u_0\Vert_{H^s}^2+T\Vert u\Vert_{F^{s}(T)}\Vert \partial_t\Psi+\partial_x^3\Psi+\partial_x\partial_y^2\Psi+\partial_xf(\Psi)\Vert_{L^\infty_TH^{s}_{xy}}\nonumber
\\ & \qquad +T^{1/2^-}\Vert u\Vert_{F^{1}(T)}\Vert u\Vert_{F^{s}(T)}^2+T^{1/2^-}\Vert \Psi\Vert_{L^\infty_TW^{1^+,\infty}_{xy}}\Vert u\Vert_{F^{s}(T)}^2
\\ & \qquad +T^{1/2^-}\Vert \Psi\Vert_{L^\infty_TW^{s+1^+,\infty}_{xy}}\Vert u\Vert_{F^{1}(T)}\Vert u\Vert_{F^{s}(T)}.\nonumber
\end{align}
\end{prop}

In order to prove the previous proposition, we first show the following technical lemma, which is a consequence of Lemma \ref{lemma_bilineal}.

\begin{lem}\label{energy_bilinear_est}
Let $s\in[1,2)$ and $T\in (0,1]$ be fixed. Consider $N_i,H_i\in\mathbb{D}$ and suppose that $u_i\in F_{H_i}$ for $i=1,2,3$. Then the following holds: \begin{enumerate}
\item[1.] If $H_\mathrm{min}\ll H_\mathrm{max}$, it holds that \begin{align}\label{lema_bili_energy_1}
\left\vert \int_0^T\int_{\R^2} u_1u_2u_3\right\vert \lesssim T^{1/2^-}H_\mathrm{max}^{-1/4}H_\mathrm{min}^{1/4}\Vert u_1\Vert_{F_{H_1}}\Vert u_2\Vert_{F_{H_2}}\Vert u_3\Vert_{F_{H_2}}
\end{align}
\item[2.] If $H_\mathrm{min}\sim H_\mathrm{max}$ and $\mathrm{supp}\,\mathcal{F}(u_i)\subseteq \R\times I_{N_i}\times \R$ for $i=1,2,3$, then \begin{align}\label{lema_bili_energy_2}
\left\vert \int_0^T\int_{\R^2} u_1u_2u_3\right\vert \lesssim T^{1/2^-}N_\mathrm{max}^{-1}H_\mathrm{min}^{1/2}\Vert u_1\Vert_{F_{H_1}}\Vert u_2\Vert_{F_{H_2}}\Vert u_3\Vert_{F_{H_2}}.
\end{align}
\end{enumerate}
\end{lem}

\begin{proof}
First of all, without loss of generality we can always assume that $H_1\leq H_2\leq H_3$. Notice that due to frequency localization and the previous inequalities we must also have that $H_2\sim H_3$. We seek now to prove estimate \eqref{lema_bili_energy_1}.   Let $\gamma:\R\to[0,1]$ be a smooth function supported on $[-1,1]$ with the following property \[
\sum_{m\in\Z}\gamma^3(x-m)=1,\quad \forall x\in\R.
\]
Then, from the above property it follows that \begin{align}\label{deco_bilinear_energy}
\left\vert \int_0^T\int_{\R^2}u_1u_2u_3\right\vert\lesssim \sum_{\vert m\vert \lesssim H_3^{1/2}}\left\vert \int_{\R^3}\prod_{i=1}^3\left(\gamma\big(H_3^{1/2} t-m\big)\mathds{1}_{[0,T]}u_i\right)\right\vert
\end{align}
Now observe that the right-hand side of the above inequality can be split into the following two disjoints sets \begin{align*}
\mathcal{A}&:=\big\{m\in\Z:\, \gamma\big(H_3^{1/2} t-m\big)\mathds{1}_{[0,T]}= \gamma\big(H_3^{1/2} t-m\big)\big\},
\\ \mathcal{B}&:=\big\{m\in\Z:\, \gamma\big(H_3^{1/2} t-m\big)\mathds{1}_{[0,T]}\neq  \gamma\big(H_3^{1/2} t-m\big) \hbox{ and } \gamma\big(H_3^{1/2} t-m\big)\mathds{1}_{[0,T]}\neq 0\big\}.
\end{align*} 
Roughly speaking, $\mathcal{A}$ is the set of $m\in\Z$ such that the support of $\gamma$ contained in $[0,T]$, while for $m\in\mathcal{B}$ the support of $\gamma$ intersects the boundary of $[0,T]$. 

\medskip

We start by analyzing the case $m\in\mathcal{A}$. For the sake of notation, for $i=1,2,3$ we define \[
f_{H_i,H_3^{1/2}}^m:=\eta_{\leq \lfloor H_3^{1/2}\rfloor }(\tau-\omega(\xi,\mu))\left\vert \mathcal{F}\left(\gamma\big(H_3^{1/2} t-m\big)u_i\right)\right\vert,
\]
and, for $L\in\mathbb{D}$ such that $L>H_3^{1/2}$, we set \[
f_{H_i,L}^m:=\eta_{L}\big(\tau-\omega(\xi,\mu)\big)\left\vert \mathcal{F}\left(\gamma\big(H_3^{1/2} t-m\big)u_i\right)\right\vert.
\]
Therefore, by using Plancharel identity as well as estimates \eqref{hmin_ll_hmax_bc} and \eqref{hmin_ll_hmax_aoc} we obtain that \begin{align*}
\sum_{m\in\mathcal{A}}\left\vert \int_{\R^3}\prod_{i=1}^3\left(\gamma\big(H_3^{1/2} t-m\big)\mathds{1}_{[0,T]}u_i\right)\right\vert&\lesssim \sup_{m\in\mathcal{A}}TH_3^{1/2} \sum_{L_1,L_2,L_3\geq H_3^{1/2}} \int_{\R^3} f_{H_1,L_1}^m*f_{H_2,L_2}^m \cdot f_{H_3,L_3}^m
\\ & \lesssim \sup_{m\in\mathcal{A}}TH_3^{1/4-1/2}H_1^{1/4}\sum_{L_1,L_2,L_3\geq H_3^{1/2}}\prod_{i=1}^3L_i^{1/2}\Vert f_{H_i,L_i}^m\Vert_{L^2}.
\end{align*}
The latter estimate along with Corollary \ref{key_cor_prelim} implies that \begin{align*}
\sum_{m\in\mathcal{A}}\left\vert \int_{\R^3}\prod_{i=1}^3\left(\gamma\big(H_3^{1/2} t-m\big)\mathds{1}_{[0,T]}u_i\right)\right\vert \lesssim TH_3^{1/4-1/2}H_1^{1/4}\prod_{i=1}^3\Vert u_i\Vert_{F_{H_i}},
\end{align*}
which concludes the proof for $m\in\mathcal{A}$. Now we seek to bound the case where $m\in\mathcal{B}$. Similarly as before, for the sake of notation we define \[
g_{H_i,L}^m:=\eta_L\big(\tau-\omega(\xi,\mu)\big)\left\vert \mathcal{F}\left(\gamma\big(H_3^{1/2} t-m\big)\mathds{1}_{[0,T]}u_i\right)   \right\vert,
\]
for $i=1,2,3$, $L\in\mathbb{D}$ and $m\in\mathcal{B}$. Also, it is not difficult to notice that $\#\mathcal{B}\leq 4$. Then, using once again estimates \eqref{hmin_ll_hmax_bc} and \eqref{hmin_ll_hmax_aoc} along with Lemma \ref{gro_characteristic_function_lemma}  we infer that \begin{align*}
\sum_{m\in\mathcal{B}}\left\vert \int_{\R^3}\prod_{i=1}^3\left(\gamma\big(H_3^{1/2} t-m\big)\mathds{1}_{[0,T]}u_i\right)\right\vert &\lesssim \sup_{m\in\mathcal{B}}\sum_{L_1,L_2,L_3\in\mathbb{D}} \int_{\R^3} %
\big(g_{H_1,L_1}^m*g_{H_2,L_2}^m\big)\cdot g_{H_3,L_3}^m d\xi d\mu d\tau
\\ & \lesssim \sup_{m\in \mathcal{B}}H_3^{-1/2}H^{1/4}_1 \sum_{L_1,L_2,L_3\in\mathbb{D}} L_{\mathrm{med}}^{-1/2}\prod_{i=1}\sup_{L_i\in\mathbb{D}}L_i^{1/2}\Vert g^m_{H_i,L_i}\Vert_{L^2}
\\ & \lesssim H_3^{-(1/2)^+}H^{1/4}_1\prod_{i=1}^3\Vert u_i\Vert_{F_{H_i}}.
\end{align*}
Note that in the last step of the previous estimate we have used the fact that $L_\mathrm{max}\sim \max\{L_\mathrm{med},\vert \Omega\vert\}$ to control the sum over $L_\mathrm{max}$. Indeed, the case $L_\mathrm{max}\sim L_\mathrm{med}$ follows directly, whereas in the case $L_\mathrm{max}\sim \vert \Omega\vert$ the sum over $L_\mathrm{max}$ is bounded by $H_3^{0^+}$.

\medskip

Finally, we give an sketch of the proof of estimate \eqref{lema_bili_energy_2} since it follows very similar lines to the proof of estimate \eqref{lema_bili_energy_1}. 
In fact, in this case we use again the decomposition in \eqref{deco_bilinear_energy}. Then, we split the problem into two cases, the case where the summation domain is given by $\mathcal{A}$ or $\mathcal{B}$. The former case follows exactly the same lines as above, using estimate \eqref{hmin_sim_hmax_bilinear} instead of \eqref{hmin_ll_hmax_bc} and \eqref{hmin_ll_hmax_aoc} as before, from where we obtain \begin{align*}
\sum_{m\in\mathcal{A}}\left\vert \int_{\R^3}\prod_{i=1}^3\left(\gamma\big(H_3^{1/2} t-m\big)\mathds{1}_{[0,T]}u_i\right)\right\vert \lesssim TN_\mathrm{max}^{-1}H_\mathrm{min}^{1/2}\prod_{i=1}^3\Vert u_i\Vert_{F_{H_i}},
\end{align*}
However, the latter case $m\in\mathcal{B}$ is slightly more complicated since we cannot directly sum in $L_\mathrm{max}$ due to the fact that estimate $\eqref{hmin_sim_hmax_bilinear}$ has the factor $L_\mathrm{med}^{1/2}L_\mathrm{max}^{1/2}$ in the right-hand side. Instead, in this case we interpolate inequality \eqref{general_case_bilinear} with \eqref{hmin_sim_hmax_bilinear} to obtain
\[
\int_{\R^3}(f_1*f_2)\cdot f_3\lesssim TN_{\mathrm{max}}^{-2(1-\varepsilon)/2}H_\mathrm{min}^{(1-\varepsilon)^+/4+\varepsilon(1/4+1/4)}L_\mathrm{min}^{\varepsilon/2}L_\mathrm{med}^{(1-\varepsilon)/2}L_\mathrm{max}^{(1-\varepsilon)/2}\prod_{i=1}^3\Vert f_i\Vert_{L^2},
\] 
for $\varepsilon\in(0,1)$.  This is enough to sum over $L_\mathrm{max}$ and $L_\mathrm{med}$ following the same ideas as before. Therefore, in this case we have that \begin{align*}
\sum_{m\in\mathcal{B}}\left\vert \int_{\R^3}\prod_{i=1}^3\left(\gamma\big(H_3^{1/2} t-m\big)\mathds{1}_{[0,T]}u_i\right)\right\vert\lesssim TN_\mathrm{max}^{-1^+}H_\mathrm{min}^{1/4^+}\prod_{i=1}^3\Vert u_i\Vert_{F_{H_i}},
\end{align*}
which is compatible with \eqref{lema_bili_energy_2}. The proof is complete.
\end{proof}

The following lemma allow us to treat the case where one of the functions belongs to $L^\infty(\R^3)$.
\begin{lem}\label{pseudo_holder_l2}
Let $H_i\in\mathbb{D}$ be a dyadic number, for $i=1,2,3$, satisfying that $H_1\geq H_2$. Consider $u_1,u_2\in L^2(\R^2)$ and $u_{3}\in L^\infty(\R^2)$ such that the Fourier transform of each of them is supported on $\Delta_{H_i}$, respectively. Let $H\in\mathbb{D}$ fixed and set the functional $\mathcal{I}$ as follows 
\begin{align*}
\mathcal{I}:=\left\vert \int_{\Gamma^{3}} a(\xi_1,\mu_1,...,\xi_{3},\mu_{3})\hat{u}_1(\xi_1,\mu_1)\hat{u}_{2}(\xi_{2},\mu_{2})\hat{u}_{3}(\xi_{3},\mu_{3})d \Gamma^{3}\right\vert,
\end{align*}
where the symbol $a(\xi_1,\mu_1,...,\xi_{3},\mu_{3})$ stands for the function \begin{align}\label{definition_symbol_a}
a(\xi_1,\mu_1,...,\xi_{3},\mu_{3}):=\sum_{i=1}^{3}\psi_{H}^2(\xi_i,\mu_i)\xi_i.
\end{align}
Then, the following holds \begin{align*}
\mathcal{I}\lesssim  H_3^{1/2} \Vert u_1\Vert_{L^2}\Vert u_2\Vert_{L^2}\Vert u_3\Vert_{L^\infty}.
\end{align*}
\end{lem}
\begin{proof}
In fact, notice that, except for the terms associated with $i=1,2$ in the definition of $a(\xi_1,\mu_1,...,\xi_{k+1},\mu_{k+1})$, the proof follows directly from Plancherel Theorem and H\"older inequality. Indeed, going back to physical variables, and then using H\"older and Bernstein inequalities, we obtain that
\begin{align*}
&\left\vert \int_{\Gamma^{k+1}} \psi_H^2(\xi_3,\mu_3)\xi_3\hat{u}_1(\xi_1,\mu_1)...\hat{u}_{k+1}(\xi_{k+1},\mu_{k+1})d\Gamma^{k+1}\right\vert
\\ & \qquad \lesssim \Vert u_1\Vert_{L^2}\Vert u_2\Vert_{L^2}\Vert \partial_xP_{ H}^2u_3 \Vert_{L^\infty}\\ & \qquad \lesssim H_3^{1/2} \Vert u_1\Vert_{L^2}\Vert u_2\Vert_{L^2}\Vert u_3\Vert_{L^\infty}.
\end{align*}
Therefore, we can restrict ourselves to study $\mathcal{I}$ but replacing $a(\xi_1,\mu_1,...,\xi_{3},\mu_{3})$ in its definition by the following symbol \[
\widetilde{a}(\xi_1,\mu_1,...,\xi_{3},\mu_{3}):= \psi_{H}^2(\xi_1,\mu_1)\xi_1+\psi_{H}^2(\xi_2,\mu_2)\xi_2.
\]
Next, we split this symbol into two parts as follows
\begin{align*}
\widetilde{a}(\xi_1,\mu_1,...,\xi_{3},\mu_{3})&=
\psi_{H}^2(\xi_1,\mu_1)\big(\xi_1+\xi_2\big)-\big(\psi_{ H}^2(\xi_1,\mu_1)-\psi_{H}^2(\xi_2,\mu_2)\big)\xi_2
\\ &=:\widetilde{a}_1(\xi_1,\mu_1,...,\xi_{3},\mu_{3})+\widetilde{a}_2(\xi_1,\mu_1,...,\xi_{3},\mu_{3}).
\end{align*}
Now notice that, due to the additional restriction imposed by $\Gamma^{3}$ (that is $\xi_1+\xi_2+\xi_{3}=0$), in this integration domain we can rewrite $\widetilde{a}_1(\xi_1,\mu_1,...,\xi_{3},\mu_{3})$ as \[
\widetilde{a}_1(\xi_1,\mu_1,...,\xi_{3},\mu_{3})=-\psi_{H}^2(\xi_1,\mu_1)\xi_3.
\]
Hence, this case also follows from the above analysis. Therefore, it only remains to consider the symbol $\widetilde{a}_2$. In fact, by using Plancherel Theorem, integration by parts and then H\"older inequality we immediately obtain that
\begin{align*}
&\left\vert\int_{\Gamma^3}\widetilde{a}_2(\xi_1,\mu_1,...,\xi_3,\mu_3)\hat{u}_1(\xi_1,\mu_1)\hat{u}_2(\xi_2,\mu_2)\hat{u}_3(\xi_3,\mu_3)d\Gamma^3 \right\vert
\\ &   \qquad \qquad \qquad =\left\vert\int_{\R^2} \big(u_{2,x}P_{H}^2u_1-u_1P_{H}^2u_{2,x}\big) u_3dxdy \right\vert
\\ &   \qquad \qquad \qquad = \left\vert\int_{\R^2} u_2 \partial_x\big( u_3P_{H}^2u_1-P_{H}^2(u_1u_3)\big)dxdy \right\vert
\\ &  \qquad \qquad \qquad \lesssim \Vert u_1\Vert_{L^2}\Vert u_2\Vert_{L^2}\Vert \partial_x u_3\Vert_{L^\infty}+\Vert u_2\Vert_{L^2}\big\Vert [P_{H}^2\partial_x,u_3]u_1\big\Vert_{L^2}.
\end{align*}
Then, since $\Vert \partial_xu_3\Vert_{L^\infty}\lesssim H_3^{1/2}\Vert u_3\Vert_{L^\infty}$, it only remains to control the latter factor of the above inequality. In order to do that, first notice that by direct computations we can write \begin{align*}
&[P_{H}^2\partial_x,u_3] u_1(x,y) 
\\ & \qquad =ic\mathcal{F}_{\xi,\mu}^{-1}\left[\xi \psi_{H}^2(\xi,\mu)\big(\widehat{u}_1*\widehat{u}_3\big)-\big(\xi \psi_{H}^2(\xi,\mu)\widehat{u}_1\big)* \widehat{u}_3\right](x,y)
\\ & \qquad = c\int_{\R^2}\left(\mathcal{F}_{\xi,\mu}^{-1}\big[i\xi \psi_{H}^2(\xi,\mu)\big] (x-\tilde{x},y-\tilde{y})\big(u_3(x,y)-u_3(\tilde{x},\tilde{y})\big)\right)u_1(\tilde{x} , \tilde{y})d\tilde{x} d\tilde{y},
\end{align*} 
for some constant $c\in\R$. Now we define $K(x,y,\tilde{x},\tilde{y})$ as the above kernel, that is,  \begin{align*}
K(x,y,\tilde{x},\tilde{y})&:=\mathcal{F}^{-1}\big[i\xi \psi_{H}^2(\xi,\mu)\big] (x-\tilde{x},y-\tilde{y})\big(u_3(x,y)-u_3(\tilde{x},\tilde{y})\big).
\end{align*}
Then, a direct application of the Mean Value Theorem along with some basic properties of the Fourier Transform, lead us to the following uniform bound \[
\sup_{(\tilde{x},\tilde{y})\in\R^2}\int_{\R^2}\big\vert K(x,y,\tilde{x},\tilde{y})\big\vert dxdy +\sup_{(x,y)\in\R^2}\int_{\R^2}\big\vert K(x,y,\tilde{x},\tilde{y})\big\vert d\tilde{x}d\tilde{y}\lesssim \Vert \nabla u_3\Vert_{L^\infty},
\]
where the implicit constant does not depends of $H_{k+2}$. Therefore, applying Schur lemma, and then Bernstein inequality in the resulting right-hand side, we conclude that 
\begin{align*}
\big\Vert [P_{H}^2\partial_x,u_3]u_1\big\Vert_{L^2}\lesssim \Vert u_1\Vert_{L^2}\Vert \nabla u_3\Vert_{L^\infty}%
\lesssim H_3^{1/2} \Vert u_1\Vert_{L^2}\Vert u_3\Vert_{L^\infty},
\end{align*}
which concludes the proof of the lemma.
\end{proof}
\smallskip
\begin{proof}[Proof of Proposition \ref{energy_prop_state}]
First of all in order to take advantage of the short-time Bourgain spaces, we extend the solution from $[-T,T]$ to the whole line $\R$. Indeed, let $u\in C([-T,T],E^\infty)$ be a solution to equaiton \eqref{zk_psi}. Then, we choose an extension $\widetilde{u}$ of $u$ on $\R^2$ satisfying that \[
\widetilde{u}\big\vert_{[-T,T]\times \R^2}\equiv u \quad\, \hbox{ and } \quad \, \Vert \widetilde{u}\Vert_{F^s}\leq 2\Vert u\Vert_{F^s(T)}.
\] 
For the sake of notation, from now on we drop the tilde and simply write $u$ to denote the previous function. Moreover, we denote by $f(x):=\tfrac12x^2$ the nonlinearity function.

\medskip

Now, we apply the frequency projector $P_H$ to equation \eqref{zk_psi}, with $H>0$ dyadic but arbitrary.  Then, taking the $L^2_{xy}$-scalar product of the
resulting equation against $P_Hu$ and then integrating on $(0, t)$ with $0<t<T$, we obtain
\begin{align*}
\Vert P_Hu(t)\Vert_{L^2}^2&=\Vert P_Hu_0\Vert_{L^2}^2-\int_0^t\int_{\R^2} P_H\big(\partial_t\Psi-D_x^2\partial_x\Psi+\partial_x\partial_y^2\Psi+\partial_xf(u+\Psi)\big)P_Hu. 
\end{align*}
Thus, taking the supremum in time and then multiplying the resulting inequality by $\langle H\rangle^{s}$  we are lead to 
\begin{align*}
\Vert P_H u\Vert_{B^s(T)}^2&\lesssim \Vert P_Hu_0\Vert_{H^s}^2
\\ & \quad +\langle H\rangle^{s}\sup_{t\in(0,T)} \left\vert\int_0^t\int_{\R^2} P_H\big(\partial_t\Psi-D_x^2\partial_x\Psi+\partial_x\partial_y^2\Psi+\partial_xf(u+\Psi)\big)P_Hu\right\vert.
\end{align*}
Therefore, in order to conclude the proof of the proposition, we need to control the sum over all $H\in\mathbb{D}$ of the second term in the right-hand side of the latter inequality. We divide the analysis into several steps, each of which dedicated to bound one of the following integrals:
\begin{align}
\mathrm{I}&:=\sum_{H>0}\langle H\rangle^{s}\sup_{t\in(0,T)}\left\vert\int_0^t\int_{\R^2}\partial_xP_H(f(u+\Psi)-f(\Psi))P_Hu\right\vert,\label{sum_I_proof}
\\ \mathrm{II}&:= \sum_{H>0}\langle H\rangle^{s}\sup_{t\in(0,T)}\left\vert\int_0^t\int_{\R^2} P_H\big(\partial_t\Psi-D_x^2\partial_x\Psi+\partial_x\partial_y^2\Psi+\partial_xf(\Psi)\big)P_Hu\right\vert.\nonumber
\end{align}
For the sake of simplicity, from now on we denote by $\mathrm{I}_{u^2}$ and $\mathrm{I}_{u\Psi}$ the quantities given by \begin{align}
\mathrm{I}_{u^2}&:= \sum_{H>0}\langle H\rangle^{s}\sup_{t\in(0,T)}\left\vert\int_0^t\int_{\R^2}\partial_xP_H(u^2)P_Hu\right\vert, \label{I_un}
\\ \mathrm{I}_{u\Psi}&:=\sum_{H>0}\langle H\rangle^{s}\sup_{t\in(0,T)}\left\vert\int_0^t\int_{\R^2}\partial_xP_H(u\Psi)P_Hu\right\vert.\label{I_un_psim}
\end{align}

\smallskip

\textbf{Step 1:} We begin by controlling $\mathrm{II}$ right away. In fact, by using hypothesis \eqref{hyp_psi_general}, we infer that it is enough to use Cauchy-Schwarz and Bernstein inequalities, from where we obtain \begin{align*}
\mathrm{II} \lesssim T\Vert u\Vert_{F^s(T)}\Vert \partial_t \Psi+\partial_x^3\Psi+\partial_xf(\Psi)\Vert_{L^\infty_TH^{s}_{xy}}.
\end{align*}
This concludes the proof of the first case.

\medskip

\textbf{Step 2:} Now we aim to control $\mathrm{I}_{u^2}$, that is, by Plancherel Theorem, we aim to control the following quantity
\begin{align}\label{def_iu2}
\mathrm{I}_{u^2}&=\sum_{H>0}\langle H\rangle^{s}\sup_{t\in(0,T)}\left\vert\int_0^t\int_{\Gamma^{3}}a_2(\xi_1,\mu_1,...,\xi_{3},\mu_{3})\hat{u}(s,\xi_1,\mu_1)...\hat{u}(s,\xi_{3},\mu_3)d\Gamma^{3}ds\right\vert,
\end{align}
where the symbol $a_2(\xi_1,\mu_1,...,\xi_{3},\mu_{3})$ is explicitly given by \begin{align*}
a_2\big(\xi_1,\mu_1,...,\xi_{3},\mu_{3}\big):=i\psi_H^2(\xi_{3},\mu_{3})\xi_{3}.
\end{align*}
We point out that in the previous identity \eqref{def_iu2} we have used both, the fact that $u(t,\cdot,\cdot)$ is real-valued as well as the fact that $\psi_{H}(\xi,\cdot)$ and $\psi_H(\cdot,\mu)$ are both even. Then, in order to deal with this case, we symmetrize the symbol $a_2$, that is, from now on we consider \begin{align*}
\widetilde{a}_2(\xi_1,\mu_1,...,\xi_{3},\mu_{3}):=\big[a_2(\xi_1,\mu_1,...,\xi_{3},\mu_{3})\big]_{\mathrm{sym}}=\dfrac{i}{3}\sum_{i=1}^{3}\psi_H^2(\xi_i,\mu_i)\xi_i.
\end{align*}
Thus, by using frequency decomposition and the above symmetrization, the problem of bounding \eqref{def_iu2} is reduced to control the following quantity \begin{align}\label{sums_step31}
\sum_{H>0}\langle H\rangle^{s}\sup_{t\in(0,T)}\left\vert\int_0^t\sum_{H_1,H_2,H_{3}}\int_{\Gamma^{3}}\widetilde{a}_2(\xi_1,\mu_1,...,\xi_{3},\mu_3)\prod_{i=1}^{3}\psi_{H_i}(\xi_i,\mu_i)\hat{u}(t',\xi_i,\mu_i)d\Gamma^{3}dt'\right\vert.
\end{align}
Moreover, by symmetry, without loss of generality we can always assume that $H_1\leq H_2\leq H_3$.  On the other hand, from the explicit form of $\widetilde{a}_2$, it is not difficult to see that  $\widetilde{a}_2\equiv 0$, unless $ H_3\geq \tfrac{1}{2}H$. Furthermore, due to the additional constraint\footnote{By this we mean the condition $\xi_1+\xi_2+\xi_{3}=0$. In the sequel, each time we mention ``the constraint imposed by $\Gamma^k$'' we refer to the previous condition with $k$ frequencies.} imposed by $\Gamma^{3}$, we must also have that $H_2\geq \tfrac{1}{4}H_1$. Therefore, we have that $H_1\sim H_2$ with $H_1\geq \tfrac{1}{2}H$. Before going further, it is important to observe that we can split the symbol $\widetilde{a}_2$ as the sum of the following symbols \begin{align*}
\widetilde{a}_2^1(\xi_1,\mu_1,...,\xi_{3},\mu_{3})&:=\psi_H^2(\xi_3,\mu_3)(\xi_2+\xi_3),
\\ \widetilde{a}_2^2(\xi_1,\mu_1,...,\xi_{3},\mu_{3})&:=\big(\psi_H^2(\xi_2,\mu_2)-\psi_H^2(\xi_3,\mu_3)\big)\xi_2,
\\ \widetilde{a}_2^3(\xi_1,\mu_1,...,\xi_{3},\mu_{3})&:=\widetilde{a}_2(\xi_1,\mu_1,...,\xi_{3},\mu_{3})-\psi_H^2(\xi_2,\mu_2)\xi_2-\psi_H^2(\xi_3,\mu_3)\xi_3.
\end{align*}
We claim that, due to frequency localization along with the definition of $\Gamma^2$, the following inequality holds
\begin{align}\label{symbol_bound_l2}
\big\vert \widetilde{a}_2^1(\xi_1,\mu_1,...,\xi_{3},\mu_{3})\big\vert+\big\vert \widetilde{a}_2^2(\xi_1,\mu_1,...,\xi_{3},\mu_{3})\big\vert+\big\vert \widetilde{a}_2^3(\xi_1,\mu_1,...,\xi_{3},\mu_{3})\big\vert \lesssim H_1^{1/2}.
\end{align}
In fact, the bound for $\widetilde{a}_2^3$ follows directly. Then, on the one-hand, to bound $\widetilde{a}_2^1$ it is enough to use that $\xi_2+\xi_3=-\xi_1$, while  on the other hand, in order to estimate $\widetilde{a}_2^2$ we use the Mean Value Theorem, which leads us to \[
\big\vert \widetilde{a}_2^2(\xi_1,\mu_1,...,\xi_{3},\mu_{3})\big\vert\lesssim  \vert\xi_1 \vert+\vert\mu_1\vert.
\]
what concludes the proof of the claim. Now, with these estimates at hand, we split the analysis into two possible cases, namely \begin{align}\label{cases_step31}
\hbox{either }\quad H_1\ll H_3 \quad \hbox{ or }\quad H_1\sim H_3.
\end{align}
From now on we denote by\footnote{Recall that we are also assuming that $H_1\leq H_2\leq H_3$.} \begin{align*}
\mathbf{H}^1:=\big\{(H_1,H_2,H_3)\in \mathbb{D}^{3}: \, H_1\ll H_3\big\} \quad \hbox{ and }\quad \mathbf{H}^2:=\big\{(H_1,H_2,H_3)\in \mathbb{D}^{3}: \, H_1\sim H_3\big\},
\end{align*} 
and by $\mathcal{G}_1$ and $\mathcal{G}_2$, the corresponding contribution of \eqref{sums_step31} associated with each of these regions\footnote{That is, the quantity obtained once restricting the inner sum in \eqref{sums_step31} to $\mathbf{H}^1$ and $\mathbf{H}^2$, respectively.}, respectively. Then, to bound $\mathcal{G}_1$ it is enough to notice that, by Lemma \ref{energy_bilinear_est}, along with estimate \eqref{symbol_bound_l2}, we have that \begin{align*}
\big\vert \mathcal{G}_1\big\vert &\lesssim \sum_{H\in\mathbb{D}}\sum_{\mathbf{H}^1}\langle H\rangle^{s}H_\mathrm{max}^{-1/4}H_\mathrm{min}^{3/4}\prod_{i=1}^3\Vert P_{H_i}u\Vert_{F_{H_i}}\lesssim \Vert u\Vert_{F^{1}(T)}\Vert u\Vert_{F^s(T)}^2,
\end{align*}
where we have used that we can sum over $H_\mathrm{min}$ first, so that we do not loose an $H_\mathrm{min}^{+}$ factor. In fact, this latter observation comes from the fact that \[
H_\mathrm{max}^{-1/4}\sum_{H_{\mathrm{min}}\ll H_\mathrm{max}}H_\mathrm{min}^{1/4}\ll 1,
\]
uniformly in $H_\mathrm{max}$. Then, we sum over $H_\mathrm{max}$ using that $H^{s/2}\Vert P_H f\Vert_{F_H}$ defines a square summable sequence, provided that $f\in F^s$. On the other hand, in order to treat $\mathcal{G}_2$, first of all notice that, due to the factor $P_H$ coming from the symbol $\widetilde{a}_2$, in this case we actually have \[
H_1\sim H_2\sim H_3\sim H.
\]
Now we decompose into frequencies again, using the $P_N^x$ projectors as follows \begin{align*}
P_{H_1}uP_{H_2}uP_{H_3}u&=\sum_{N_1,N_2,N_3}P_{H_1}P_{N_1} ^xuP_{H_2}P_{N_2}^xuP_{H_3}P_{N_3}^xu.
\end{align*} 
Notice that, by frequency localization, we must have $N_\mathrm{max}\sim N_\mathrm{med}$, otherwise $\mathcal{G}_2=0$. Hence, by symmetry, we infer that it is enough to bound the contribution of the following quantity \begin{align*}
\sum_{N\in\mathbb{D}}P_{H_1}P_N ^xuP_{H_2}P_{\sim N}^xu\left(P_{H_3}P_{\ll N}^xu+P_{H_3}P_{\sim N}^xu\right)=:\Sigma_1+\Sigma_2,
\end{align*}
since all the remaining cases are equivalent to the latter one. We denote by $\mathcal{G}_2^1$ and $\mathcal{G}_2^2$ the restriction of $\mathcal{G}_2$ to each of these sums respectively. Then, in order to bound $\mathcal{G}_2^2$, we use the second part of Lemma \ref{energy_bilinear_est}, but instead of using estimate \eqref{symbol_bound_l2} (which would lead us to a factor $H_\mathrm{min}^{1/2}\sim H_\mathrm{max}^{1/2}$), we bound the contribution of the symbol directly from its explicit definition (that only involves $N_i$), from where we obtain \begin{align*}
\big\vert \mathcal{G}_2^2\big\vert &\lesssim \sum_{H\in\mathbb{D}}\sum_{\mathbf{H}^2}\sum_{N\in\mathbb{D}}\langle H\rangle^{s} H_{\mathrm{min}}^{1/2}\Vert P_{H_1}P_{N}^xu\Vert_{F_{H_1}}\Vert P_{H_2}P_{\sim N}^xu\Vert_{F_{H_2}}\Vert P_{H_3}P_{\sim N}^xu\Vert_{F_{H_3}}
\\ & \lesssim \sum_{H\in\mathbb{D}}\sum_{\mathbf{H}^2}\langle H\rangle^{s} H_{\mathrm{min}}^{1/2}\Vert P_{H_1}u\Vert_{F_{H_1}}\Vert P_{H_2}u\Vert_{F_{H_2}}\Vert P_{H_3}u\Vert_{F_{H_3}}
\\ &\lesssim \Vert u\Vert_{F^{1}(T)}\Vert u\Vert_{F^s(T)}^2.
\end{align*}
Once again, here we have used the fact that both $\Vert P_Hu\Vert_{F_H}$ and $\Vert P_Nu\Vert_{L^2}$ are square summable under the current hypotheses, and hence, using that $H_\mathrm{min}\sim H_\mathrm{max}$ and $N_\mathrm{min}\sim N_\mathrm{max}$ we bound one of the three factors above by its corresponding whole series, and for the other two we use the summability property. Notice that the estimate for $\mathcal{G}_2^1$ follows the exact same lines. For this latter case the important point is to not expand $P_{\ll N}u$ into a sum, otherwise we might loose a $N^{0+}$ factor, while written in the above fashion it suffices to use the continuity of the operator $P_{\ll N}$ in $L^2$. Therefore, we conclude that \[
\big\vert \mathcal{G}_2\big\vert\lesssim \Vert u\Vert_{F^{1}(T)}\Vert u\Vert_{F^s(T)}^2.
\]
\textbf{Step 3:} Now we aim to control $\mathrm{I}_{u\Psi}$, that is, we aim to control the following quantity  \begin{align}\label{def_iupsi}
\mathrm{I}_{u\Psi}&=\sum_{H>0}\langle H\rangle^{s}\sup_{t\in(0,T)}\left\vert\int_0^t\int_{\Gamma^{3}}a_p(\xi_1,\mu_1,...,\xi_{3},\mu_{3})\hat{u}(s,\xi_1,\mu_1)\hat{u}(s,\xi_{2},\mu_2)\widehat{\Psi}(s,\xi_3,\mu_3)\right\vert,
\end{align}
where the symbol $a_2(\xi_1,\mu_1,...,\xi_{3},\mu_{3})$ is explicitly given by \begin{align*}
a_p\big(\xi_1,\mu_1,...,\xi_{3},\mu_{3}\big):=i\psi_H^2(\xi_{1},\mu_{1})\xi_{1}.
\end{align*}
As before, in the previous identity \eqref{def_iupsi} we have used both, the fact that $u(t,\cdot,\cdot)$ and $\Psi(t,\cdot,\cdot)$ are real-valued as well as the fact that $\psi_{H}(\xi,\cdot)$ and $\psi_H(\cdot,\mu)$ are both even. Then, in order to deal with this case, we symmetrize the symbol $a_p$, that is, from now on we consider \begin{align*}
\widetilde{a}_p(\xi_1,\mu_1,...,\xi_{3},\mu_{3}):=\big[a_p(\xi_1,\mu_1,...,\xi_{3},\mu_{3})\big]_{\mathrm{sym}}=\dfrac{i}{2}\sum_{i=1}^{2}\psi_H^2(\xi_i,\mu_i)\xi_i.
\end{align*}
Hence, by using frequency decomposition and the above symmetrization, the problem of bounding \eqref{def_iupsi} is reduced to control the following quantity \begin{align}\label{sums_step31}
\sum_{H>0}\langle H\rangle^{s}\sup_{t\in(0,T)}\Big\vert\int_0^t\sum_{H_1,...,H_{3}}\int_{\Gamma^{3}}\widetilde{a}_p(\xi_1,\mu_1,...,\xi_{3},\mu_3)\psi_{H_3}(\xi_3,\mu_3)\widehat{\Psi}(s,\xi_3,\mu_3)\prod_{i=1}^{2}\psi_{H_i}(\xi_i,\mu_i)\hat{u}\Big\vert.
\end{align}
Moreover, by symmetry, without loss of generality we can always assume that $H_1\leq H_2$.  On the other hand, from the explicit form of $\widetilde{a}_p$, it is not difficult to see that  $\widetilde{a}_p\equiv 0$, unless $ H_2\geq \tfrac{1}{2}H$. Besides, due to the additional constraint imposed by $\Gamma^{3}$, we must also have that \[
H_2\sim \max\{H_1,H_3\}.
\]
Now, we introduce a smooth function $\gamma:\R\to[0,1]$, supported on $[-1,1]$, satisfying that \[
\sum_{m\in\Z}\gamma^2(x-m)=1, \qquad \forall x\in\R.
\] 
Thus, proceeding similarly as at the beginning of the proof\footnote{By this we are only referring to the fashion we chop the time interval, introducing the functions $f^m_{H_i,L}$ and so on, but not to the use of the bilinear estimates, since they are not useful here.} of Lemma  \ref{energy_bilinear_est},  and then using Lemma \ref{pseudo_holder_l2}, it is not difficult to see that \begin{align*}
\big\vert I_{u\Psi}\big\vert &\lesssim \sum_{H\in \mathbb{D}}\sum_{H_1,H_2,H_3}\langle H\rangle^{s}H_3^{1/2} \Vert P_{H_1}u\Vert_{F_{H_1}}\Vert P_{H_2}u\Vert_{F_{H_2}}\Vert P_{H_3}\Psi\Vert_{L^\infty_{txy}}
\\ & \lesssim \Vert \Psi\Vert_{L^\infty_tW^{s+1^+,\infty}_{xy}}\Vert u\Vert_{F^{0}(T)}\Vert u\Vert_{F^s(T)}+\Vert \Psi\Vert_{L^\infty_tW^{1^+,\infty}_{xy}}\Vert u\Vert_{F^s(T)}^2.
\end{align*}
Notice that in order to sum over $(H,H_1,H_2,H_3)$ we have used the fact that $H_2\sim\max\{H_1,H_3\}$. In fact, if $H_2\sim H_3$, with $H_1\ll H$, then it is enough to absorb the factor $\langle H\rangle^{s/2+1/2^+}$ with the $L^\infty$ norm of $\Psi$, so that we can easily sum over $H_2$. While if $H_2\sim H_3$ with $H_1\sim H$ we absorb $H_\mathrm{max}^{1/2^+}$ with $\Psi$. The case where $H_1\sim H_2$ follows directly from the square summability property of $\Vert P_{H_i}u\Vert_{F_{H_i}}$, $i=1,2$. In this latter case we use $\Psi$ to absorb the factor $H_3^{1/2^+}$. The proof is complete.
\end{proof}

\smallskip

\subsection{A priori estimates for the difference of two solutions}

In this subsection we seek to establish the key a priori estimate for the difference of two solutions.

\begin{prop}\label{diff_sol_prop}
Let $T\in(0,1]$ and $s\geq 1$, both fixed. Consider $u,v\in C([-T,T]:\, H^\infty(\R^2))$ to be a smooth solution of the IVP \eqref{zk_psi}. Denote by $w:=u-v$ and by $z:=u+v$. Then, \begin{align}\nonumber
\Vert w\Vert_{B^{s}(T)}^2&\lesssim \Vert w_0\Vert_{H^s}^2+\Vert z\Vert_{F^{1}(T)}\Vert w\Vert_{F^{s}(T)}^2+\Vert z\Vert_{F^{s+1/2}(T)}\Vert w\Vert_{F^{1/2^+}(T)}\Vert w\Vert_{F^{s}(T)}
\\ & \qquad +\Vert \Psi\Vert_{L^\infty_tW^{s+1^+,\infty}_{xy}}\Vert w\Vert_{F^{0}(T)}\Vert w\Vert_{F^{s}(T)}+\Vert \Psi\Vert_{L^\infty_tW^{1^+,\infty}_{xy}}\Vert w\Vert_{F^{s}(T)}^2.\label{diff_s_hs}
\end{align}
Moreover, the following holds \begin{align}\label{diff_scero}
\Vert w\Vert_{B^0(T)}^2&\lesssim \Vert w_0\Vert_{E^s}^2+\Vert z\Vert_{F^{1}(T)}\Vert w\Vert_{F^0(T)}^2+\Vert \Psi\Vert_{L^\infty_TW^{1^+,\infty}_{xy}}\Vert w\Vert_{F^0(T)}^2.
\end{align}
\end{prop}

\begin{proof}
We only prove the first inequality since \eqref{diff_scero} follows almost the same lines (and is in fact easier to prove). In a similar fashion as before, in order to take advantage of the short-time Bourgain estimates, we extend the solutions from $[0,T]$ to the whole real line $\R$. For the sake of notation, we keep denoting these extensions by $u$ and $v$. Now, let us denote by $w:=u-v$ the difference of both solutions. Then, $w(t,x)$ satisfies the equation \begin{align*}
\partial_t w+\partial_x^3w+\partial_x\partial_y^2w+\partial_x\big(f(u+\Psi)-f(v+\Psi)\big)=0.
\end{align*}
Then, we apply the frequency projector $P_H$ to the latter equation, take the $L^2_{xy}$-scalar product of the
resulting equation against $P_Hw$ and multiply the result by $\langle H\rangle^{s}$. Finally, integrating on $(0, t)$ with $0<t<T$ and taking the supremum in time we are lead to
\begin{align*}
\Vert P_H w\Vert_{B^s(T)}^2&\lesssim \Vert P_Hw_0\Vert_{H^s}^2
\\ & \quad +\langle H\rangle^{s}\sup_{t\in(0,T)} \left\vert\int_0^t\int_{\R^2} P_H\big(f(u+\Psi)-f(v+\Psi)\big)\partial_xP_Hw\right\vert.
\end{align*}
Before getting into the details, let us comment that, the ideas behind the estimates we shall prove below are the same as those of the
proof of the latter proposition. However, since in this case we have more (different) functions, we shall have more cases as well, since we cannot order all the frequencies appearing
in $P_H(u^iv^{1-i}\Psi w)$, $i=0,1$, as we did in the previous proposition where there was only $u$.

\medskip

As we shall see, the estimates below are symmetric with respect to $u$ and $v$, and hence it suffices to bound the contribution of 
\begin{align}
\mathrm{I}_z&:=\sum_{H>0}\langle H\rangle^{s}\sup_{t\in(0,T)}\left\vert\int_0^t\int_{\R^2}\partial_xP_H(zw)P_Hw\right\vert,\label{sum_Iz_proof}
\\ \mathrm{I}_\Psi&:= \sum_{H>0}\langle H\rangle^{s}\sup_{t\in(0,T)}\left\vert\int_0^t\int_{\R^2} \partial_xP_H(\Psi w)P_Hw\right\vert,\label{sum_Izpsi_proof}
\end{align}
for $z\in F^s(T)$, where we have used the fact that $f(u+\Psi)-f(v+\Psi)=w\big(u+v+2\Psi\big)$. Notice that, in the same fashion as in the previous proposition, in this case we begin by symmetrizing the underlying symbol in both above cases, which leads us to study the symbol \[
a(\xi_1,\mu_1,...,\xi_3,\mu_3):=\dfrac{i}{2}\psi_H^2(\xi_1,\mu_1)\xi_1+\dfrac{i}{2}\psi_H^2(\xi_2,\mu_2)\xi_2,
\]
where $(\xi_1,\mu_1)$ and $(\xi_2,\mu_2)$ denote (respectively) the frequencies of each of the occurrences of $w$ in \eqref{sum_Iz_proof} and \eqref{sum_Izpsi_proof}. 
Hence, by using frequency decomposition and the above symmetrization, the is reduced to control the following quantity \begin{align}\label{diff_reduced_p}&\mathbf{I}:=\sum_{H>0}\langle H\rangle^{s}\times
\\ & \ \   \times \sup_{t\in(0,T)}\Big\vert\int_0^t\sum_{H_1,H_2,H_{3}}\int_{\Gamma^{2}}a(\xi_1,\mu_1,...,\xi_{3},\mu_3)\Phi(t'\xi_3,\mu_3)\prod_{i=1}^{2}\psi_{H_i}(\xi_i,\mu_i)\hat{w}(t',\xi_i,\mu_i)d\Gamma^{2}dt'\Big\vert,\nonumber
\end{align}
where $\Phi(t,\xi_3,\mu_3)$ is either $\psi_{H_3}(\xi_3,\mu_3)\hat{z}(t',\xi_3,\mu_3)$ or $\psi_{H_3}(\xi_3,\mu_3)\widehat{\Psi}(t',\xi_3,\mu_3)$. Then, by symmetry, without loss of generality we can always assume that $H_1\geq H_2$. Also, from the explicit form of $a(\xi_1,\mu_1,...,\xi_3,\mu_3)$, it is not difficult to see that  $a\equiv 0$, unless $ H_1\geq \tfrac{1}{2}H$. Moreover, due to the additional constraint imposed by $\Gamma^{2}$, we must also have that $H_1\sim \max\{H_2,H_3\}$.

\medskip

\textbf{Step 1:} We seek to estimate the contribution of \eqref{diff_reduced_p} in the case where $\Phi(t',\xi_3,\mu_3)=\psi_{H_3}(\xi_3,\mu_3)\hat{z}(t',\xi_3,\mu_3)$. We split the analysis into three possible cases, namely \begin{align}\label{cases_H_diff_sol}
\mathbf{H}^1&:=\big\{(H_1,H_2,H_3)\in\mathbb{D}^3:\, H_1\sim H_2 \ \hbox{ and } \ H_1\gg H_3 \big\},\nonumber
\\ \mathbf{H}^2&:=\big\{(H_1,H_2,H_3)\in\mathbb{D}^3:\, H_1\sim H_3 \ \hbox{ and } \ H_1\gg H_2 \big\},
\\ \mathbf{H}^3&:=\big\{(H_1,H_2,H_3)\in\mathbb{D}^3:\, H_1\sim H_2 \sim H_3 \big\}.\nonumber
\end{align}
As before, let us denote by $\mathcal{G}_1$, $\mathcal{G}_2$ and $\mathcal{G}_3$ the corresponding contribution of \eqref{diff_reduced_p} associated with each of these regions, respectively. Then, in order to bound $\mathcal{G}_1$ first notice that, due to the explicit form of the symbol, we infer htat $H_1\sim H_2\sim H$. On the other hand, it is not difficult to see, as an application of the Mean Value Theorem, that \begin{align}\label{symbol_bound_diff}
\big\vert a(\xi_1,\mu_1,...,\xi_3,\mu_3)\big\vert \lesssim H_3^{1/2}.
\end{align}
Then, as a consequence of Lemma \ref{energy_bilinear_est}, along with the estimate for the symbol above, we obtain that \begin{align*}
\big\vert \mathcal{G}_1\big\vert &\lesssim \sum_{H\in\mathbb{D}}\sum_{\mathbf{H}^1}\langle H\rangle^{s}H^{-1/4}H_3^{3/4}\Vert P_{H_3}z\Vert_{F_{H_3}}\prod_{i=1}^2\Vert P_{H_i}w\Vert_{F_{H_i}}\lesssim \Vert z\Vert_{F^{1}(T)}\Vert w\Vert_{F^s(T)}^2.
\end{align*}
On the other hand, in order to deal with $\mathcal{G}_2$, we proceed similarly as in the latter estimate, except that in this case we cannot control the symbol by $H_2^{1/2}$, which is the minimum frequency (as in the previous case). Therefore, using Lemma \ref{energy_bilinear_est}, using the trivial bound $H_1^{1/2}$ for the symbol, we obtain that \[
\big\vert \mathcal{G}_2\big\vert \lesssim \sum_{H\in\mathbb{D}}\sum_{\mathbf{H}^1}\langle H\rangle^{s}H_1^{1/4}H_2^{1/4}\Vert P_{H_3}z\Vert_{F_{H_3}}\prod_{i=1}^2\Vert P_{H_i}w\Vert_{F_{H_i}}\lesssim \Vert z\Vert_{F^{s+1/2}(T)}\Vert w\Vert_{F^{1/2^+}(T)}\Vert w\Vert_{F^s(T)}.
\]
Finally, in order to treat $\mathcal{G}_3$, we decompose into frequencies again, using the $P_N^x$ projectors as before, that is, we write \begin{align*}
P_{H_1}wP_{H_2}wP_{H_3}z&=\sum_{N_1,N_2,N_3}P_{H_1}P_{N_1} ^xwP_{H_2}P_{N_2}^xwP_{H_3}P_{N_3}^xz.
\end{align*} 
Then, by frequency localization, we must have that $N_\mathrm{max}\sim N_\mathrm{med}$, otherwise $\mathcal{G}_3=0$. Hence, by symmetry, we infer that it is enough to bound the contribution of the following quantity \begin{align*}
&\sum_{N\in\mathbb{D}}P_{H_1}P_N ^xwP_{H_2}P_{\sim N}^xw\left(P_{H_3}P_{\ll N}^xz+P_{H_3}P_{\sim N}^xz\right)
\\ & \qquad + \sum_{N\in\mathbb{D}}P_{H_1}P_N ^xwP_{H_2}P_{\ll N}^xwP_{H_3}P_{\sim N}^xz =:\Sigma_1+\Sigma_2+\Sigma_3,
\end{align*}
since all the remaining cases are equivalent to one of the above cases. We denote by $\mathcal{G}_3^1$, $\mathcal{G}_3^2$ and $\mathcal{G}_3^3$ the restriction of $\mathcal{G}_3$ to each of these sums respectively. Then, in order to bound $\mathcal{G}_3^1$, arguing exactly as in the corresponding case in the proof of the previous proposition, we use the second part of Lemma \ref{energy_bilinear_est}, but instead of using estimate \eqref{symbol_bound_diff} and we bound the contribution of the symbol directly from its explicit definition, from where we obtain
\[
\big\vert \mathcal{G}_3^1\big\vert \lesssim \sum_{H\in\mathbb{D}}\langle H\rangle^{s}H^{1/2} \Vert P_{H_1}w\Vert_{F_{H_1}}\Vert P_{H_2}w\Vert_{P_{H_2}}\Vert P_{H_3}z\Vert_{P_{H_3}}\lesssim \Vert z\Vert_{F^{1}(T)}\Vert w\Vert_{F^s(T)}^2.
\]
Notice that, proceeding in a similar fashion, we obtain that the exact same bound holds for the remaining two cases $\mathcal{G}_3^2$ and $\mathcal{G}_3^3$, and hence we conclude that \[
\big\vert \mathcal{G}_3\big\vert \lesssim  \Vert z\Vert_{F^{1}(T)}\Vert w\Vert_{F^s(T)}^2.
\]

\smallskip

\textbf{Step 2:} We seek now to estimate the contribution of \eqref{diff_reduced_p} in the case where $\Phi(t',\xi_3,\mu_3)=\psi_{H_3}(\xi_3,\mu_3)\widehat{\Psi}(t',\xi_3,\mu_3)$. This case follows very similar lines as the corresponding case in the proof of the previous proposition. In fact, introducing the $\gamma$ function $\gamma:\R\to[0,1]$ as at the beginning of the proof of Lemma  \ref{energy_bilinear_est},  and then using Lemma \ref{pseudo_holder_l2}, we conclude that \begin{align*}
\vert \mathbf{I}\vert &\lesssim \sum_{H\in \mathbb{D}}\sum_{H_1,H_2,H_3}\langle H\rangle^{s}H_3^{1/2}\Vert P_{H_1}w\Vert_{F_{H_1}}\Vert P_{H_2}w\Vert_{F_{H_2}}\Vert P_{H_3}\Psi\Vert_{L^\infty_{txy}}
\\ & \lesssim \Vert \Psi\Vert_{L^\infty_tW^{s+1^+,\infty}_{xy}}\Vert w\Vert_{F^{0}(T)}\Vert w\Vert_{F^{s}(T)}+\Vert \Psi\Vert_{L^\infty_tW^{1^+,\infty}_{xy}}\Vert w\Vert_{F^{s}(T)}^2.
\end{align*}
Once again, in order to sum over $(H,H_1,H_2,H_3)$ we have used the fact that $H_1\sim\max\{H_2,H_3\}$ and that either $H_1\sim H$ or $H_2\sim H$, otherwise the symbol $a\equiv0$. In fact, if $H_1\sim H_3$ with $H_2\ll H$, then it is enough to absorb the factor $\langle H\rangle^{s/2+1/2^+}$ with the $L^\infty$ norm of $\Psi$ so that we can then sum over $H_1$. While if $H_1\sim H_3$ with $H_2\sim H$ we absorb $H_\mathrm{max}^{1/2^+}$ with $\Psi$. The case where $H_1\sim H_2$ follows directly from the square summability property of $\Vert P_{H_i}w\Vert_{F_{H_i}}$, $i=1,2$. The proof is complete.
\end{proof}

\begin{rem}
For the existence Theorem we shall actually use a slight modification of inequality \eqref{diff_s_hs}, namely \begin{align}\nonumber
\Vert w\Vert_{B^{s}(T)}^2&\lesssim \Vert w_0\Vert_{H^s}^2+\Vert z\Vert_{F^{1}(T)}\Vert w\Vert_{F^{s}(T)}^2+\Vert z\Vert_{F^{s+1}(T)}\Vert w\Vert_{F^{0}(T)}\Vert w\Vert_{F^{s}(T)}
\\ & \qquad +\Vert \Psi\Vert_{L^\infty_tW^{s+1^+,\infty}_{xy}}\Vert w\Vert_{F^{0}(T)}\Vert w\Vert_{F^{s}(T)}+\Vert \Psi\Vert_{L^\infty_tW^{1^+,\infty}_{xy}}\Vert w\Vert_{F^{s}(T)}^2,\label{diff_s_util}
\end{align}
which follows exactly the same lines as before, but when bounding $\mathcal{G}_2$ we absorb $H_1^{1/4}H_2^{1/4}$ with $z$ instead of splitting it between $z$ and $w$.
\end{rem}
\medskip

\section{Local well-posedness}\label{sec_loc}

In this section we seek to prove Theorem \ref{MT1}. Our starting point is a local well-posedness result for smooth solutions of the generalized model \eqref{zk_psi}. The proof is standard and follows classical estimates that can be found in \cite{AbBoFeSa, Ioro}, for instance. The main idea comes from introducing a parabolic regularizing term $-\varepsilon\Delta u$ to the equation, and then passing to the limit $\varepsilon\to 0$. Let us recall that $\Psi\in L^\infty(\R,W^{4+\varepsilon_\star,\infty}_{xy}(\R^2))$. This latter hypothesis allows us to perform the main commutator estimates involved in the proof of the smooth LWP Theorem without additional problems.

\begin{thm}\label{regular_lwp_thm}
For all $u_0\in H^{s}(\R^2)$ with $s\in(2,3+\tfrac12\varepsilon_\star)$, there exist a positive time $T>0$ and a unique solution $u\in C([-T,T],H^s(\R^2))$ to the initial value problem  \eqref{zk_psi}. Moreover, the minimal existence time satisfies that \[
T=T\big(\Vert u_0\Vert_{H^s},\Vert \Psi\Vert_{L^\infty_tW^{4+\varepsilon_\star,\infty}_{xy}},\Vert \partial_t\Psi+\partial_x\Delta\Psi+\tfrac{1}{2}\partial_x(\Psi^2)\Vert_{L^\infty_tH^{3+\varepsilon_\star}_{xy}}\big)>0,
\]
and it can be chosen as a nonincreasing function of its arguments. Moreover, the data-to-solution map $\Phi:u_0\mapsto u$ is continuous from $H^s(\R^2)$ into $C([-T,T],H^s(\R^2))$.
\end{thm}

\subsection{A priori estimates for smooth solutions}

The following proposition is the main result of this subsection. It ensures us that the minimal existence time for smooth solutions can be chosen only depending on some rougher norms. For the sake of simplicity, from now on we shall denote by \begin{align}\label{norm_psi_3b}
|||\Psi|||_{s}:=\Vert \partial_t\Psi+\partial_x\Delta\Psi+\tfrac{1}{2}\partial_x(\Psi^2)\Vert_{L^\infty_t H^{s}_{xy}}.
\end{align}
\begin{prop}\label{apriori_smooth}
Let $\sigma\in(2,3+\frac12 \varepsilon_\star)$ and $s\in [1,\sigma]$ both be fixed. Then, for any $M>0$ there exists a positive existence time $T=T(M)$ such that, for any initial data $u_0\in H^\sigma(\R^2)$ satisfying $\Vert u_0\Vert_{H^s}\leq M$, the smooth solution $u(t)$ given by Theorem \ref{regular_lwp_thm} is defined on $[-T,T]$. Moreover, the solution $u(t)$ satisfies that \begin{align}\label{apriori_regularity_smooth}
u\in C([-T,T],H^\sigma(\R^2)) \quad \hbox{ and } \quad \Vert u\Vert_{L^\infty_TH^s_{xy}}\lesssim \Vert u_0\Vert_{H^s}.
\end{align}
\end{prop}
In order to prove the previous proposition we shall need the following technical Lemma.
\begin{lem}\label{cont_incr_tech_lem}
Let $s\in\R_+$ and $T>0$ both fixed. Consider $u\in C([-T,T],H^\infty(\R^2))$. Define \begin{align*}
\Lambda_{T'}^s(u):=\max\big\{\Vert u\Vert_{B^s(T')},\Vert \partial_x(u^2)\Vert_{N^s(T')},\Vert \partial_x(u\Psi)\Vert_{N^s(T')}\big\},
\end{align*}
for $0\leq T'\leq T$. Then, the map $T'\mapsto \Lambda_{T'}^s(u)$ defines a nondecreasing continuous function on $[0,T)$. Morover, we have that \begin{align*}
\lim_{T'\to 0}\Lambda_{T'}^s(u)\lesssim \Vert u(0)\Vert_{H^s}.
\end{align*}
\end{lem}

First we assume Lemma \ref{cont_incr_tech_lem} holds and we prove Proposition \ref{apriori_smooth}, then we prove Lemma \ref{cont_incr_tech_lem} which is much more technical.

\begin{proof}[Proof of Proposition \ref{apriori_smooth}]
In fact, let us consider $\varepsilon_\star=\min\{\varepsilon_1,\varepsilon_2\}>0$, where $\varepsilon_1$ and $\varepsilon_2$ are given by the regularity hypothesis \eqref{hyp_psi_general} on $\Psi$, namely \begin{align}\label{norms_tobe_small}
\Psi\in L^\infty(\R,W^{4+\varepsilon_1,\infty}(\R^2)) \quad \hbox{ and } \quad \big(\partial_t\Psi+\partial_x\Delta\Psi+\tfrac12\partial_x(\Psi^2)\big)\in L^\infty(\R,H^{3+\varepsilon_2}(\R^2)).
\end{align}
Now, consider $u_0\in H^\sigma(\R^2)$, with $\sigma\in (2,3+\tfrac12\varepsilon_\star)$. We denote by $u\in C([-T,T],H^\sigma(\R^2))$ the solution of \eqref{zk_psi} emanating from $u_0\in H^\sigma(\R^2)$, given by Theorem \ref{regular_lwp_thm}.   On the other hand, consider $s\in [1, \sigma]$ fixed. We claim that, by a scaling argument on $u$ and $\Psi$, without loss of generality we can always assume that the initial data $u_0$ has a small $H^s$-norm. Indeed, if $u(t,x,y)$ solves equation \eqref{zk_psi} on the time interval $[0,T]$, then, we define $u_\lambda$ and $\Psi_\lambda$  as follows\[
u_\lambda(t,x,y):=\lambda u\Big(\lambda^{3/2}t,\lambda^{1/2}x,\lambda^{1/2}y\Big) \quad \hbox{ and } \quad \Psi_\lambda(t,x,y):=\lambda\Psi\Big(\lambda^{3/2}t,\lambda^{1/2}x,\lambda^{1/2}y\Big).
\]
It is not difficult to see that $u_\lambda$ solves equation \eqref{zk_psi} in the background of $\Psi_\lambda$, i.e., equation \eqref{zk_psi} with $\Psi_\lambda$ taking the role of $\Psi$, with initial data $u_\lambda(0,x,y):=\lambda u_0(x,y)$, on the time interval $[-\lambda^{-3/2}T,\lambda^{-3/2}T]$. On the other hand, by direct computations we check that \begin{align}
\Vert u_\lambda(0,x,y)\Vert_{H^s}&\lesssim \lambda^{1/2}(1+\lambda^s)\Vert u_0\Vert_{H^s},\nonumber
\\ \Vert \Psi_\lambda\Vert_{L^\infty_tW^{r,\infty}_{xy}} & \lesssim  \lambda(1+\lambda^r)\Vert \Psi\Vert_{L^\infty_tW^{r,\infty}_{xy}},\label{delta_zkpsi_2}
\end{align}
and \begin{align}\label{delta_zkpsi} \Vert \partial_t\Psi_\lambda+\partial_x\Delta\Psi_\lambda+\tfrac12 \partial_x(\Psi_\lambda^2)\Vert_{L^\infty_tH^r_{xy}}\lesssim \lambda^{2}(1+\lambda^r) \Vert \partial_t\Psi+\partial_x\Delta\Psi+\tfrac12 \partial_x(\Psi^2)\Vert_{L^\infty_tH^r_{xy}},
\end{align}
for any $s,r>0$. Thus, for $\delta>0$ as small as desired, it suffices to choose $\lambda$ of the form \begin{align}\label{choice_delta}
\lambda\sim \delta^{2}\min\big\{1,\Vert u_0\Vert_{H^s}^{-2}\big\}, \quad \hbox{ and hence } \quad \Vert u_\lambda(0,x,y)\Vert_{H^s}\lesssim \delta.
\end{align}
Notice that this procedure makes the norms in \eqref{delta_zkpsi_2}-\eqref{delta_zkpsi} small as well. In fact, without loss of generality we can assume both norms associated with  \eqref{norms_tobe_small} are smaller than $\delta$, simply modifying the choice of $\lambda$ as follows \[
\lambda\sim \delta^2\min\big\{1,\Vert u_0\Vert_{H^s}^{-2},\Vert \Psi\Vert_{L^\infty_tW^{4+\varepsilon_1,\infty}_{xy}}^{-1},|||\Psi|||_{3+\varepsilon_2}^{-1/2}\big\}.
\]
Notice that, this choice of $\lambda$ is independent $\sigma$ and $T$ since $\Psi$ is given and fixed. Therefore, denoting by $\mathcal{B}^s(\delta)$ the $H^s(\R^2)$-ball with radius $\delta$ centered at the origin, we conclude that it is enough to prove that, if $u_0\in\mathcal{B}^s(\delta)$, then Proposition \ref{apriori_smooth} holds with $T=1$. Notice that this would in fact imply that Proposition \ref{apriori_smooth}
holds for arbitrarily large initial data in $H^s(\R^2)$, with an existence time $T\sim \Vert u_0\Vert_{H^s}^{-3}$.

\medskip

In view of the above analysis, now we consider $u_0\in H^\sigma(\R^2)\cap \mathcal{B}^s(\delta)$, and denote by $u\in C([-T,T],H^\sigma(\R^2))$, the solution to \eqref{zk_psi}  given by Theorem \ref{regular_lwp_thm}, with $T\in (0,1]$, with background function\footnote{Notice that we keep denoting the background function $\Psi_\lambda$, while the solution simply by $u$.} $\Psi_\lambda$. Before going further notice that, from the linear estimate \eqref{linear_prop_estimate} we get\begin{align*}
\Vert u\Vert_{F^{s}(T)}&\lesssim \Vert u\Vert_{B^{s}(T)}+\Vert \partial_x(u^2)\Vert_{N^{s}(T)}+\Vert \partial_x(u\Psi_\lambda)\Vert_{N^{s}(T)}
\\ & \qquad +\Vert \partial_t\Psi_\lambda+\partial_x\Delta \Psi_\lambda+\tfrac{1}{2}\partial_x(\Psi_\lambda^2)\Vert_{N^{s}(T)}.
\end{align*}
On the other hand, from the definition of the $N^{s}(T)$ norm in \eqref{NsT_norm}, and the hypotheses on $\Psi$, it is not difficult to see that \begin{align}\label{psi_ns_linf_hs_T}
\Vert \partial_t\Psi_\lambda+\partial_x\Delta \Psi_\lambda+\tfrac{1}{2}\partial_x(\Psi_\lambda^2)\Vert_{N^{s}(T)}\lesssim T^{1/2}\Vert \partial_t\Psi_\lambda+\partial_x\Delta \Psi_\lambda+\tfrac{1}{2}\partial_x(\Psi_\lambda^2)\Vert_{L^\infty_tH^{s}_{xy}}.
\end{align}
At this point it is worth to notice that, thanks to estimate \eqref{delta_zkpsi} and the choice of $\lambda$ in \eqref{choice_delta}, along with Cauchy-Schwarz inequality,  the second term appearing in the right-hand side of the energy estimate \eqref{energy_estimate_ineq} can be bounded as follows \[
\Vert u\Vert_{F^{s}(T)}|||\Psi_\lambda|||_{s}\lesssim \delta^2 \Vert u\Vert_{F^{s}(T)}^2+\delta^2|||\Psi|||_{s}^2.
\]
Being able to recover the $\delta$ factor in the above estimate shall be important in the sequel.  Therefore, gathering the above estimates, the energy estimate \eqref{energy_estimate_ineq},  the linear estimate \eqref{linear_prop_estimate} along with the $L^2$ bilinear estimates \eqref{short_t_estimate} and \eqref{short_t_psi_bil},  plugging them into the definition of $\Lambda^s_T(u)$, several applications of Cauchy-Schwarz inequality for $T\in (0,1)$ yield us to \begin{align}\label{lambda_f_estimate}
\Lambda_T^{\beta}(u)^2&\lesssim \Vert u_0\Vert_{H^{\beta}}^2  +  \delta^2|||\Psi|||_{\beta}^2+|||\Psi_\lambda|||_{\beta}^3+\Vert\Psi_\lambda\Vert_{L^\infty_TW^{(\beta+1)^+,\infty}_{xy}}|||\Psi_\lambda|||_{\beta}^2\nonumber 
\\ & \quad +\delta^4|||\Psi|||_{\beta}^4+\Vert\Psi_\lambda\Vert_{L^\infty_TW^{(\beta+1)^+,\infty}_{xy}}^2|||\Psi_\lambda|||_{\beta}^2+ \Big(\Lambda^{s}_T(u)+\Lambda^{s}_T(u)^2 \Big)\Lambda^{\beta}
_T(u)^2
\\ & \quad +\Big(\delta^2+\Vert \Psi_\lambda\Vert_{L^\infty_TW^{(\beta+1)^+,\infty}_{xy}}+\Vert \Psi_\lambda\Vert_{L^\infty_TW^{(\beta+1)^+,\infty}_{xy}}^2+|||\Psi_\lambda|||_{\beta}+|||\Psi_\lambda|||_{\beta}^2\Big)\Lambda^{\beta}_T(u)^2   , \nonumber
\end{align}
for any $\beta\in [s,\sigma]$. Therefore, using inequality \eqref{lambda_f_estimate}
 with $\beta=s$ we conclude that $\Lambda^{s}_T(u)\lesssim \delta$ provided that $\Vert u_0\Vert_{H^s}\lesssim \delta$. Now we define $\Gamma^{s}_T(u):=\max\big\{\Vert u\Vert_{B^{s}_T},\Vert u\Vert_{F^{s}_T}\big\}$. Thus,  inequality \eqref{lambda_f_estimate} with $\beta=s$, the linear estimate \eqref{linear_prop_estimate} along with inequality \eqref{psi_ns_linf_hs_T} and estimate \eqref{basic_embedding} in particular ensures, by a continuity argument, the existence of $\delta_s>0$ and $C_s>0$ such that $\Gamma^{s}_T(u)\leq C_s\delta$, provided that $\Vert u_0\Vert_{H^s}\leq \delta\leq \delta_s$. Hence, gathering inequality \eqref{basic_embedding}, the linear estimate \eqref{linear_prop_estimate} 
along with estimate \eqref{lambda_f_estimate} yield us to 
\begin{align}\label{uniform_bound}
\Vert u\Vert_{L^\infty_{T}H^\beta_{xy}}\lesssim  \Gamma^{\beta}_{T}(u)\lesssim \Vert u_0\Vert_{H^\beta},
\end{align}
for all $\beta\in[s,\sigma]$, provided that $\Vert u_0\Vert_{H^s}\leq \delta\leq \delta_s$, with $s\in [1,\sigma]$. Therefore, taking advantage of  estimate \eqref{uniform_bound}, we can reapply Theorem \ref{regular_lwp_thm} a finite number of times so that we extend the solution $u$ to the whole time interval $[-1,1]$. This concludes the proof of the Proposition.
\end{proof}

\begin{proof}[Proof of Lemma \ref{cont_incr_tech_lem}]
First of all notice that, from the definition of $B^s(T')$ it immediately follows that $T'\mapsto \Vert u\Vert_{B^s(T')}$ defines a nondecreasing continuous map on $[0,T]$ provided that $u\in C([-T,T],H^\infty(\R^2))$. Moreover, it is not difficult to see that \[
\lim_{T'\to0}\Vert u\Vert_{B^s(T')}\lesssim \Vert u(0)\Vert_{H^s}.
\]
Thus, in order to conclude the proof, it only remains to deal with $\Vert \partial_x(u^2)\Vert_{N^s(T')}$. In fact, in this case we shall prove a more general property. Let $f\in C([-T,T],H^\infty(\R^2))$, we claim that the following holds \[
T'\in[0,T)\mapsto \Vert f\Vert_{N^s(T')} \ \hbox{is nondecreasing and continuous,} \quad \hbox{ and }\quad \lim_{T'\to 0}\Vert f\Vert_{N^s(T)}=0.
\]
Indeed, first notice that, from the definition of $N^s$ in \eqref{Ns_norm} we see that \[
\Vert g\Vert_{N^s}\lesssim \Vert g\Vert_{L^2_tH^s_{xy}},
\]
for all $g\in L^2_tH^s_{xy}$. Then, taking $g=\mathds{1}_{[-T',T']}(t)f(t,x,y)$ we conclude that \begin{align}\label{Ns_bounded_by_LHs}
\Vert f\Vert_{N^s(T')}\lesssim \Vert g\Vert_{N^s}\lesssim \Vert g\Vert_{L^2_tH^s_{xy}}\lesssim T'^{1/2}\Vert f\Vert_{L^\infty_TH^s_{xy}}\xrightarrow{T'\to0}0.
\end{align}
Hence, the second part of the claim follows. Now, on the one-hand, the nondecreasing property of the claim follows directly from the definition of $N^s(T')$. On the other hand, in order to prove the continuity of the map at some point, let us say, $T_0\in(0,T)$ fixed, we define the scaling operator \[
D_r[f](t,x,y):=f\big(t/r,x/r^3,y/r^3\big).
\]
Therefore, by using inequality \eqref{Ns_bounded_by_LHs}, along with the triangle inequality  and the fact that $f\in C([-T,T],H^\infty(\R^2))$, we infer that \begin{align*}
\Big\vert \Vert f\Vert_{N^s(T')}-\Vert D_{T'/T_0}[f]\Vert_{N^s(T')}\Big\vert&\lesssim \Vert f-D_{T'/T_0}[f]\Vert_{N^s(T')}\\ & \lesssim T'^{1/2}\Vert f-D_{T'/T_0}[f]\Vert_{L^\infty_TH^s_{xy}}\xrightarrow{T'\to T_0}0.
\end{align*}
Thus, in order to conclude the proof of the claim, it only remains to prove tat \[
\lim_{r\to 1}\Vert D_r[f]\Vert_{N^s(rT_0)}=\Vert f\Vert_{N^s(T_0)}.
\]
In order to prove the latter property, we shall show the following two inequalities \begin{align}\label{liminflimsup_tecn}
\Vert f\Vert_{N^s(T_0)}\leq \liminf_{r\to1}\Vert D_r[f]\Vert_{N^s(rT_0)} \quad \hbox{ and } \quad \limsup_{r\to1}\Vert D_r[f]\Vert_{N^s(rT_0)}\leq \Vert f\Vert_{N^s(T_0)}.
\end{align}
In fact, let us begin with the first of them. Let $\varepsilon>0$ be an arbitrarily small number. For $r>0$ close to $1$, we choose an extension $f_r$ of $D_r[f]$ outside the interval $[-rT_0,rT_0]$ satisfying \begin{align}\label{scaling_hyp_estimate}
f_r\big\vert_{[-rT_0,rT_0]}\equiv D_r[f] \quad \hbox{ and } \quad \Vert f_r\Vert_{N^s}\leq \Vert D_r[f]\Vert_{N^s(rT_0)}+\varepsilon.
\end{align}
Now observe that $\Vert D_r[f]\Vert_{N^s(rT_0)}<M$ for a positive constant $M\in\R$ independent of $r$, since $f\in C([-T,T],H^\infty(\R^2))$. Also, notice that $D_{1/r}[f_r]$ is an extension of $f$ outside the interval $[-T_0,T_0]$, and hence \begin{align}\label{scalingoperator_estimate}
\Vert f\Vert_{N^s(T_0)}\leq \Vert D_{1/r}[f_r]\Vert_{N^s}.
\end{align}
Finally, in order to prove the latter property, we shall prove the existence of a continuous function $\phi(r)$, defined on a neighborhood of $r=1$, satisfying that $\phi(1)=1$, and  such that \begin{align}\label{claim_nh_xk}
\Vert D_{1/r}[f_r]\Vert_{N^s}\leq \phi(r)\Vert f_r\Vert_{N^s}.
\end{align}
Notice that the first estimate in \eqref{liminflimsup_tecn}  would follow from the above claim, along with \eqref{scaling_hyp_estimate} and \eqref{scalingoperator_estimate}. Now, in order to prove the latter inequality, let us fix some $H\in\mathbb{D}$ dyadic. Then, by definition of the $N_H$ norm we have that \begin{align*}
\Vert P_HD_{1/r}[f_r]\Vert_{N_H}=\sup_{t_H\in\R}\big\Vert \big(\tau-\omega(\xi,\mu)+iH^\beta\big)^{-1} \mathcal{F}\big(\eta_0(H^\beta(\cdot-t_H))P_HD_{1/r}[f_r]\big)\big\Vert_{X_H}.
\end{align*}
Then, by explicit computations we obtain that \[
\eta_0\big(H^\beta(\cdot-t_H)\big)D_{1/r}[f_r]=D_{1/r}\left[\eta_{0,r}(H^\beta(\cdot-rt_H))f_r\right],
\]
where $\eta_{0,r}(\cdot):=\eta_0(\cdot/r)$. Therefore, we have the following identity \begin{align*}
&\mathcal{F}\left[\eta_0(H^\beta(\cdot-t_H))P_HD_{1/r}[f_r]\right](\tau,\xi,\mu)
\\ & \qquad \qquad \qquad =r^{-5/3}\eta_H(\xi,\mu)\mathcal{F}\left[\eta_{0,r}(H^\beta(\cdot-rt_H))f_r\right](\tfrac\tau r,\tfrac\xi {r^{1/3}},\tfrac\mu{r^{1/3}}).
\end{align*}
Hence, by plugging the latter identity into the definition of the $X_H$ norm in \eqref{xh_norm}, we infer that the right-hand side of \eqref{claim_nh_xk} is equal to \[
r^{-5/6}\sup_{t_H\in\R}\sum_{L\in\mathbb{D}}L^{1/2}\left\Vert \dfrac{\eta_L(r(\tau-\omega(\xi,\mu)))}{(r(\tau-\omega(\xi,\mu))+iH^\beta)}\eta_H(r^{1/3}\xi,r^{1/3}\mu)\mathcal{F}\big[\eta_{0,r}(H^\beta(\cdot-t_H))f_r\big] \right\Vert_{L^2(\R^3)}.
\]
Now, observe that for any $a\in\R$ and any $r\in[\tfrac12,2]$  we have that \[
\left\vert \dfrac{a^2+H^{2\beta}}{r^2a^2+H^{2\beta}}-1\right\vert=\dfrac{a^2 \vert 1-r^2\vert}{r^2a^2+H^{2\beta}}\leq 4\vert 1-r^2\vert,
\]
which in turn trivially implies that \[
\dfrac{1}{r^2a^2+H^{2\beta}}\leq \dfrac{1+4\vert 1-r^2\vert}{a^2+H^{2\beta}}.
\]
Thus, we infer from the latter inequality that, for each $H\in\mathbb{D}$ fixed, it holds that \begin{align*}
&\left\Vert \dfrac{\eta_L(r(\tau-\omega(\xi,\mu)))}{r(\tau-\omega(\xi,\mu)+iH^\beta)}\eta_H(r^{1/3}\xi,r^{1/3}\mu)\mathcal{F}\big[\eta_{0,r}(H^\beta(\cdot-t_H))f_r\big] \right\Vert_{L^2(\R^3)}
\\ & \  \lesssim \big(1+4\vert 1-r^2\vert\big)^{1/2}\left\Vert \dfrac{\eta_L(r(\tau-\omega(\xi,\mu)))}{(\tau-\omega(\xi,\mu)+iH^\beta)}\eta_H(r^{1/3}\xi,r^{1/3}\mu)\mathcal{F}\big[\eta_{0,r}(H^\beta(\cdot-t_H))f_r\big] \right\Vert_{L^2(\R^3)}.
\end{align*}
On the other hand, by using the Mean Value Theorem in the $r$ variable we infer that \begin{align*}
\big\vert \eta_L\big(r(\tau-\omega(\xi,\mu))\big)-\eta_L\big(\tau-\omega(\xi,\mu)\big)\big\vert &\lesssim \vert r-1\vert \sup_{s\in[r,1]\cup[1,r]}\big\vert\eta_H'\big(s(\tau-\omega(\xi,\mu))\big)\big\vert
\\ & \lesssim \vert r-1\vert \sum_{ H'\in[\frac12H,2H]}\eta_{H'}\big(\tau-\omega(\xi,\mu)\big),
\end{align*}
provided that $r\in[\frac34,\frac54]$, which in turn trivially implies that \[
 \eta_L\big(r(\tau-\omega(\xi,\mu))\big)  \lesssim \eta_L\big(\tau-\omega(\xi,\mu)\big)+\vert r-1\vert \sum_{ H'\in[\frac12H,2H]}\eta_{H'}\big(\tau-\omega(\xi,\mu)\big).
\]
Therefore, gathering the above estimates we conclude that \begin{align}\label{phir_reduced_techlem}
&\Vert P_HD_{1/r}[f_r]\Vert_{N_H}
\\ & \qquad \leq \phi(r)\sup_{r_H\in\R}\big\Vert\big(\tau-\omega(\xi,\mu)+iH^\beta\big)^{-1}\eta_H(r^{1/3}\xi,r^{1/3}\mu)\mathcal{F}\big[\eta_{0,r}(H^\beta(\cdot-t_H))f_r\big]\big\Vert_{X_H},\nonumber
\end{align} 
where $\phi$ is a continuous function defined in a neighborhood of $1$, satisfying that $\phi(1)=1$. Finally, in order to treat the $\eta_{0,r}$ appearing on the right-hand side of the latter inequality, we proceed similarly fashion as before. Indeed, by using the Fundamental Theorem of Calculus we get that \[
\eta_{0,r}\big(H^\beta(t-t_H)\big)-\eta_0\big(H^\beta(t-t_H)\big)=\int_1^{1/r}\theta_s\big(H^\beta(t-t_H)\big)ds,
\]
where $\theta_s(t):=t\eta_0'(st)$. On the other hand, it is not difficult to see that, on the support of the latter integral above, the following holds \[
\eta_0\big(H^\beta(t-(t_H+2H^{-\beta}))\big)+\eta_0\big(H^\beta(t-(t_H-2H^{-\beta}))\big)\equiv1.
\]
Hence, plugging the latter identities into \eqref{phir_reduced_techlem} and then using Minkowski inequality along with Lemma \ref{lem_tau_w_ihb_xH} yields us to \begin{align*}
&\Vert P_HD_{1/r}[f_r]\Vert_{N_H}
\\ & \qquad \leq \widetilde{\phi}(r)\sup_{t_H\in\R}\big\Vert \big(\tau-\omega(\xi,\mu)+iH^\beta\big)^{-1}\eta_H(r^{1/3}\xi,r^{1/3}\mu)\mathcal{F}\big[\eta_0(H^\beta(t-t_H))f_r\big]\big\Vert_{X_H}.
\end{align*}
where $\widetilde\phi$ is a continuous function defined in a neighborhood of $1$, satisfying that $\widetilde\phi(1)=1$. Now notice that, proceeding similarly as before, using the Mean Value Theorem, we get that \[
\eta_H(r^{1/3}\xi,r^{1/3}\mu)\lesssim \eta_H(\xi,\mu)+\big\vert r^{1/3}-1\big\vert\sum_{H'\in[\frac12 H,2H]}\eta_{H'}(\xi),
\]
for $r\in[3/4,5/4]$. Thus, gathering the last two inequalities we obtain that \begin{align}
&\Vert P_HD_{1/r}f_r\Vert_{N_H}\label{r131_final_estimate}
\\ & \qquad \leq \big(1+\vert r^{1/3}-1\vert\big)\widetilde{\phi}(r)\Vert P_Hf_r\Vert_{N_H}\nonumber
\\ & \quad \qquad + \vert r^{1/3}-1\vert\widetilde{\phi}(r)\sup_{t_H\in\R}\big\Vert \big(\tau-\omega(\xi,\mu)+iH^\beta\big)^{-1}\eta_{2H}(\xi,\mu)\mathcal{F}\big[\eta_0(H^\beta(t-t_H))f_r\big]\big\Vert_{X_{2H}}\nonumber
\\ & \quad \qquad + \vert r^{1/3}-1\vert\widetilde{\phi}(r)\sup_{t_H\in\R}\big\Vert \big(\tau-\omega(\xi,\mu)+iH^\beta\big)^{-1}\eta_{H/2}(\xi,\mu)\mathcal{F}\big[\eta_0(H^\beta(t-t_H))f_r\big]\big\Vert_{X_{H/2}}.\nonumber
\end{align}
The first term in the right-hand side of the latter inequality is compatible with our claim. In order to deal with the second term, it is enough to notice that \[
\big\vert \tau-\omega(\xi,\mu)+iH^\beta\big\vert^{-1}\leq 4\big\vert \tau-\omega(\xi,\mu)+i(2H)^\beta\big\vert^{-1},
\]
and that, via a trivial change of variables, we have that \[
(2H)^\beta\int_{-t_H-8(2H)^{-\beta}}^{-t_H+8(2H)^{-\beta}}\eta_0\big((2H)^\beta(t+s)\big)ds=\int_\R \eta_0(s)ds>0,
\]
provided that $t\in\supp\eta_0\big(H^\beta(\cdot-t_H)\big)$. 
Therefore, from the above observations along with Minkowski inequality and an application of Lemma \ref{lem_tau_w_ihb_xH} we conclude that \begin{align}\label{r131_second_term}
\sup_{t_H\in\R}\big\Vert \big(\tau-\omega(\xi,\mu)+iH^\beta\big)^{-1}\eta_{2H}(\xi,\mu)\mathcal{F}\big[\eta_0(H^\beta(t-t_H))f_r\big]\big\Vert_{X_{2H}}\lesssim \Vert P_{2H}f_r\Vert_{N_{2H}}.
\end{align}
Similarly, to deal with the third term in \eqref{r131_final_estimate}, it is enough to notice that \[
\big\vert \tau-\omega(\xi,\mu)+iH^\beta\big\vert^{-1}\leq \big\vert \tau-\omega(\xi,\mu)+i(H/2)^\beta\big\vert^{-1},
\]
and that \[
\eta_0\big((H/2)^\beta(\cdot-t_H)\big)\equiv 1 \quad \hbox{ on the support of } \quad \eta_0\big(H^\beta(\cdot-t_H)\big).
\]
The last two observation along with an application of Lemma \ref{lem_tau_w_ihb_xH} imply that \begin{align}\label{r131_third_term}
\sup_{t_H\in\R}\big\Vert \big(\tau-\omega(\xi,\mu)+iH^\beta\big)^{-1}\eta_{H/2}(\xi,\mu)\mathcal{F}\big[\eta_0(H^\beta(t-t_H))f_r\big]\big\Vert_{H/2}\lesssim \Vert P_{H/2}f_r\Vert_{N_{H/2}}.
\end{align}
Then, we conclude the proof of the claim \eqref{claim_nh_xk} by plugging estimates \eqref{r131_second_term}, \eqref{r131_third_term} into \eqref{r131_final_estimate} and then summing over $H\in\mathbb{D}$. The proof for the second estimate in \eqref{claim_nh_xk} follows very similar lines, and hence we omit it. The proof is complete.
\end{proof}

\subsection{$L^2$-Lipschitz bound for the difference of two solutions and uniqueness}

Let us consider $u_1$ and $u_2$ two solutions of equation \eqref{zk_psi} defined on a time interval $[-T,T]$ for some $T\in (0,1]$, with initial data $u_1(0,x,y):=u_{1,0}(x,y)$ and $u_2(0,x,y):=u_{2,0}(x,y)$ respectively. Furthermore, we also assume that \begin{align}\label{hip_diff_lipschitz}
u_{1,0},u_{2,0}\in \mathcal{B}^s(\delta) \qquad \hbox{ and } \qquad \Gamma^{s}_T(u_i)\lesssim \delta,\ \, i=1,2.
\end{align}
Additionally, we make an smallness assumption on $\Psi$, namely, we assume that \begin{align}\label{hip_diff_psi_lipschitz}
\Vert \Psi\Vert_{L^\infty_tW^{4^+}_{xy}}+\Vert \partial_t\Psi+\partial_x\Delta\Psi+\tfrac12\partial_x (\Psi^2)\Vert_{L^\infty_tH^{3^+}_{xy}}\lesssim \delta.
\end{align}
From now on we denote by $w:=u_1-u_2$ and by $v:=u_1+u_2$. Then, observe that $w$ solves the equation \begin{align}\label{eq_w}
\partial_tw+\partial_x\Delta w+\tfrac12 \partial_x\Big[w(v+2\Psi)\Big]=0
\end{align}
Then, we conclude by gathering the linear estimate \eqref{linear_prop_estimate},   the bilinear estimates \eqref{short_t_2_estimate} and \eqref{short_t_psi_bil}, the energy estimate \eqref{diff_scero}, and the smallness assumptions  \eqref{hip_diff_lipschitz}-\eqref{hip_diff_psi_lipschitz}, that there exists $\delta'>0$ sufficiently small such that \begin{align}\label{smallness_condition}
\Gamma^0_T(w)\lesssim \Vert w_0\Vert_{L^2},
\end{align}
provided that $u_{1}$ and $u_2$ satisfy \eqref{hip_diff_lipschitz}, with $\delta\in (0,\delta')$. With the above $L^2$-bound, we are now in position to prove our uniqueness result.
\begin{prop}
Let $s\in[1,2]$ be fixed. Consider two solutions $u_1$ and $u_2$ to equation \eqref{zk_psi} in the class $C([-T,T],H^s)\cap B^{s}(T)\cap F^{s}(T)$ for some $T>0$, emanating from initial datum satisfying that $u_1(0,\cdot,\cdot)=u_2(0,\cdot,\cdot)=:u_0(x,y)$. Then, $u_1=u_2$ on the time interval $[-T,T]$.
\end{prop}

\begin{proof}
In fact, for the sake of simplicity let us denote by $K:=\max\big\{\Gamma^s_T(u_1),\Gamma^s_T(u_2)\big\}$. In the same fashion as before, we consider the same dilations $u_{i,\lambda}$ of $u_i$ for $i=1,2$ and $\lambda>0$, which are also solutions to the equation \eqref{zk_psi} on the time interval $[-T',T']$, with $T'=\lambda^{-3/2}T$, with initial data and background function given by (respectively) \[
u_{i,\lambda}(0,x,y)=\lambda u\big(0,\lambda^{1/2}x,\lambda^{1/2}y) \quad \hbox{ and } \quad \Psi_\lambda(t,x,y)=\lambda\big(\lambda^{3/2}t,\lambda^{1/2}x,\lambda^{1/2}y),
\]
as in the proof on Proposition \ref{apriori_smooth}. The above dilations imply that \[
\Vert u_{i,\lambda}\Vert_{L^\infty_{T'}H^s_{xy}}+\Vert u_{i,\lambda}\Vert_{B^{s}(T')}\lesssim \lambda^{1/2}\big(1+\lambda^s\big)\big(\Vert u_{i}\Vert_{L^\infty_TH^s_{xy}}+\Vert u_{i}\Vert_{B^{s}(T)}\big)\lesssim K\lambda^{1/2}\big(1+\lambda^{s}\big),
\]
for $i=1,2$. Thus, we can always choose $\lambda=\lambda(K)$ sufficiently small such that \begin{align}\label{small_k_lambda}
\Vert u_{i,\lambda}\Vert_{L^\infty_{T'}H^s}\lesssim \delta, \quad  \hbox{ and } \quad \Vert u_{i,\lambda}\Vert_{B^{s}(T')}\lesssim \delta.
\end{align}
Moreover, by making $\lambda$ smaller if necessary, we can also force the following inequality to hold \begin{align}\label{small_k_psi_lambda}
\Vert \Psi_\lambda\Vert_{L^\infty_tW^{4^+}_{xy}}+\Vert \partial_t\Psi_\lambda+\partial_x\Delta\Psi_\lambda+\tfrac12\partial_x (\Psi_\lambda^2)\Vert_{L^\infty_tH^{3^+}_{xy}}\lesssim \delta.
\end{align}
Next, we prove that, for $T''<T'$ sufficiently small, we also have that \begin{align}\label{claim_fsT_small}
\Vert u_{i,\lambda}\Vert_{F^{s}(T'')}\lesssim \delta.
\end{align}
In fact, first of all notice that, since $\Vert u_{i, \lambda}\Vert_{F^{s}(T)}<\infty$, we infer the existence of an $H\in\mathbb{D}$ such that \begin{align}\label{small_highf}
\Vert P_{\geq H}u_{i,\lambda}\Vert_{F^{s}(T'')}\leq \Vert P_{\geq H}u_{i,\lambda}\Vert_{F^{s}(T)}\leq \delta , \qquad i=1,2.
\end{align}
On the other hand, recalling that $\Vert u\Vert_{N^{s}(T'')}\lesssim \Vert u\Vert_{L^2_{T''}H^s_{xy}}$, we deduce from the linear estimate \eqref{linear_prop_estimate}, along with H\"older inequality, the Sobolev embedding $H^{1/2}(\R^2)\hookrightarrow L^4(\R^2)$ and the smallness condition \eqref{small_k_lambda}, that 
\begin{align*}
& \Vert P_{\leq H}u_{i,\lambda}\Vert_{F^{s}(T'')}
\\ & \qquad \lesssim \Vert u_{i,\lambda}\Vert_{B^{s}(T'')}+\Vert P_{\leq H}\partial_x(u^2_{i,\lambda})\Vert_{N^{s}(T'')}+\Vert P_{\leq H}\partial_x(u_{i,\lambda}\Psi_\lambda)\Vert_{N^{s}(T'')}
\\ & \qquad  \qquad +\Vert P_{\leq H}(\partial_t\Psi_\lambda+\partial_x\Delta\Psi_\lambda+\tfrac12\partial_x (\Psi^2_\lambda))\Vert_{N^{s}(T'')}
\\ & \qquad \lesssim \Vert u_{i,\lambda}\Vert_{B^{s}(T'')}+T''^{1/2}H^{s/2+1/2}\big(\Vert P_{\leq H}(u^2_{i,\lambda})\Vert_{L^\infty_{T''}L^2_{xy}}+\Vert P_{\leq H}(u_{i,\lambda}\Psi_\lambda)\Vert_{L^\infty_{T''}L^2_{xy}}\big)
\\ & \qquad \qquad +T''^{1/2}H^{s/2}\Vert P_{\leq H}(\partial_t\Psi_\lambda+\partial_x\Delta\Psi_\lambda+\tfrac12\partial_x (\Psi^2_\lambda))\Vert_{L^\infty_{T''}L^2_{xy}}
\\ & \qquad \lesssim \delta+T''^{1/2}H^{s/2+1/2}\Vert u_{i,\lambda}\Vert_{L^\infty_{T''}L^4_{xy}}^2+T''^{1/2}H^{s/2+1/2}\Vert u_{i,\lambda}\Vert_{L^\infty_{T''}L^2_{xy}}\Vert \Psi_\lambda\Vert_{L^\infty_{txy}}
\\ & \qquad \qquad  +T''^{1/2}H^{s/2}\Vert P_{\leq H}(\partial_t\Psi_\lambda+\partial_x\Delta\Psi_\lambda+\tfrac12\partial_x (\Psi^2_\lambda))\Vert_{L^\infty_{t}L^2_{xy}}
\\ & \qquad \lesssim \delta+T''^{1/2}H^{s/2+1/2}\Vert u_{i,\lambda}\Vert_{L^\infty_{T''}H^{s}_{xy}}^2+T''^{1/2}H^{s/2+1/2}\Vert u_{i,\lambda}\Vert_{L^\infty_{T''}H^s_{xy}}\Vert \Psi_\lambda\Vert_{L^\infty_{txy}}
\\ & \qquad \qquad  +T''^{1/2}H^{s/2}\Vert P_{\leq H}(\partial_t\Psi_\lambda+\partial_x\Delta\Psi_\lambda+\tfrac12\partial_x (\Psi^2_\lambda))\Vert_{L^\infty_{t}L^2_{xy}}
\\ & \qquad \lesssim \delta+T''^{1/2}H^{s/2+1/2}\big(1+\Vert u_{i,\lambda}\Vert_{L^\infty_{T''}H^s_{xy}}\big)^2\lesssim \delta,
\end{align*} 
for $T''$ sufficiently  small. Therefore, by choosing $T''=T''(K)$ sufficiently small, we deduce from the latter estimate, along with \eqref{small_k_lambda}, \eqref{small_k_psi_lambda} and \eqref{small_highf} that both $u_{1,\lambda}$ and $u_{2,\lambda}$ satisfy the smallness condition \eqref{hip_diff_lipschitz} on $[-T'',T'']$, that is, \[
\Gamma^s_{T''}(u_{i,\lambda})\lesssim \delta \qquad \hbox{for } \quad i=1,2.
\]
Then, the latter inequality along with  \eqref{smallness_condition} implies that $u_{1,\lambda}\equiv u_{2,\lambda}$ on $[-T'',T'']$, and hence, by applying this argument a finite number of times, we infer that the equality in fact holds in the whole interval $[-T',T']$, and so in $[-T,T]$ by dilation. The proof is complete.
\end{proof}

\subsection{Existence}

Let $s\in[1,2]$ be fixed and consider $u_0\in H^s(\R^2)$. By using a scaling argument as before, without loss of generality we can assume that $u_0\in\mathcal{B}^s(\delta)$ and \begin{align*}
\Vert \Psi\Vert_{L^\infty_tW^{4^+}_{xy}}+\Vert \partial_t\Psi+\partial_x\Delta\Psi+\tfrac12\partial_x (\Psi^2)\Vert_{L^\infty_tH^{3^+}_{xy}}\lesssim \delta.
\end{align*}
We seek to use the Bona-Smith argument (c.f. \cite{BS}) to obtain the existence of a solution $u(t)$ emanating from the initial data $u_0$. We shall follow the proofs in \cite{KePi,RiVe}. Before getting into the details, let us set a function $\rho\in\mathcal{S}(\R^2)$ satisfying that \[
\int  \rho(x,y) dxdy=1 \quad \hbox{ and } \quad \int x^iy^j\rho(x,y)dxdy=0,
\]
for $0\leq i\leq [s]+1$ and $0\leq j\leq [s]+1$, with $i+j\geq 1$. Moreover, from now on we denote by $\rho_\lambda$ the function given by $\rho_\lambda(x,y):=\lambda^{-2}\rho(x/\lambda,y/\lambda)$. We shall need the following technical lemma as well, whose proof can be found in \cite{BS,KaPo}.
\begin{lem}
Let $s\geq 0$, $\phi\in H^s(\R^2)$, and for any $\lambda>0$, set $\phi_\lambda:=\rho_\lambda*\phi$. Then, it holds that \begin{align}\label{bs_tech_ineq1}
\Vert \phi_\lambda\Vert_{H^{s+\gamma}}\lesssim \lambda^{-\gamma}\Vert \phi\Vert_{H^s}, \quad \hbox{ for all } \quad  \gamma\geq0,
\end{align}
and
\begin{align}\label{bs_tech_ineq2}
\Vert \phi-\phi_\lambda\Vert_{H^{s-\gamma}}\underset{\lambda\to0}{=}o(\lambda^\gamma), \quad \hbox{ for all } \quad  \gamma\in[0,s].
\end{align}
\end{lem}
With this lemma in hand, now we regularize the initial data. More specifically, we consider $u_{0,\lambda}:=\rho_\lambda*u_0$. Then, since $u_{0,\lambda}\in H^\infty(\R^2)$ for all $\lambda>0$, Theorem \ref{regular_lwp_thm} provide us the existence of a positive time $T_\lambda$ and a unique solution \[
u_\lambda\in C([-T_\lambda,T_\lambda],H^{3^+}(\R^2)),
\]
to equation \eqref{zk_psi} satisfying that $u_\lambda(0,\cdot,\cdot)=u_{0,\lambda}$. We observe that, from Young convolution inequality, we have that $\Vert u_{0,\lambda}\Vert_{H^s}\leq \Vert u_0\Vert_{H^s}\leq \delta$. It follows then, from the proof of Proposition \ref{apriori_smooth} and estimate \eqref{bs_tech_ineq1} that the sequence of solutions $\{u_\lambda\}$ can be extended to the whole time interval $[-1,1]$. Moreover, we have that \begin{align}\label{existence_1_ineq}
\Gamma_1^{s}(u_\lambda)\lesssim \Vert u_{0}\Vert_{H^s}\lesssim \delta \quad \hbox{ and } \quad \Gamma^{s+1}_1(u_\lambda)\lesssim \Vert u_{0,\lambda}\Vert_{H^{s+1}}\lesssim \lambda^{-1}\Vert u_0\Vert_{H^s},
\end{align}
for all $\lambda>0$. Then, we infer from \eqref{smallness_condition} and  \eqref{bs_tech_ineq2} that, for any $\lambda'\in(0,\lambda)$ it holds that
\begin{align}\label{existence_2_ineq}
\Gamma_1^0(u_\lambda-u_{\lambda'})\lesssim \Vert u_{0,\lambda}-u_{0,\lambda'}\Vert_{L^2}\underset{\lambda\to0}{=}o(\lambda^s).
\end{align}
On the other hand, from the linear estimate \eqref{linear_prop_estimate}, the $L^2$-bilinear estimates  
\eqref{short_t_estimate}
 and \eqref{short_t_psi_bil}, along with the energy estimate \eqref{diff_s_util} and \eqref{existence_1_ineq}, choosing $\delta$ small enough, we infer that  \begin{align*}
\Gamma^{s}_1(u_\lambda-u_{\lambda'})^2&\lesssim \Vert u_{0,\lambda}-u_{0,\lambda'}\Vert_{H^s}^2+\Big(\delta+\Gamma^{s+1}_1(u_\lambda)+\Gamma_1^{s+1}(u_{\lambda'})\Big)\Gamma_1^{0}(u_\lambda-u_{\lambda'})^2,
\end{align*}
since $s\geq 1$. Thus, combining \eqref{bs_tech_ineq2}, \eqref{existence_1_ineq} and \eqref{existence_2_ineq}  with the latter estimate, we get that 
\begin{align}\label{ex_diff_lamb_zero}
\Vert u_\lambda-u_{\lambda'}\Vert_{L^\infty_1H^s_{xy}}\lesssim \Gamma_1^{s}(u_\lambda-u_{\lambda'})\to 0 \quad \hbox{ as } \quad \lambda\to0.
\end{align}

Therefore, we conclude that the sequence $\{u_\lambda\}$ converges in the $\Gamma^s_1$-norm to a solution $u(t)$ of \eqref{zk_psi} in the class $C([-T,T],H^s(\R^2))\cap F^s(T)\cap B^s(T)$. The proof is complete.

\medskip

\subsection{Continuity of the flow map}

Let $s\in [1,2]$ and $u_0\in H^s(\R^2)$ both be fixed. In the same fashion as in the previous subsections, by a scaling argument, without loss of generality we can assume that $u_0\in \mathcal{B}^s(\delta)$ and \[
\Vert \Psi\Vert_{L^\infty_tW^{4^+}_{xy}}+\Vert \partial_t\Psi+\partial_x\Delta\Psi+\tfrac12\partial_x (\Psi^2)\Vert_{L^\infty_tH^{3^+}_{xy}}\lesssim \delta,
\]
for $\delta>0$ sufficiently small. Then, the solution $u(t)$ emanating from $u_0$ is defined on the whole time interval $[-1,1]$ and belongs to the class $u\in C([-1,1],H^s(\R^2))$. Now, let $\varepsilon\in (0,1)$ fixed. Thus, it suffices to prove that, for any initial data $v_0\in \mathcal{B}^s(\delta)$ such that $\Vert u_0-v_0\Vert_{H^s}\leq \tilde\epsilon$, with $\tilde\epsilon=\tilde\epsilon(\varepsilon) >0$ small enough to be fixed, the solution $v\in C([-1,1],H^s(\R^2))$ emanating from $v_0$ satisfies that \begin{align}\label{continuitymap_proof}
\Vert u-v\Vert_{L^\infty_1H^s_{xy}}\leq \varepsilon.
\end{align}
Next, as in the previous subsection, we regularize the initial datum $u_0$ and $v_0$ by defining \[
u_{0,\lambda}=\rho_\lambda*u_0 \quad \hbox{ and } \quad v_{0,\lambda}=\rho_\lambda*v_0,
\]
for $\lambda>0$, and consider the associated smooth solutions $u_\lambda, v_\lambda\in C([-1,1],H^{\sigma}(\R^2))$ with $\sigma\in (2,3+\tfrac12\varepsilon_\star)$. Then, from the triangle inequality it follows that \begin{align}\label{triangle}
\Vert u-v\Vert_{L^\infty_1H^s_{xy}}\leq \Vert u-u_\lambda\Vert_{L^\infty_1H^s_{xy}}+\Vert u_\lambda-v_\lambda\Vert_{L^\infty_1H^s_{xy}}+\Vert v-v_\lambda\Vert_{L^\infty_1H^s_{xy}}.
\end{align}
Now notice that, from \eqref{ex_diff_lamb_zero} it follows that we can choose $\lambda_*>0$ sufficiently small such that \[
\Vert u-u_{\lambda_*}\Vert_{L^\infty_1H^s_{xy}}+\Vert v-v_{\lambda_*}\Vert_{L^\infty_1H^s_{xy}}\leq\tfrac23\varepsilon.
\] 
On the other hand, we infer from \eqref{bs_tech_ineq1} that \[
\Vert u_{0,\lambda_*}-v_{0,\lambda_*}\Vert_{H^\sigma}\lesssim \lambda_*^{-(\sigma-s)}\Vert u_0-v_0\Vert_{H^s}\lesssim \lambda_*^{-(\sigma-s)}\tilde\epsilon.
\]
Therefore, by using the continuity result of the flow map for regular initial data (see Theorem \ref{regular_lwp_thm}), we obtain the existence of an $\tilde\epsilon>0$ small enough such that \[
\Vert u_{\lambda_*}-v_{\lambda_*}\Vert_{L^\infty_1H^s_{xy}}\leq \tfrac13\varepsilon.
\]
Thus, gathering the above estimates we conclude the proof.

\medskip 

\section{Global well-posedness}\label{sec_gwp}

In this section we seek to prove the global well-posedness Theorem \ref{MT_gwp}. We emphasize once again that, due to the presence of $\Psi(t,x,y)$, equation \eqref{zk_psi} has no evident well-defined conservation laws. However, a slight modification of the energy functional along with Gr\"onwall inequality shall be enough to conclude the proof of our GWP Theorem \ref{MT_gwp}. Our first lemma states that the $L^2$-norm of the solution can grows at most exponentially fast in time.
\begin{lem}\label{lemma_growth_l2}
Let $u(t)\in C([0,T],H^1(\R^2))$ be a solution to equation \eqref{zk_psi} emanating from an initial data $u_0\in H^1(\R^2)$. Then, for all $t\in[0,T]$ we have
\begin{align}\label{L2_gronwall_general}
\Vert u(t)\Vert_{L^2_{xy}}^2\leq C_{u_0,\Psi}\exp(C_{\Psi}t),
\end{align}
where $C_{\Psi}>0$ is a positive constants that only depends on $\Psi$, while $C_{u_0,\Psi}>0$ depends on $\Psi$ and $\Vert u_0\Vert_{L^2}$.
\end{lem}

\begin{proof}
First of all, by using the continuity of the flow with respect to the initial data, given by Theorem \ref{MT1}, we can assume $u(t)$ is sufficiently smooth so that all the following computations hold. Now,  by taking the time derivative of the mass functional, using equation \eqref{zk_psi}, after suitable integration by parts we obtain 
\begin{align*}
\dfrac{1}{2}\dfrac{d}{dt}\int_\R u^2(t,x)dx&=-\dfrac12\int u\partial_x\big(u\Psi\big) -\int u\big(\partial_t\Psi+\partial_x^3\Psi+\tfrac12\partial_x( \Psi^2)\big)
\\ &=: \mathrm{I}+\mathrm{II}.
\end{align*}
Notice that, thanks to our hypotheses on $\Psi$, we can immediately bound $\mathrm{II}$ by using Young inequality for products, from where we get
\[
\vert \mathrm{II}\vert\leq\Vert u(t)\Vert_{L^2_{xy}}^2+\Vert \partial_t\Psi+\partial_x^3\Psi+\partial_xf(\Psi)\Vert_{L^\infty_tL^2_{xy}}^2.
\]
Similarly, after suitable integration by parts and a direct use of H\"older inequality we obtain \[
\vert \mathrm{I}\vert\leq2\Vert \Psi_x\Vert_{L^\infty_{txy}}\Vert u(t)\Vert_{L^2_{xy}}^2.
\]
Therefore, Gronwall inequality provides \eqref{L2_gronwall_general}. The proof is complete.
\end{proof}

Now, in order to control the $H^1$-norm, we consider the following modified energy functional \begin{align}\label{modi_energy}
\mathcal{E}\big(u(t)\big):=\int_{\R^2} \vert \nabla u\vert ^2-\dfrac13\int_{\R^2} u^ 2\big(u+3\Psi\big).
\end{align}
It is worth to notice that the previous functional is well defined for all times $t\in[0,T]$, however, it is clearly not conserved by the ZK-flow. The following lemma give us the desired control on the growth of the $H^1$-norm of the solution $u(t)$, and hence, it finishes the proof of Theorem \ref{MT_gwp}.

\begin{lem}
Let $u(t)\in C([0,T],H^1(\R^2))$ be a solution to equation \eqref{zk_psi} emanating from an initial data $u_0\in H^1(\R^2)$. Then, for all $t\in[0,T]$ we have
\begin{align*}
\Vert u(t)\Vert_{H^1_{xy}}\lesssim  C_{u_0,\Psi}^*\exp(C_{\Psi}^*t).
\end{align*}
where $C_{\Psi}^*>0$ is a positive constants that only depends on $\Psi$, while $C_{u_0,\Psi}^*>0$ depends on $\Psi$ and $\Vert u_0\Vert_{H^1}$.
\end{lem}

\begin{proof}
By using the continuity of the flow with respect to the initial data given by Theorem \ref{MT1}, we can assume $u(t)$ is sufficiently smooth so that all the following computations hold. Now, let us begin by explicitly computing the time derivative of the modified energy functional $\mathcal{E}(u(t))$. Indeed, by using equation \eqref{zk_psi}, after suitable integration by parts we obtain  that
\begin{align}\label{derivative_energy_gwp}
\dfrac{d}{dt}\mathcal{E}&=-2\int (u_{xx}+u_{yy})u_t-\int u_t\big(u^2+2u\Psi\big) -\int u^2\Psi_t\nonumber
\\ & =\int (\Delta u)\partial_x\big( u^2+2u\Psi\big)+2\int u\Delta\big(\Psi_t+\partial_x\Delta\Psi+\tfrac12\partial_x(\Psi^2)\big)\nonumber
\\ & \quad +\int (\partial_x\Delta u)\big(u^2+2u\Psi\big)
+\dfrac12\int \big( u^2+2u\Psi\big)\partial_x\big( u^2+2u\Psi\big)\nonumber
\\ & \quad +\int \big( u^2+2u\Psi\big)\big(\Psi_t+\partial_x^3\Psi+\tfrac12 \partial_x (\Psi^2)\big) -\int u^2\Psi_t\nonumber
\\ & = 2\int u\Delta\big(\Psi_t+\partial_x^3\Psi+\tfrac12 \partial_x(\Psi^2)\big) -\int u^2\Psi_t\nonumber
\\ & \quad +\int \big(u^2+2u\Psi\big)\big(\Psi_t+\partial_x^3\Psi+\tfrac12 \partial_x(\Psi^2)\big)\nonumber
\\ & \lesssim \big(1+\Vert \Psi_t\Vert_{L^\infty_{txy}}+\Vert \Psi\Vert_{L^\infty_{txy}}^2+\Vert \partial_t\Psi+\partial_x^3\Psi+ \tfrac12 \partial_x(\Psi^2)\Vert_{L^\infty_{txy}}\big)\Vert u(t)\Vert_{L^2_{xy}}^2
\\ & \quad +\Vert \Psi_t+\partial_x^3\Psi+\tfrac12 \partial_x(\Psi^2)\Vert_{L^\infty_tH^2_{xy}}^2. \nonumber
\end{align}
On the other hand, by using Gagliardo-Nirenberg interpolation inequality along with Young inequality for products, we have that \begin{align*}
\left\vert\int_{\R^2} u^3(t,x,y)dxdy\right\vert&\leq C\Vert u(t)\Vert_{H^1_{xy}}\Vert u(t)\Vert_{L^2_{xy}}^{2} \lesssim \varepsilon\Vert u(t)\Vert_{H^1_{xy}}^2+\dfrac{1}{\varepsilon}\Vert u(t)\Vert_{L^2_{xy}}^{4},
\end{align*}
where $\varepsilon>0$ is a small parameter. Thus, by plugging the latter inequality into the second integral in the definition of the modified energy functional \eqref{modi_energy}, we get that \begin{align*}
\left\vert \int u^2\big(u+3\Psi\big)\right\vert &\lesssim \Vert u(t)\Vert_{L^3_{xy}}^3+\Vert \Psi\Vert_{L^\infty_{txy}}\Vert u(t)\Vert_{L^2_{xy}}^2
\\ &\lesssim \varepsilon\Vert u(t)\Vert_{H^1_{xy}}^2+\varepsilon^{-1}\Vert u(t)\Vert_{L^2_{xy}}^{4}+\Vert \Psi\Vert_{L^\infty_{txy}}\Vert u(t)\Vert_{L^2_{xy}}^2.
\end{align*}
Therefore, integrating \eqref{derivative_energy_gwp} in time on $[0,T]$, plugging the latter inequality into the resulting right-hand side, choosing $\varepsilon\in(0,1)$ sufficiently small, and then using the $L^2$-bound found in  Lemma \ref{lemma_growth_l2}, we infer that \begin{align*}
\int_\R \vert \nabla u\vert^2(t,x,y)dxdy&\lesssim \int \vert\nabla u_{0}\vert ^2dxdy -\int u_0^2\big(u_0+3\Psi\big)dxdy+C_{u_0,\Psi}\exp(C_\Psi t),
\end{align*}
where $C_{\Psi}$ only depend on the above norms of $\Psi$ and $C_{u_0,\Psi}$ only depends on norms of $\Psi$ and $\Vert u_0\Vert_{L^2}$. The proof is complete.
\end{proof}

\bigskip

\textbf{Acknowledgements:} The author is very grateful to Professor Luc Molinet for encouraging him to solve this problem and for several remarkably helpful comments and conversations.

\end{document}